\documentclass{article}
\usepackage{amsmath,amsfonts,amssymb,amsthm}
\usepackage[french, british]{babel}
\usepackage{color}
\usepackage{hyperref}
\usepackage{dsfont}
\usepackage{multirow}
\usepackage{enumerate}

\providecommand{\bbR}{\mathbb{R}}

\providecommand{\calF}{\mathcal{F}}
\providecommand{\calS}{\mathcal{S}}
\providecommand{\calB}{\mathcal{B}}
\providecommand{\calL}{\mathcal{L}}
\providecommand{\pS}{\partial \mathcal{S}}

\providecommand{\eps}{\varepsilon}

\def\longrightharpoonup{\relbar\joinrel\rightharpoonup}
\def\cv{\stackrel{w}{\longrightharpoonup}}
\def\cvwstar{\stackrel{w*}{\longrightharpoonup}}

\providecommand{\calC}{\mathcal{C}}
\providecommand{\calT}{\mathcal{T}}
\providecommand{\calJ}{\mathcal{J}}
\providecommand{\calK}{\mathcal{K}}

\renewcommand{\leq}{\leqslant}
\renewcommand{\geq}{\geqslant}

\renewcommand{\div}{\operatorname{div}}
\newcommand{\curl}{\operatorname{curl}}

\newcommand{\tr}{\operatorname{tr}}
\newcommand{\dist}{\operatorname{dist}}
\newcommand{\Id}{\operatorname{Id}}

\allowdisplaybreaks

\newtheorem{theorem}{Theorem}
\newtheorem{definition}{Definition}
\newtheorem{corollary}{Corollary}
\newtheorem{proposition}{Proposition}
\newtheorem{Lemma}{Lemma}
\newtheorem{remark}{Remark}

\begin{document}

\date{\today}
\title{On the existence of weak solutions for the 2D incompressible Euler equations with in-out flow and source and sink points.}
\author{Marco Bravin\footnote{BCAM - Basque Center for Applied Mathematics, Mazarredo 14, E48009 Bilbao, Basque Country - Spain.}}
\maketitle

\begin{abstract}
Well-posedness for the two dimensional Euler system with given initial vorticity is known since the works of Judovi\v{c}. In this paper we show existence of solutions in the case where we allowed the fluid to enter in and exit from the boundaries and from some points of the fluid domain. In particular we derive the equations of the model as the limit when we replace the points by some small holes. To do that we extend the DiPerna-Lions theory with non-tangent velocity field on the boundary to the case of time-dependent domain, we extend the existence result for the two dimensional Euler system with in-out flow to time-dependent domain and finally we derive the system that models a fluid which is allowed to enter in and exit from the boundary and some points. The solutions are characterized by the presence of source, sink and  vortex points.      
\end{abstract}

\section*{Introduction}

In this paper we study the existence for solutions of a system describing the flow of an incompressible inviscid fluid in a domain where the fluid is allowed to enter and exit from the boundary and also the presence of source and sink points is allowed. 

We start by presenting the model. Let $ \Omega  \subset \bbR^2 $ open, bounded and with smooth enough boundary and let $ \calS^{+}(t), \calS^{-}(t) \subset \Omega $ two closed, connected, disjoint subset of $ \Omega $ with smooth enough boundaries. Then if $ \calS^{+}(t) $ and $ \calS^{-}(t)$ have non-empty interiors the flow of a incompressible, inviscid fluid exiting from $ \partial \calS^+ $ and entering in $ \partial \calS^- $ is described by the following system.
\begin{align}
\partial_{t} v + v \cdot \nabla v + \nabla p = & \, 0 && \text{ for } x \in \Omega \setminus (\calS^{+}(t)\cup \calS^-(t)), \nonumber \\
\div v = & \, 0 && \text{ for } x \in \Omega \setminus (\calS^{+}(t)\cup \calS^-(t)), \label{in:sys} \\
v \cdot n = & \, g && \text{ for } x \in \partial \Omega \cup \partial \calS^{+}(t) \cup \partial \calS^-(t), \nonumber \\
\curl v = & \, \omega^{+} && \text{ for  } x \in \partial \calS^{+}(t), \nonumber
\end{align}  
where $ v = (v_1, v_2) $ and $ p $ are respectively the velocity field and the pressure. The vector $ n $ is the normal exiting from the fluid domain, the function $ g $ describes the normal component of the velocity field on the boundary and in particular $ g = 0 $ on $ \partial \Omega $, $ g < 0 $ on $\partial \calS^{+}(t) $ and $ g > 0 $ on $ \partial \calS^{-}(t) $. Finally $ \omega^{+} $ denotes the entering vorticity. This system was introduced by Judovi\v{c} in \cite{Jud}, in particular he proved existence and uniqueness of regular solutions. Later in \cite{Ale}, the author proved existence of solutions in the setting of bounded vorticity. In the case where the fluid domain is time independent, i.e. $ \Omega \setminus (\calS^{+}(t)\cup \calS^-(t)) =  \Omega \setminus (\calS^{+}(0)\cup \calS^-(0)) $, existence of solutions to the system \eqref{in:sys} was shown in \cite{IO3} for initial and entering vorticity in $ L^p $ for $ p\geq 1 $. Uniqueness was shown in \cite{NFS} in the case of $ L^{\infty} $ initial and entering vorticity.  

Judovi\v{c} in \cite{Jud} left open the question of well-posedness for the system satisfied by the flow when the domain has some holes that may shrink to points, in particular the condition of having non-empty interior is not any more fulfilled. In this article we derive the equations satisfied by the fluid when the source and the sink are allowed to be points as the limiting system satisfied by solutions of \eqref{in:sys} when the size of the source and sink becomes small at certain times. Moreover we prove existence of solutions in an appropriate sense. 

We write down the system when the source and the sink are allowed to be points for some times. To do that we define 
\begin{gather*}
 \calT^{+} = \left\{ t \in \bbR^{+} \text{ such that } \calS^{+}(t) = x_{+}(t) \text{ a point in } \bbR^{2} \right\}, \\ 
 \calT^{-} = \left\{ t \in \bbR^{+} \text{ such that } \calS^{-}(t) = x_{-}(t) \text{ a point in } \bbR^{2} \right\}, 
\end{gather*}
and their complementaries $ \calT^{+}_{NP} = \bbR^{+} \setminus \calT^{+} $ and $ \calT^{-}_{NP} = \bbR^{+} \setminus \calT^{-} $. The equations in the vorticity form read as
\begin{align}
& \partial_t \omega + \div(v \omega) = j \mathds{1}_{\calT^{+}} \delta_{x_{+}(t)}- \left( \frac{d}{dt} \int \omega  + \oint_{\partial S^{+}(t)} \omega^{+}(g - q)\mathds{1}_{\calT^{+}_{NP}} + j \mathds{1}_{\calT^{+}} \right) \mathds{1}_{\calT^{-}}\delta_{x_{-}(t)} \quad && \text{ for } x \in \calF(t), \nonumber \\
&  \begin{aligned}
v =  \calJ_{\calF(t)} \left[\mu \mathds{1}_{\calT^{+}}\delta_{x_{+}(t)} - \left(\oint_{\pS^{+}(t)} g \mathds{1}_{\calT^{+}_{NP}} + \mu \mathds{1}_{\calT^{+}} \right) \mathds{1}_{\calT^{-}}\delta_{x_{-}(t)} \right] \\  + \calK_{\calF(t)} \left[ \omega + \calC_{+} \mathds{1}_{\calT^{+}} \delta_{x_{+}(t)} + \calC_{-} \mathds{1}_{\calT^{-}} \delta_{x_{-}(t)} \right] 
\end{aligned} \quad &&\text{for } x \in \calF(t), \label{main:sys} \\
& \calC_{+}(t) = \calC_{+}^{in} - \int_{0}^{t} \left( \oint_{\partial S^{+}(t)} \omega^{+}(g - q)\mathds{1}_{\calT^{+}_{NP}} + j \mathds{1}_{\calT^{+}}  \right), &&  \nonumber \\
& \calC_{-}(t) = \calC_{-}^{in} + \int_{\calF(t)} \omega - \int_{\calF_0} \omega^{in} + \int_{0}^{t} \left( \oint_{\partial S^{+}(t)} \omega^{+}(g - q)\mathds{1}_{\calT^{+}_{NP}} + j \mathds{1}_{\calT^{+}}  \right), \nonumber &&
\end{align}
where $ \omega = \curl u $ denotes the vorticity, $ \calF(t) = \Omega \setminus \overline{\text{int}(\calS^{+}(t) \cup \calS^{-}(t))} $ with $ \text{int}(.) $ we denote the interior of a set, $ j $ and $ \mu $ are two real quantities defined on $ \calT^{+} $ and they are associated with the entering vorticity and entering flow, $ q $ is the normal component of the velocity of the boundary $ \partial \calF(t) $, $ g $ is the normal velocity of the fluid on $ \partial \calF(t) $, in particular $ g = 0 $ on $ \partial \Omega $, $ g - q < 0 $ on $ \pS^{+}(t) $ for $ t \in \calT^{+}_{NP} $ and $ g - q > 0 $ on $ \pS^{-}(t) $ for $ t \in \calT^{-}_{NP} $, $ \omega^{+} $ is the entering vorticity on $ \pS^{+}(t) $ for $ t \in \calT^{+}_{NP} $. Finally we define the operators $\calJ $ and $ \calK $ as follow. For $ \mathfrak{m} $ such that $ \int_{\calF(t)} \mathfrak{m} = \int_{\partial \calF(t) } g $, we define $ \calJ_{\calF(t)}(\mathfrak{m}) = u $ the solution of
\begin{align*}
\div u = & \, \mathfrak{m} && \text{ for } x \in \calF(t), \\
\curl u = & \, 0 && \text{ for } x \in \calF(t),\\
u\cdot n = & \, g && \text{ for } x \in \partial  \calF(t), \\
\left(\oint_{\pS^{i}(t)} u \cdot \tau \right) \mathds{1}_{\calT^{i}_{NP}} = & \, 0, &&  
\end{align*}    
and analogously for $ \mathfrak{n} $ such that $ \int_{\calF(t)} \mathfrak{n} = \calC_{+}(t) + \calC_{-}(t) +\calC_{\Omega} $, then we define $ \calK_{\calF(t)}[\mathfrak{n}] = u $ the solution of 
\begin{align*}
\div u = & \, 0 && \text{ for } x \in \calF(t), \\
\curl u = & \, \mathfrak{n} && \text{ for } x \in \calF(t),\\
u\cdot n = & \, 0 && \text{ for } x \in \partial  \calF(t), \\
\left(\oint_{\pS^{i}(t)} u \cdot \tau \right) \mathds{1}_{\calT^{i}_{NP}} = & \, \calC_{i}(t). &&  
\end{align*}

The idea of the paper is to consider some auxiliary systems characterized by the presence of a source $ \calS^{+}_{\eps}(t) $ and $ \calS^{-}_{\eps}(t) $ with non-empty interior that approximate $ \calS^{+}(t) $ and $ \calS^{-}(t) $ respectively. In particular $ \calS^{+}_{\eps}(t) $ and $ \calS^{-}_{\eps}(t) $ converge to $ \calS^{+}(t) $ and $ \calS^{-}(t) $ in an appropriate sense as $ \eps $ converges to zero. We will show that solutions of the approximate system with appropriate boundary data converges to a weak solution of \eqref{main:sys}. 

\section{Formal derivation of the system}

 The system \eqref{in:sys} in the vorticity form was derived by Judovich in \cite{Jud} in the case where the source and the sink $ \calS^{i}(t) $ have non-empty interior. Before presenting the system we recall the Kelvin's theorem for a closed curve that does not follow the flow. It holds 
\begin{equation}
\label{kelvin:theorem}
\frac{d}{dt} \oint_{\pS^{i}(t)} v(t,.) \cdot \tau = -\oint_{\pS^i(t)} (v \cdot n - q) \curl v, 
\end{equation}
which can be deduce for smooth enough solutions by computing the time derivative of the circulation and by using the equations \eqref{in:sys}. By applying $ \curl $ to the system \eqref{in:sys}, we obtain 
\begin{align}
& \partial_t \omega + \div(v \omega) = 0 && \text{ for } x \in \calF(t), \nonumber \\
& 
v =  \calK_{\calF(t)} \left[ \omega \right]  && \text{ for } x \in \calF(t), \label{judv:sys} \\
& \calC_{+}(t) = \calC_{+}^{in} - \int_{0}^{t} \oint_{\partial S^{+}(t)} \omega^{+}(g - q), &&  \nonumber \\
& \calC_{-}(t) = \calC_{-}^{in} + \int_{\calF(t)} \omega - \int_{\calF_0} \omega^{in} + \int_{0}^{t} \oint_{\partial S^{+}(t)} \omega^{+}(g - q). \nonumber
\end{align}
In the case where the source and the sink are time-independent point, i.e. $ x_{+}(t) = x_{+} $  and $ x_{-}(t) = x_{-} \neq x_{+} $, the system was derived in \cite{IO3} and it reads
\begin{align}
& \partial_t \omega + \div(v \omega) = j  \delta_{x_{+}} - \left( \frac{d}{dt} \int \omega + j  \right) \delta_{x_{-}} && \text{ for } x \in \calF(t), \nonumber \\
&  v =  \calJ_{\calF(t)} \left[\mu \delta_{x_{+}} - \mu \delta_{x_{-}} \right]  + \calK_{\calF(t)} \left[ \omega + \calC_{+} \delta_{x_{+}} + \calC_{-} \delta_{x_{-}} \right] 
 && \text{ for } x \in \calF(t), \label{BS:main:sys} \\
& \calC_{+}(t) = \calC_{+}^{in} - \int_{0}^{t} j , &&  \nonumber \\
& \calC_{-}(t) = \calC_{-}^{in} + \int_{\calF(t)} \omega - \int_{\calF_0} \omega^{in} + \int_{0}^{t} j, \nonumber
\end{align}
where $ \mu $ and $ j $ are quantities associated with respectively the entering flow and the entering vorticity.  
System \eqref{main:sys} is a mixture of the equations \eqref{judv:sys}, \eqref{BS:main:sys} and \eqref{kelvin:theorem} to describe the behaviours of the vorticity in dependence of the geometric property of the source and of the sink. More precisely close to an hole with non-empty interior the system is described by \eqref{judv:sys}, in the case it is a point by \eqref{main:sys} which is characterized by the presence of a point source/sink and a point vortex.

\section{Definition of solution and main result}

\label{section:setting}

Let start by presenting the geometry that we allow for the domain $ \Omega $ and the holes $ \calS^{i} $. The properties that we required are basically two. The first one is that the set $ \Omega, \calS^+ $ and $ \calS^- $ have regular enough boundaries in space and time variables to prove elliptic estimates uniformly in time and to have the velocity of the boundary regular enough. The second one is that $ \Omega, \calS^+ $ and $ \calS^-$ are compatible in the sense that the mutual distance is greater than a positive constant in any compact time interval $ [0,T] $.   

To give an idea we are considering  $ \Omega, \calS^+ $ and $ \calS^-$ such that $ \overline{\calS^+},  \overline{\calS^-} \subset \text{int}(\Omega) $,  $ \overline{\calS^+} \cap \overline{\calS^-} = \emptyset $ and the boundaries are smooth except the transition time where an hole of non-empty interior becomes a point or the other way around.

In the next subsection we formalize this idea by introducing the definition of regular compatible geometry. In a first reading this subsection can be skipped. 

\subsection{Assumptions on the fluid domain}

The holes $ \calS^i \subset \Omega $ can have non-empty interior or can be points. In the transition times where an hole passes from a subset of $ \bbR^2 $ with non-empty interior to a point it is not clear how to study the regularity of the boundary. To avoid this issue we restrict only to holes $ \calS^i $ that are given as follow. There exist a shape map $ S^i: \bbR^{+}\times \partial B_1(0) \longrightarrow  \bbR^{2}  $, a radius map $ r_i: \bbR^{+} \longrightarrow \bbR^{+} $ and a position map  $ h^i: \bbR^{+} \longrightarrow \Omega  $ such that the boundary of the i-th hole is 
\begin{equation}
\label{def:holes}
\partial \calS^{i}(t) = \left\{ r^i(t)S^i(t,x) + h^i(t) \quad \text{ such that } \quad x \in \partial B_{1}(0) \right\}.
\end{equation}
The set $ \calS^{i}(t) $ is the closure of the bounded component of $ \bbR^2 $ separated by $ \partial \calS^i(t,.) $ if $ r^i(t) > 0 $. It is the point $ h^i(t) $ if $ r^i(t) = 0$. With this notation $ \calT^i = \{ t \in \bbR^{+} $ such that $ r^i(t) = 0 \} $ and the normal velocity of the boundary $ q(t,x) = \partial_t(r^i(t)S^i(t,x) + h^i(t))\cdot n $ for $ (t,x) \in \mathcal{T}^i_{NP} \times \partial \calS^{i}(t) $.

\begin{definition}[Regular septuple]
Let $ \Omega \subset \bbR^2 $ a connected simply-connected bounded domain. For $ i \in {+,-} $, let $ S^i: \bbR^{+}\times \partial B_1(0) \longrightarrow  \bbR^{2}  $ the shape maps, $ r_i: \bbR^{+} \longrightarrow \bbR^{+} $ the radius map and $ h^i: \bbR^{+} \longrightarrow \Omega $ the position map. Then we say that the septuple $ (\Omega, S^+, S^-, r^+,r-,h^+,h^-) $ is regular if for some $ \alpha > 0 $
\begin{itemize}

	\item $ \Omega $ has $ C^{2,\alpha} $ boundary,
	
	\item $ S^i $ is a $ C^{1, \alpha}_{loc}(\bbR^{+};C^{1,\alpha}(\partial B_1(0);\bbR^2)) \cap C^{0, \alpha}_{loc}(\bbR^{+};C^{3,\alpha}(\partial B_1(0);\bbR^2))  $	embedding.
	
	\item $ r_i \in C^{1,\alpha}_{loc}(\bbR^{+};\bbR^+) $ and $ h^i_{loc} \in C^{1,\alpha}(\bbR^+;\bbR^2) $,
	
	\item the origin is contained in the bounded component of $ \bbR^{2} $ separated by the image  of $ S^i(t,.) $.
\end{itemize} 
\end{definition}
We define now the concept of compatibility.

\begin{definition}[Compatible geometry]
We say that a triple of collection of subsets of $ \bbR^2 $ indexed by $ t\in \bbR^+$ $ ( \Omega, \calS^+(.), \calS^{-}(.)) $  or  a septuble $ (\Omega, S^+, r^+,h^+, S^-, r^-,h^-) $ is a compatible geometry if for any $ t $
	\begin{itemize}
		\item $ \overline{\calS^i(t)} \subset \text{int}(\Omega) $  for $ i \in \{+,-\} $
		
		\item $ \overline{\calS^{+}(t)} \cap  \overline{\calS^{-}(t)} = \emptyset $,
		
	\end{itemize}
where we recall that given a septuple $ (\Omega, S^+, r^+,h^+, S^-, r^-,h^-) $, $ \calS^+ $ and $ \calS^-$ are the bounded domain with boundary defined in \eqref{def:holes}.	
		
\end{definition}

\begin{definition}[Regular compatible geometry]

Let $ (\Omega, \calS^+(.), \calS^-(.)) $ a triple of collection of subsets of $ \bbR^2 $ indexed by $ t\in \bbR^+$. Then we say that it is a regular compatible geometry if there exists a regular and compatible septuple such that $ \calS^i $ is defined via \eqref{def:holes}. 
		
\end{definition}

\begin{remark}
	In the following we always consider regular compatible geometry. In particular the mutual distances between the source, the sink and the boundary of $ \Omega $ is always lower bounded by a positive constant in any compact time interval $ [0,T] $. 
\end{remark}

\subsection{Boundary data and definition of weak solution}

To simplify the notation for a collection of set $ X(t) \subset \bbR^2$, we denote
\begin{gather*}
\bbR^{+} \odot X(t) = \bigcup_{t \in \bbR^{+}} t \times X(t) \quad \text{ and } \\ \mathcal{L}_{loc}^{p}(\bbR^{+}; W^{s,p}(X(t))) = \{ u : \bbR^{+} \odot X(t) \to \bbR \text{ s.t. }  \|u(t,.)\|_{W^{s,p}(X)(t)} \in L^{p}_{loc}   \}. 
\end{gather*}
Moreover for any measurable subset $ I \subset \bbR^{+} $ and any function $ f $ defined on $ I \odot \pS^{i}(t) $, we also denote by $ f $ the extension by zero on $ \bbR^{+} \odot \pS^{i}(t) $.

We briefly discuss the boundary data that we require in the definition and in the existence result of weak solution. Regarding $ \partial \mathcal{S}^+(t) $, we impose  the normal component of the velocity $ g $ for $ t \in \mathcal{T}^{+}_{NP} $. Moreover because $ g - q < 0 $, we are allowed to impose also the entering vorticity $ \omega^{+}$. For $ t \in \mathcal{T}^{+} $, $ \partial \mathcal{S}^+(t) = x_+(t) $ and we impose $ \mu $ and $ j $, the strengths respectively of the point source and of the point vortex. Regarding $ \partial \mathcal{S}^-(t) $, we impose  the normal component of the velocity $ g $ for $ t \in \mathcal{T}^{-}_{NP} $. In this case $ g - q > 0 $, which implies that the exiting vorticity is an unknown of the problem. For $ t \in \mathcal{T}^{-} $, no data are required, in fact both the strength of the point sink and point vortex can be recovered due to respectively the mass conservation and the conservation of the vorticity, where here by conservation we mean the sum of what is entering minus what is exiting plus what is inside is constant. For example for $ t \in \calT^{-}_{NP} $, the compatibility condition
\begin{equation}
\label{comp:con:mu}
\left(\oint_{\partial \calS^{+}(t)} (g -q)\right) \mathds{1}_{\calT^{+}_{NP}} + \mu \mathds{1}_{\calT^{+}} + \left(\oint_{\partial \calS^{-}(t)} (g -q)\right) = 0,
\end{equation}
have to be satisfied. Let us introduce the concept of compatible in-out velocity. 

\begin{definition}[Compatible in-out velocities]
	Let $ ( \Omega, \calS^+(.), \calS^{-}(.)) $ a regular compatible geometry. A couple $ (g, \mu) $ with 
	$$ g : \left( \bbR^+ \times \partial \Omega \right) \cup \left(\bigcup_{i \in  \{+,-\}} \calT^i_{NP} \odot \calS^i(t)\right) \longrightarrow \bbR \quad \text{ and } \quad \mu : \calT^+ \longrightarrow \bbR^+  $$ 
	is a compatible in-out velocities if 	
	\begin{itemize}
		\item it holds $ g- q <  0 $ on $ \calT^+_{NP} \odot \partial \calS^+ $, $ g- q >  0 $ on $ \calT^-_{NP} \odot \partial \calS^- $ and $ g = 0 $ on $ \bbR^+ \times \partial \Omega $.
		
		\item for any $ t \in \calT^{-}_{NP} $ equality \eqref{comp:con:mu} holds.
	
	\end{itemize}

\end{definition}

Let us remark that we do not define the strength $ \mu^-$  of the sink point because for $ t \in \calT^- $ it holds
\begin{equation*}
\mu^{-} = - \left(\oint_{\partial \calS^{+}(t)} (g -q)\right) \mathds{1}_{\calT^{+}_{NP}} - \mu \mathds{1}_{\calT^{+}},
\end{equation*}
due to the incompressibility of the velocity field. 

We present the definition of weak solution.

\begin{definition}
\label{main:def:def}

Let $ (\Omega, \calS^+, \calS^-) $ a regular compatible geometry. Let $ p \in (2,+ \infty) $, $ q \in [p/(p-1), 2) $ and $ r > 1 $. Let $ \omega^{in} \in L^{p}(\calF(0)) $ the initial vorticity, let $ \calC^{in}_i \in \bbR $ the initial circulations around $ \pS^i $. Let $ (g,\mu) $ a compatible in-out velocities such that $ g \in \mathcal{L}^{r}_{loc}(\bbR^{+}; W^{1-1/p,p}(\partial \calF (t)) $ and  $ \mu \in L^{r}_{loc}(\calT^{+}) $. Let $ \omega^{+} : \calT^{+}_{NP} \odot \pS^{i}(t) \longrightarrow \bbR $ such that $ \omega^{+} \in L^{p}_{loc}(\calT^{+}_{NP} \odot \pS^{i}(t), \lvert g-q \rvert \, dt ds) $ and let $ j : \calT^{+} \longrightarrow \bbR $ such that $ j \in L^{r}_{loc}(\bbR^{+})$, $ \text{supp}(j) \subset \text{supp}(\mu) $ and $ j/\mu \in L^{\infty}_{loc}(\bbR^{+}) $ where $  j/\mu $ is defined $ 0 $ for $ \mu = 0 $. Then we say that a triple $ ( \omega, \omega^{-}, v) $ is a weak solution of the system \eqref{main:sys} if $ \omega \in \mathcal{L}^{\infty}_{loc}(\bbR^{+}, L^{p}(\calF(t)))$, $ \omega^{-} \in \mathcal{L}^{p}_{loc}(\calT^{-}_{NP} \odot \pS^{-}(t), \lvert g -q \rvert dt ds)$, $ v \in \mathcal{L}^{r}_{loc}(\bbR^{+}; L^{q}(\calF(t)))$, $ \int_{\calF(t)} \omega \in W^{1,1}_{loc}(\bbR^{+}) $ and it holds
\begin{align*}
\int_{\calF(0)} \omega^i\varphi(0,.) +  & \int_{\bbR^{+}} \int_{\calF(t)} \omega \partial_t \varphi +  \int_{\bbR^{+}} \int_{\calF(t)} \omega v \cdot \nabla \varphi = \\ & \sum_{i \in \{+,-\}}\int_{\calT^{i}_{NP}} \oint_{\pS^{i}(t)} \omega^{i}(g-q)\varphi + \int_{\calT^+} j \varphi(t,x_{+}(t)) \\ - & \int_{\calT^{-}} \left( \frac{d}{dt} \int \omega +  \oint_{\partial S^{+}(t)} \omega^{+}(g - q)\mathds{1}_{\calT^{+}_{NP}} + j \mathds{1}_{\calT^{+}} \right) \varphi(t,x_{-}(t)),
\end{align*}    
for any $ \varphi \in C^{\infty}_{c}(\bbR^{+}\times \overline{\Omega}) $,
\begin{align*}
\int_{\calF(t)} v(t,.) \cdot \nabla \xi = & \, \sum_{i\in \{+,-\}} \left(\oint_{\pS^i(t)} g(t,.) \xi \right)\mathds{1}_{\mathcal{T}^i_{NP}}(t) + \mu(t)  \mathds{1}_{\calT^{+}} \xi(x_+(t))\\ & \,  - \left(\oint_{\pS^{i}(t)} g(t,.) \mathds{1}_{T_{NP}^{+}} +\mu(t) \mathds{1}_{T^{+}}(t) \right) \mathds{1}_{\calT^{-}}(t)\xi(x_{-}(t)),
\end{align*}
for any $ \xi \in C^{\infty}(\overline{\Omega})$,
\begin{align*}
& \, \int_{\calF(t)} v(t,.) \cdot \nabla^{\perp} \zeta =  - \int_{\calF(t)} \omega(t,.) \zeta  \\ & \, + \left[\calC_{+}^{in} - \int_{0}^{t} \left( \oint_{\partial S^{+}(t)} \omega^{+}(g - q)\mathds{1}_{\calT^{+}_{NP}} + j \mathds{1}_{\calT^{+}}  \right) \right] \zeta_{\pS^{+}(t)} \\
& \, +\left[ \calC_{-}^{in} + \int_{\calF(t)} \omega - \int_{\calF_0} \omega^{in} + \int_{0}^{t} \left( \oint_{\partial S^{+}(t)} \omega^{+}(g - q)\mathds{1}_{\calT^{+}_{NP}} + j \mathds{1}_{\calT^{+}}  \right) \right] \zeta_{\pS^{-}(t)},
\end{align*}
for $ \zeta \in C^{\infty}(\overline{\calF(t)}) $ such that $ \zeta = 0 $ on $ \partial \Omega $ and constant on any $ \pS^{i}(t) $. Such constants are denoted by $ \zeta_{\pS^i(t)} $, in particular for $ t \in \calT^{i}$, it holds $ \zeta_{\pS^i(t)}= \zeta(x_i(t)) $.  

\end{definition}

The main result of this paper is the proof of existence for weak solutions of \eqref{main:sys} in the sense of Definition \ref{main:def:def}, in particular we will show existence by considering an approximated sequence of problem where the holes have non empty interior and we pass to the limit in the weak formulation. 

\begin{theorem}
\label{main:theorem}

Let $ (\Omega, \calS^+, \calS^-) $ a regular compatible geometry. Let $ p \in (2,+\infty) $, $ q \in [p/(p-1), 2) $ and $ r > 1 $. Let $ \omega^{in} \in L^{p}(\calF(0)) $ the initial vorticity, let $ \calC^{in}_i \in \bbR $ the initial circulations around $ \pS^i $. Let $ (g,\mu) $ a compatible in-out velocities such that $ g \in \mathcal{L}^{r}_{loc}(\bbR^{+}; W^{1-1/p,p}(\partial \calF (t))) $ and $ \mu \in L^{r}_{loc}(\calT^{+}) $. Let $ \omega^{+} : \calT^{+}_{NP} \odot \pS^{+}(t) \longrightarrow \bbR $ such that $ \omega^{+} \in \mathcal{L}^{p}_{loc}(\calT^{+}_{NP} \odot \pS^{+}(t),\lvert g- q \rvert \, dtds) $ and let $ j : \calT^{+} \longrightarrow \bbR $ such that $ j \in L^{r}_{loc}(\bbR^{+})$, $ \text{supp}(j) \subset \text{supp}(\mu) $ and $ j/\mu \in L^{\infty}_{loc}(\bbR^{+}) $ where $  j/\mu $ is defined $ 0 $ for $ \mu = 0 $. Then there exists a weak solution $ ( \omega, \omega^{-}, v) $ of the system \eqref{main:sys} in the sense of Definition \ref{main:def:def}. 

\end{theorem}

As far as we know the literature of existence of solution for the Euler system \eqref{judv:sys} is quite incomplete due to the fact that we consider time dependent domains. For this reason we extend all the existence results proved in \cite{IO3} in this new setting.

\section{Renormalized solution for transport equation in time dependent domains}
\label{sec:ren:sol}
In this section we show that the DiPerna-Lions theory \cite{DPL} on renormalized solutions for the transport equation extends to the case of a smooth enough time dependent domains with incoming and outgoing flow. In particular we will extend Section 3 of \cite{Boyer} for time depending domains. 

This theory will be used to show existence of solutions for the $ 2D $-Euler system with prescribed entering vorticity, time dependent domain and $p$-integrable initial vorticity for $ p > 1$. Finally for $ p = 1 $, i.e. only integrable vorticity, we show existence of Delort type solutions and in particular we answer to a question posed in \cite{FLF}. These results will also extend the work \cite{HH}.    

The domain $ \mathbb{F} = \bbR^{+} \odot \calF(t) $, where we assume $ \partial \calF(0) $ is a $ C^{3,\alpha}$ and there exists $ (\Id,b) : \bbR^{+} \times \partial \calF(0) \to \bbR^{+} \odot \partial \calF (t) $ a $ C^{1,\alpha}_{loc}(\bbR^{+}, C^{1,\alpha}(\partial \calF(t)) ) \cap C^{0,\alpha}_{loc}(\bbR^{+}, C^{3,\alpha}(\partial \calF(t)) )$ diffeomorphism with $ b(0,y) = y$ in $ \partial \calF(0)$. For the velocity field $ v $ we assume, for $ p \in [1,\infty) $ that
\begin{equation}
\label{hyp:ren}
v \in \calL^{1}_{loc}(\bbR^{+}, W^{1,p}(\calF(t))) \quad \text{ and } \quad \div v = 0. 
\end{equation}
Moreover, we have to assume one of the following conditions.
\begin{enumerate}

\item (Geometric condition). There exists $ \psi \in C^{\infty}(\mathbb{F}) $ such that $ \psi = 1 $ on $ \partial \calF^+ $ and $ \psi =  0 $ on $ \partial \calF^- $.

\item (Extra regularity). There exists $ \alpha > 1 $ such that $$ v\cdot n \in L^{\alpha}_{loc}(\bbR^{+} \odot \partial \calF(t)), $$
\end{enumerate}
where we denote by $ \partial \calF^{+} = \{ (t,x) \in \bbR^{+} \odot \partial \calF(t) $ such that $  v\cdot n - q < 0 \}$ and analogously $ \partial \calF^{-} = \{ (t,x) \in \bbR^{+} \odot \partial \calF(t) $ such that $  v\cdot n - q > 0 \} $. Note that we use this convention to recall that in $ \partial \calF^{+} $ the flow is entering and in $ \partial \calF^{-} $ is exiting.

Given $ \rho^{in} $ and $ \rho^{+} $ measurable functions respectively on $ \partial \calF(0) $  and $ \partial \calF^{+} $, we look for solutions $ (\rho, \rho^{-}) $ to the transport equation
\begin{align}
\partial_ t \rho + v \cdot \nabla \rho = & \, f \quad && \text{ for } x \in \calF(t), \nonumber \\
\rho =  & \, \rho^{+} \quad && \text{ for } x \in \partial \calF^{+}(t), \label{DP:theo:tdep}\\
\rho(0,.) = & \, \rho^{in} \quad && \text{ for } \calF(0).   \nonumber
\end{align}  

Let $ \mu > 0 $ a measurable function on $ \Omega $ such that $ \int_{\Omega} \mu < + \infty $. We denote by
\begin{align*}
L^{0}(\Omega, \mu dL) = \big\{ f: \Omega \to \bar{\bbR} \quad \text{ L-measurable}   \big\}.
\end{align*}

We are now stating a well-posedness result for \eqref{DP:theo:tdep} with initial and boundary data in $L^{0} $ spaces.

\begin{theorem}
\label{exi:transport:ppppppppp}
Let $ v $ a vector field in $ \bbR^{+} \odot \calF(t) $ satisfying the hypothesis \eqref{hyp:ren} and or the geometric condition or the extra regularity. Let $ \rho^{in} \in L^{0}(\calF(0)) $ an initial datum, $ \rho^{+}\in L^{0}_{loc}(\bbR^{+}\odot \partial \calF^{+}(t), \lvert v\cdot n - q \rvert \, dt ds) $ an entering information and let $ f \in \calL^1_{loc}(\bbR^+;L^1(\calF(t)))$ a source term, then there exists a unique renormalized solution $ (\rho, \rho^{-}) \in \calL^{\infty}_{loc}(\bbR^{+}; L^{0}(\calF(t))) \times L^{0}_{loc}(\bbR^{+} \odot \partial \calF^{-}(t), \lvert v\cdot n -q \rvert \, dt ds) $. More precisely for any $ \varphi \in C^{\infty}_{c}(\bbR^{+}\times \bbR^2) $, it holds 
\begin{align}
\label{ren:equ}
\int_{\calF(0)} \beta(\rho^{in}) \varphi(0,.) + & \,  \int_{\bbR^{+}} \int_{\calF(t)} \beta(\rho) \partial_t \varphi + \int_{\bbR^{+}} \int_{\calF(t)} \beta(\rho) v \cdot \nabla \varphi = \nonumber  \\ & \, \sum_{i \in \{+,- \} }\int_{\bbR^{+}} \int_{\partial \calF^{i}} \beta(\rho^{i})(v\cdot n - q) \varphi - \int_{\bbR+}\int_{\calF(t)} f \beta'(\rho)\varphi,
\end{align}
for any $ \beta \in C^1_b (\bar{\bbR} ) = \{ \beta \in C^{0}(\bar{\bbR};\bbR) $ such that $ \|\beta \|_{C^1} < + \infty \}$.
\end{theorem}

Let us write a straight-forward corollary.

\begin{corollary}
Let $ v $ a vector field in $ \bbR^{+} \odot \calF(t) $ satisfying the hypothesis \eqref{hyp:ren} and or the geometric condition or the extra regularity. Let $ \rho^{in} \in L^{p}(\calF(0)) $ an initial datum, $ \rho^{+}\in L^{p}_{loc}(\bbR^{+}\odot \partial \calF^{+}(t), \lvert v\cdot n - q \rvert \, dtds) $ an entering information and let $ f \in \calL^1_{loc}(\bbR^+;L^p(\calF(t)))$ a source term, then there exists a unique renormalized solution $ (\rho, \rho^{-}) \in C^{0}_{loc}(\bbR^{+}; L^{p}(\calF(t))) \times L^{p}_{loc}(\bbR^{+} \odot \partial \calF^{-}(t), \lvert v\cdot n -q \rvert  \, dtds) $. Moreover for any $ T \in \bbR^+ $, it holds 
\begin{align*}
\int_{\calF(T)} \lvert \rho(T,.)\rvert^p + \int_{0}^{T} \int_{\partial \calF^{-}} \lvert \rho^{-}\rvert^p(v\cdot n - q) = & \, \int_{\calF(0)} \lvert\rho^{in}\rvert^p  \\ - \int_{0}^{T} \int_{\partial \calF^{+}} \lvert \rho^{+}\rvert^p(v\cdot n - q) & \, + \int_{0}^{T} \int_{\calF(t)} f p \rho \lvert \rho\rvert^{p-2}. 
\end{align*}

\end{corollary} 

Finally the following duality formula holds true. 

\begin{theorem}
\label{duality:formula}
Let $ p, q \in [1,+\infty] $ such that $ 1/p + 1/q = 1 $. Let $ v \in \calL^1(\bbR^+;W^{1,1}(\calF(t)))$ with $ \div(v) = 0 $. Let $ u \in \calL^{\infty}(0,T;L^p(\calF(t))) $ a renormalized solution to \eqref{ren:equ} and let $ r \in \calL^{\infty}(0,T;L^q(\calF(t))) $ a renormalized solution to the transport \eqref{DP:theo:tdep} with source term $ f \in  \calL^{1}(0,T;L^q(\calF(t)) )$. Then we have the  following duality formula
\begin{align*}
\int_{\calF(t)} u(t,.) r(t,.) \, dx  & \, - \int_{\calF(0)} u(0,.) r(0,.) \, dx \\ & \, +  \sum_{i = +,-} \int_{0}^T \int_{\calF^i_t} u^i r^i (v\cdot n - q ) \, ds dt = \int_0^T \int_{\calF(t)} f u . 
\end{align*}

\end{theorem}	

We prove the above results in Section \ref{sec:ren}. Then we will apply them to show existence of solution for the Euler-type system \eqref{judv:sys}.

\subsection{Application to existence for Euler-system with in-out flow}

Distributional solution for the transport of vorticity with data in $ L^p $ are defined only for $ p \geq 4/3 $, in fact for $ p < 4/3 $ the term $ \omega v $ is not in general $ L^{1} $. For this reason we look for renormalized solution for the transport equation for the vorticity. We will show existence of renormalized solutions through a vanishing viscosity limit. More precisely we will show that solutions can be obtained as limits of a Navier-Stokes system with appropriate boundary conditions. The Euler-system reads as 

\begin{align}
& \partial_t \omega + \div(v \omega) = 0 && \text{ for } x \in \calF(t), \nonumber \\
& 
v =  \calK_{\calF(t)} \left[ \omega \right]  && \text{ for } x \in \calF(t), \nonumber \\
& \calC_{+}(t) = \calC_{+}^{in} - \int_{0}^{t} \oint_{\partial S^{+}(t)} \omega^{+}(g - q), &&  \label{judv:sys:dddddddddd}  \\
& \calC_{-}(t) = \calC_{-}^{in} + \int_{\calF(t)} \omega - \int_{\calF_0} \omega^{in} + \int_{0}^{t} \oint_{\partial S^{+}(t)} \omega^{+}(g - q). \nonumber
\end{align}
 
We are now able to define weak solutions for the above system.

\begin{definition}
\label{def:di:di:di:di}

Let $ p > 1 $, let $ \omega^{in} \in L^{p}(\calF(0)) $, let $ \calC^{in}_{i}\in \bbR $, let $ g \in \calL^{r}_{loc}(\bbR^{+}, W^{1-1/p, p}(\partial \calF(t))) $ such that $ \int_{\partial \calF (t)} g = 0$ and let $ \omega^{+} \in L^{p}_{loc}(\bbR^{+} \odot \partial \calS^{+}(t), \lvert g - q \rvert  \, dtds)$. Then a triple $ ( \omega, \omega^{-},v) \in \calL^{\infty}_{loc}(\bbR^{+}, L^{p}(\calF(t)))\times L^{p}_{loc}(\bbR^{+} \odot \partial \calS^{-}(t), \lvert g - q \rvert \, dtds) \times \calL^{r}_{loc}(\bbR^{+}; W^{1,p}(\calF(t))) $ is a weak solution of the system \eqref{judv:sys:dddddddddd} if for any $ \varphi \in C^{\infty}_{c}(\bbR^{+}\times \bbR^{2}) $ and $  \beta \in C^1_b (\bar{\bbR}) $,  it holds
\begin{align}
\label{rem:tras:form}
\int_{\calF(0)} \beta(\omega^{in}) \varphi(0,.)+ \int_{\bbR^{+}} \int_{\calF(t)} \beta(\omega) \partial_t \varphi & \, + \int_{\bbR^{+}} \int_{\calF(t)} \beta(\omega) v \cdot \nabla \varphi =  \\ & \,  \sum_{i \in \{+,- \} }\int_{\bbR^{+}} \int_{\partial \calS^{i}} \beta(\omega^{i})(v\cdot n - q) \varphi, \nonumber
\end{align}
the velocity field $ v $ satisfies in a strong sense the Div-Curl system 
\begin{equation}
\label{div_curl:blabla}
\div v = 0, \text{ } \curl v = \omega \text{ in } \calF(t), \quad v\cdot n = g \text{ on } \partial \calF(t) \quad \text{ and } \quad \oint_{\pS^{i}(t)} v(t,.) \cdot \tau = \calC_{i}(t), 
\end{equation}
where $ \calC_+$ and $\calC_{-} $ satisfy the last two equations of \eqref{judv:sys:dddddddddd}.

\end{definition} 
 
The existence result reads as follows.

\begin{theorem}
\label{exi:Lp:Nempty}
Let $ p > 1 $, let $ \omega^{in} \in L^{p}(\calF(0)) $, let $ g \in \calL^{r}_{loc}(\bbR^{+}, W^{1-1/p, p}(\partial \calF(t))) $ such that $ \int_{\partial \calF (t)} g = 0$ and let $ \omega^{+} \in L^{p}_{loc}(\bbR^{+} \odot \partial \calS^{+}(t), \lvert g - q \rvert  \, dt ds)$. Then there exists a solution $ ( \omega, \omega^{-}, v) \in \calL^{\infty}_{loc}(\bbR^{+}, L^{p}(\calF(t)))\times L^{p}_{loc}(\bbR^{+} \odot \partial \calS^{-}(t), \lvert g - q\rvert  \, dt ds) \times \calL^{r}_{loc}(\bbR^{+}; W^{1,p}(\calF(t))) $ to the system \eqref{judv:sys:dddddddddd} in the sense of Definition \ref{def:di:di:di:di}.    

\end{theorem}

Let conclude the section with the limit case $ p = 1 $. For this case we prove existence of Delort type solutions introduced in \cite{Del}, see also \cite{Scho}. The idea is to rewrite the non linear term $ \omega v $ as an integral of a kernel multiplied by the vorticity. To do that we decompose the velocity
\begin{equation}
\label{u:dec:ss}
v = v_g + \calK^{0}_{\calF(t)}[\omega]+\sum_{i} \left[ \int_{\calF} \Psi_i \omega + \calC_i(t)  \right] X_i
\end{equation}
where $ v_g = \nabla \varphi $ and $ \calK^{0}_{\calF}[\omega] = \nabla^{\perp}\phi $ solutions respectively of  
\begin{gather*}
\begin{cases}  - \Delta \varphi = 0 \quad & \text{ in } \calF(t), \\ \nabla \varphi\cdot n = g \quad & \text{ in } \partial \calF(t), \end{cases} \quad  \begin{cases} \Delta \phi = \omega \quad & \text{ in } \calF(t), \\ \phi = 0 \quad & \text{ in } \partial \calF(t),\end{cases} \quad  \\ \text{ and } \quad \begin{cases}  
- \Delta \Psi_i = 0 \quad & \text{ in } \calF(t), \\ \Psi_i = 0 \quad & \text{ on } \partial \calF(t) \setminus \pS^{i}(t), \\
\Psi_i = 1 \quad & \text{ on } \pS^{i}(t),
\end{cases}
\end{gather*}
Finally $ X_i  $ is the unique harmonic vector field with circulation around $ \pS_i $ equal to the Kronecker delta $ \delta_{ij} $.

Note that $ v_{g} $ and $  \left[ \int_{\calF} \Psi_i \omega + \calC_i(t)  \right] X_i $ are $ L^{\infty } $ vector fields if we assume $ g $ smooth enough. We are left with the term $ \omega \calK^{0}_{\calF(t)}[\omega] $ which is not in general $ L^1 $, in particular the following expression, that appears in the weak formulation, 
\begin{equation*}
\int_{0}^{t}\int_{\calF} \omega \calK^{0}_{\calF(t)}[\omega] \cdot \nabla \varphi 
\end{equation*}
does not make sense. To avoid this issue we rewrite it in the same spirit of \cite{Del}. Let introduce the trilinear map 
\begin{equation}
\label{nl:term:delort:form}
\left\langle w, \omega, \varphi \right\rangle = \int_{\calF} \int_{\calF}H_{\varphi}(x,y)w(t,x)\omega(t,y) \, dx \, dy,  
\end{equation}
where 
\begin{equation}
\label{H:varphi}
H_{\varphi}(t,x,y) = - \nabla^{\perp}G(t,x,y)\cdot \frac{\nabla \varphi(t,x)-\nabla \varphi(t,y)}{2}.
\end{equation}
It holds 
\begin{equation*}
\left\langle \omega, \omega, \varphi \right\rangle = \int_{0}^{t}\int_{\calF} \omega \calK^{0}_{\calF(t)}[\omega] \cdot \nabla \varphi 
\end{equation*}
for smooth enough $ \omega $. Let us also notice that $ H_{\varphi} $ is in $ L^{\infty}(\calF(t)\times \calF(t)) $ if $ \varphi \in \mathfrak{C} = \{ \varphi \in W^{2,+\infty}(\calF(t))$ such that $ \varphi $ is constant in any connected component of $ \partial \calF(t)  \} $, see  Lemma \ref{H:phi:bounded}. This implies that \eqref{nl:term:delort:form} makes sense for any vorticity $ w = \omega \in L^{1}(\calF(t)) $.

\begin{Lemma}
\label{H:phi:bounded}

Let $ \varphi \in W^{2,+\infty}(\calF(t)) $ then
\begin{equation*}
 \|H_{\varphi}\|_{L^{\infty}(\calF(t) \times \calF(t))}  \leq M \|\varphi\|_{W^{2,+\infty}(\calF(t))},
\end{equation*}
where $ M $ does not depend on time $t$.

\end{Lemma}

Let us postpone the proof to the appendix \ref{app:H:phi:bounded}.

We are now able to define weak solutions for the system \eqref{judv:sys} with vorticity initially in $ L^{1} $.

\begin{definition}
\label{weak:for:del}
Let $ \omega^{in} \in L^1(\calF(0)) $ the initial vorticity, let $ \calC_{i}^{in} \in \bbR$ the initial circulations, let $ g \in L^{r}_{loc}(\bbR^{+}; W^{1/2,2 + \eps}(\partial \calF))$ for some $ \eps > 0 $, such that $ \oint_{\partial \calF } g = 0 $. Let $ \omega^{+} \in L^{1}_{loc}\left(\bbR^+ \odot \partial \calF^{+}(t), \lvert g-q \rvert dt ds )\right)$. Then a couple $ (\omega, \omega^{-}) \in \calL^{\infty}_{loc}(\bbR^{+}; L^{1}(\calF(t))) \times L^{1}_{loc}( \bbR^+ \odot \partial \calF^{-}(t), \lvert g- q \rvert dt ds )) $
is a weak solution of \eqref{judv:sys} if for any $ \varphi \in C^{\infty}_{c}(\bbR^{+}\times \overline{\calF}) \cap \mathfrak{C} $, 

\begin{align}
\label{wf:1:equ:ss}
\int_{\calF(t)} & \omega^{in}\varphi(0,.) dx  + \int_{\bbR^{+}}  \int_{\calF(t)} \omega \partial_t \varphi dx dt +\int_{\bbR^+} \int_{\calF(t)} \omega v_g \cdot \nabla \varphi dx dt +\int_{\bbR^{+}} \langle \omega, \omega, \varphi \rangle dt  \nonumber \\ & + \int_{\bbR^{+}}   \sum_{i} \left[ \int_{\calF(t)} \Psi_i \omega + \calC_i(t)  \right] \int_{\calF(t)} X_i \cdot \varphi  \omega  dx dt = \int_{\bbR^+}\int_{\partial \calF^{+}} (g-q) \omega^+ \varphi ds dt \\ & \, + \int_{\bbR^+}\int_{\partial \calF^{-}} (g-q) \omega^{-} \varphi ds dt,  \nonumber
\end{align}
and for almost any $ t \in \bbR^{+} $, the following estimate holds
\begin{equation*}
\int_{\calF} \lvert \omega(t,.)\rvert + \int_{0}^{t} \int_{\partial \calF^{-}} (g-q) \lvert \omega^{-}\rvert \leq \int_{\calF} \lvert \omega^{in}\rvert - \int_{0}^{t} \int_{\partial \calF^{+}} (g-q) \lvert \omega^{+}\rvert.
\end{equation*}

\end{definition}

\begin{remark}

In contrast to the case $ p > 1 $, we need to assume a better integrability on the normal component of the velocity $ g $ on $ \partial \calF(t) $ due to the fact that we need $ v_{g} \in L^{\infty } $ to make sense of the term $\int \omega v_g \cdot \nabla \varphi $. 

\end{remark}

\begin{remark}
In the case the fluid domain is $ \bbR^2 $ and $ p=1$ existence of renormalized solution for the transport equation of the vorticity was shown in \cite{CNSS}. It is not clear how to extend this result in our setting. 
\end{remark}

The existence result reads as follow.

\begin{theorem}
\label{exi:L1:ss}

Let $ \omega^{in} \in L^1(\calF(0)) $ the initial vorticity, let $ \calC_{i}^{in} \in \bbR$ the initial circulations, let $ g \in \calL^{r}_{loc}(\bbR^{+}; W^{1/2,2 + \eps}(\partial \calF))$ for some $ \eps > 0 $, such that $ \oint_{\partial \calF } g = 0 $ and let $ \omega^{+} \in L^{1}_{loc}\left(\bbR^+ \odot \partial \calF^{+}(t), \lvert g-q \rvert \, dt ds \right)$. Then there exists a solution $ (\omega, \omega^{-}) $ of the system \eqref{judv:sys} in the sense of Definition \ref{weak:for:del}. 

\end{theorem}

Note that the above theorem answer to the question left open in \cite{FLF} where the authors consider the special case with $ g = q $.

\section{Approximate problem}

In the case the holes have no-empty interior the existence of weak solutions can be shown by a vanishing viscosity approximation see Theorem \ref{exi:Lp:Nempty}. Moreover they satisfy the estimate  
\begin{equation}
\label{equ:65}
\int_{\calF(t)} \lvert \omega(t,.)\rvert^{p} + \int_{0}^{t}\int_{\partial \calF^{-}(t)} (g-q)\lvert\omega^{-}\rvert^{p} = \int_{\calF(0)} \lvert \omega^{in}\rvert^{p} - \int_{0}^{t}\int_{\partial \calF^{-}(t)}(g-q)\lvert\omega^{+}\rvert^{p},
\end{equation}
that follows from the renormalized transport equation \eqref{rem:tras:form} by choosing $ \beta(x) = \lvert x \rvert^{p} $ and $ \varphi = 1$. To show existence of solutions for the system \eqref{main:sys}, we consider a sequence of domains with holes with non-empty interior that approximate the subset of $ \bbR^2 $ occupied by the fluid and we use the bounds \eqref{equ:65} to pass to the limit in the weak formulations. 

The sequence of approximate geometries is defined as follow.

\begin{definition}[Sequence of approximate geometries]
Let $ (\Omega, \calS^{+}, \calS^{-}) $ a regular compatible geometry associated with the septuple $ (\Omega, S^+, h^+, r^+, S^-,h^-, r^- ) $. For $ \eps > 0 $, we say that $ (\Omega, \calS^{+}_{\eps}, \calS^{-}_{\eps}) $ is a sequence of approximate geometries if there exist  $ \Xi_{\eps}: \bbR^{+} \longrightarrow \bbR^{+} $ such that $ \Xi_{\eps} $ is smooth, increasing, $ \Xi_{\eps}(x) = \eps $ for $ x \in [0,\eps] $, $ \Xi_{\eps}(x) = x $ for $ x \in [2\eps, \infty) $ and  $ (\Omega, \calS^{+}_{\eps}, \calS^{-}_{\eps}) $ is associated with the septuple $ (\Omega, S^+, h^+, \Xi_{\eps}(r^+), S^-,h^-, \Xi_{\eps}(r^-) ) $.
\end{definition}
\begin{remark}
Note that for any $ T > 0$, there exist $ \eps_T $ such that for $ 0 < \eps < \eps_{T} $ the triple  	$ (\Omega, \calS^{+}_{\eps}, \calS^{-}_{\eps}) $ is a regular compatible geometry for $ t \in [0,T] $. 
	
\end{remark}  
The result in the following holds only local in time so without loss of generality we always assume that for a sequence of approximate geometries there exists $ \eps_{\infty} $ such that  $ (\Omega, \calS^{+}_{\eps}, \calS^{-}_{\eps}) $ is a regular compatible geometry for all the times $ t \in \bbR^+ $ and for $ 0 < \eps < \eps_{\infty} $.

Consider the approximate problem
\begin{align}
& \partial_t \omega_{\eps} + \div(v_{\eps} \omega_{\eps}) = 0 \quad \quad \quad \quad \quad \quad \quad \quad \quad \quad \quad \quad \quad \quad \quad \quad \quad \quad \text{for } x \in \calF_{\eps}(t),&&  \nonumber \\
& 
v_{\eps} =  \calK_{\calF_{\eps}(t)} \left[ \omega_{\eps} \right] \, \, \quad \quad \quad \quad \quad \quad \quad \quad \quad \quad \quad \quad \quad \quad \quad \quad \quad \quad\quad \quad \text{for } x \in \calF_{\eps}(t), &&  \label{app:judv:sys} \\
& \calC_{+,\eps}(t) = \calC_{+,\eps}^{in} - \int_{0}^{t} \oint_{\partial S^{+}_{\eps}(t)} \omega^{+}(g_{\eps} - q_{\eps}), &&  \nonumber \\
& \calC_{-,{\eps}}(t) = \calC_{-,{\eps}}^{in} + \int_{\calF_{\eps}(t)} \omega_{\eps} - \int_{\calF_{\eps}(0)} \omega^{in}_{\eps} + \int_{0}^{t} \oint_{\partial S^{+}_{\eps}(t)} \omega^{+}(g_{\eps} - q_{\eps}). \nonumber
\end{align}

We look for conditions on $ \omega^{in}_{\eps} $, $\calC^{in}_{i,_{\eps}} $, $ g_{\eps} $ and $ \omega^{+}_{\eps} $ such that weak solutions $(\omega_{\eps}, \omega^{-}_{\eps}, v_{\eps}) $ of \eqref{app:judv:sys}  converge to a weak solution of \eqref{main:sys}. 

Let denote for  $ \delta > 0 $ the set $ \mathcal{T}^{i}_{NP,\delta} = \{ t \in \bbR^{+} $ such that $ r^{i} \geq \delta \}$. Regarding the initial datum and the initial circulation, we ask that 
\begin{gather}
\label{1}
 \omega^{in}_{\eps} \longrightarrow \omega^{in} \quad \text{ in } L^{p}(\calF_0)  \\ \quad \text{ and  } \quad \calC^{in}_{i, \eps} \longrightarrow \calC^{in}_{i}. \nonumber 
\end{gather}
Regarding the boundary condition, we assume that for any $ \delta $ 
\begin{gather}
\label{2}
g_{\eps} \to g \text{ in } \mathcal{L}^{r}_{loc}(\mathcal{T}^{i}_{NP, \delta}; L^{q}(\pS^{i}(t))) \quad \\  \left(g_{\eps}-q_{\eps}\right)^{1/p}\omega^{+}_{\eps} \to \left(g-q\right)^{1/p}\omega^{+} \text{ in } \mathcal{L}^{p}_{loc}(\mathcal{T}^{+}_{NP,\delta}; L^{p}(\pS^{+}(t))), \nonumber
\end{gather}
\begin{gather}
\label{3}
\eps^{1/q} \|g_{\eps}\|_{L^{q}(\pS^{i}_{\eps}(t))} \longrightarrow 0 \quad \text{ in } L^{r}_{loc}(\calT^{i}) \quad \text{ and } \\ \quad \left\| (g_{\eps} - q_{\eps})^{1/p}\omega_{\eps}^{+}\right\|_{\calL^{p}(0,T;L^{p}(\pS^{+}_{\eps}(t))} \leq C_T, \nonumber
\end{gather}
for $ t \in \calT^{+} $, we assume  
\begin{equation}
\label{4}
\int_{\pS^{+}_{\eps}(t)} g_{\eps} \longrightarrow \mu \text{ in } L^{r}_{loc}(\calT^{+}) \quad \text{ and } \quad \int_{\pS^{+}_{\eps}(t)} (g_{\eps}-q_{\eps}) \omega^{+}_{\eps} \longrightarrow j \text{ in } L^{r}_{loc}(\calT^{+}).  
\end{equation}
Finally we assume the following equi-integrability property. For any compact $ K \subset \bbR^{+}$ and any small parameter $ \mathfrak{e} > 0 $, there exist $ \mathfrak{d} $ such that for any $ A \subset K $ with Lebesgue measure $ \mathcal{L}(A) < \mathfrak{d}$, it holds
\begin{equation}
\label{5}
\sup_{\eps} \left\| (r^{i}_{\eps})^{1/p}(t) \|g_{\eps}(t.)\|_{L^{q}(\pS^{i}_{\eps})} \right\|_{L^{r}(A)} < \mathfrak{e}.
\end{equation}

\begin{theorem}
\label{app:theo:theo}
	
Let $ (\Omega, \calS^+, \calS^-) $ a regular compatible geometry and let  $ (\Omega, \calS_{\eps}^+, \calS_{\eps}^-) $ a sequence of approximate geometry for $ \eps > 0 $.
Let $ p \in (2,+\infty) $, $ q \in [p/(p-1), 2) $ and $ r > 1 $. 
Let $ \omega^{in}_{\eps} \in L^{p}(\calF_{\eps}(0)) $ the initial vorticity, let $ \calC^{in}_{i,\eps} \in \bbR $ the initial circulations around $ \pS^i_{\eps} $. Let $ (g_{\eps},0) $ a sequence of compatible in-out velocities such that $ g_{\eps} \in \mathcal{L}^{r}_{loc}(\bbR^{+}; W^{1-1/p,p}(\partial \calF_{\eps}(t))) $. 
Let $ \omega^{+}_{\eps} \in L^{p}_{loc}(\bbR^{+} \odot \pS^{i}_{\eps}(t),\lvert g_{\eps}-q_{\eps}\rvert  \, dtds ) $ the entering vorticity. 
Suppose there exists $ \omega^{in} \in L^{p}(\calF(0))$, $ \calC^{in}_i \in \bbR $ and $ (g, \mu)$ a compatible in-out velocity such that $ g \in \mathcal{L}^{r}_{loc}(\bbR^{+}; W^{1-1/p,p}(\partial \calF(t))) $ and  $ \mu \in L^{r}_{loc}(\calT^{+}) $, $ \omega^{+} \in L^{p}_{loc}(\calT^{+}_{NP} \odot \pS^{i}(t),\lvert g-q \rvert \, dt ds)) $ an entering vorticity and $ j : \calT^{+} \longrightarrow \bbR^{+} $ such that $ j \in L^{r}_{loc}(\bbR^{+})$, for which it holds \eqref{1}-\eqref{2}-\eqref{3}-\eqref{4}-\eqref{5}. Then up to a subsequence the weak solutions $ (\omega_{\eps}, \omega_{\eps}^{-}, v_{\eps}) $ of the system \eqref{app:judv:sys}  associated with $ \omega^{in}_{\eps} $, $ \calC^{in}_{i, \eps} $, $ g_{\eps}^i $, $\omega_{\eps}^{+} $ converges to $( \omega, \omega^{-}, v) $ a solution of \eqref{main:sys} in the sense of Definition \ref{main:def:def}. Moreover the convergence holds in the following sense.
\begin{itemize}
		
		\item $ \omega_{\eps} \cvwstar \omega $ in $\mathcal{L}^{\infty}_{loc}(\bbR^{+}; L^{p}(\calF(t))) $,
		
		\item $ (g_{\eps}-q_{\eps})^{1/p}\omega_{\eps}^{-} \cv (g - q)^{1/p} \omega^{-} $ in $ L^{p}_{loc}(\calT^{-}_{NP,\delta} \odot \pS^{-}(t))) $ for any $ \delta > 0 $.
		
		\item $ v_{\eps} \longrightarrow v $ in $ \mathcal{L}^{r}_{loc}(\bbR^{+}; L^{q}(\calF(t))) $.
		
\end{itemize}  
	
\end{theorem}

First of all we show Theorem \ref{main:theorem} as a corollary of the above Theorem \ref{app:theo:theo}. In particular it remains only to prove Theorem \ref{app:theo:theo}.

\begin{proof}[Proof of Theorem \ref{main:theorem}] We show that there exists a sequence of data $ \omega_{\eps}^{in} $, $ C^{in}_{i,\eps} $, $ g_{\eps}^i $, $\omega_{\eps}^{+} $ that satisfy the hypothesis of Theorem \ref{app:theo:theo} and in particular  \eqref{1}-\eqref{2}-\eqref{3}-\eqref{4}-\eqref{5}.
Theorem \ref{app:theo:theo} implies the existence of a weak solution of \eqref{main:sys}.

We consider $ \omega_{\eps}^{in} = \omega^{in}\rvert_{\calF_{\eps}(0)} $ and $ \calC^{in}_{i,\eps} = \calC^{in}_{i}$, for which \eqref{1} is clear.
We use the notation $ \mathbf{x} = h^i(t) + \frac{r^{i}(t)}{r_{\eps}^{i}(t)} ( x-h^i (t)) $, we define  
\begin{equation*}
g_{\eps}(t,x) = \begin{cases} g(t,x) & \text{ in } \calT^{i}_{NP, 2 \eps}\odot \pS^{i}_{\eps}(t), \\
\left(\frac{r^i(t)}{r_{\eps}^{i}}\right)^{1/q} \left(g \left(t, \mathbf{x} \right) - q\left(t, \mathbf{x} \right)\right) + q_{\eps}  & \text{ in } \left(\calT^{i}_{NP} \setminus \calT^{i}_{NP,2\eps}\right)\odot \pS^{i}_{\eps}(t), \\ \frac{\mu}{\oint_{\pS^{i}_{\eps}(t)} 1 } + q_{\eps} & \text{ in } \left(\calT^{i} \cap \{ t \vert \mu(t) \neq 0 \} \right)\odot \pS^{i}_{\eps}(t), \\ \frac{\eps}{\int_{\pS^{i}_{\eps}}1} + q_{\eps}  & \text{ in } \{ t \vert \mu(t) \neq 0 \} \odot \pS^{i}_{\eps}(t).
\end{cases}
\end{equation*} 
and 
\begin{equation*}
\omega_{\eps}^{+} = \begin{cases} \omega^{+}(t,x) \quad & \text{ for } (x,t) \in \calT^{i}_{NP,2\eps} \odot \pS^{i}_{\eps}(t), \\ \frac{j}{\mu}  \quad & \text{ for } (t,x) \in \calT^{+}\odot \pS^{i}_{\eps} (t) , \\ 0 \quad & \text{ else}.\end{cases}
\end{equation*}

With this choice of $ \omega^{in}_{\eps}$, $ \calC^{in}_{i,\eps}$, $ g_{\eps}$ and $ \omega_{\eps}^+ $, we can apply Theorem \ref{app:theo:theo} to show Theorem \ref{main:theorem}.

\end{proof}

\section{Proof of Theorem \ref{exi:transport:ppppppppp} and Theorem \ref{duality:formula}}
\label{sec:ren}

In this section we show existence and uniqueness of renormalized solutions to the transport equation with in-out flow and with time dependent domain. Moreover we prove that renormalized solutions to the transport equation satisfy a duality formula. These results are an extension to the classical DiPerna-Lions theory from \cite{DPL}, using some tools introduced in \cite{Boyer} where the case of time-independent domain was tackled.

To show Theorem \ref{exi:transport:ppppppppp} we proceed as follow. First of all we show that distributional solutions to the transport equation with test functions which vanish on the boundary admit traces on the boundary and they satisfy the renormalized transport equation in a weak sense with test functions also non-zero on the boundary.
Then we show existence of weak solutions to the transport equation and we identify them with the renormalized solutions of the transport.

\subsection{Distributional solutions are renormalized solutions}

In this subsection we show that distributional solutions of the transport equation admits traces on the boundary $ \partial \calF^+ \cup \partial \calF^{-} $ and are renormalizable under some hypothesis on the regularity of the velocity field.

Recall that we restrict our analysis to vector fields $ v \in L^{1}_{loc}(\bbR^+, W^{1,q}(\calF(t))) $ such that $ \div(v) = 0 $ and they satisfy at least one of the following hypothesis  
\begin{enumerate}
	\item (Geometric condition). There exists $ \psi \in C^{\infty}(\mathbb{F}) $ such that $ \psi = 1 $ on $ \partial \calF^+ $ and $ \psi = 0 $ on $\partial \calF^- $.
	
	\item (Extra regularity). There exists $ \alpha > 1 $ such that $$ v\cdot n \in L^{\alpha}_{loc}(\bbR^{+} \odot \partial \calF(t)). $$
\end{enumerate}

\begin{theorem}
	\label{theo:dist:equiv:ren}
	Let $ p \in (1,+\infty] $ and $ 1/p + 1/q = 1 $. Let $ v \in \calL^1_{loc}(\bbR^+;W^{1,q}(\calF)) $ such that $\div(v) = 0 $ satisfying or the geometric condition or the extra regularity hypothesis. Let $ \rho^{in} \in L^p(\calF(0)) $ and let $ f \in \calL^{1}_{loc}(\bbR^+;L^p(\calF(t))) $.
	Let $ \rho \in \calL^{\infty}_{loc}(\bbR^+;L^p(\calF)) $ a solution of
	\begin{equation}
	\label{weak:distr:tran}
	\int_{\calF(0)} \rho^{in} \varphi(0,.) \, dx + \int_{\bbR^+}\int_{\calF(t)} \rho (\partial_t \varphi + v\cdot \nabla \varphi) \, dx dt + \int_{\bbR^+}\int_{\calF(t)} f \varphi \, dx dt  = 0 
	\end{equation}
	for any  $ \varphi \in C^{\infty}_{c}(\bbR^+ \odot \calF(t) )$.
	Then there exists a unique trace $ \gamma \rho $ a measurable function on $ \bbR^+ \odot \partial \calF(t) $ defined almost everywhere respect to the measure $ \lvert v\cdot n - q \rvert \, dsdt $, such that for any $ \beta \in C^{1}_b(\bbR^+)$ it holds 
	\begin{align*}
	\int_{\calF(0)} \beta(\rho^{in}) \varphi(0,.) \, dx + \int_{\bbR^+}\int_{\calF(t)} \beta(\rho) (\partial_t \varphi & \, + v\cdot \nabla \varphi) \, dx dt \\ +  \int_{\bbR^+}\int_{\partial \calF(t)} f \beta'(\rho) \varphi \, dx dt    & \, = \int_{\bbR^+}\int_{\partial \calF(t)} \beta(\gamma \rho) v\cdot n \, ds dt. 
	\end{align*}
	for any $ \varphi \in C^{\infty}_c(\bbR^+ \odot \overline{\calF(t)}) $.
\end{theorem}

\begin{remark}
	
	Note that we exclude the case $ p = 1 $ because in this case $ q = +\infty $ and we can follow characteristics. 
	
\end{remark}

\begin{remark}
	
	It is not true that $ \gamma \rho $ is $  L^1_{loc}(\bbR^+ \odot \partial \calF(t); \lvert v\cdot n - q \rvert \, ds dt )$ in general, in fact Bardos present a counterexample in \cite{Bardos}. But in the case where the velocity fields satisfies the geometric condition we are able to show that $\gamma \rho \in L^1_{loc}(\bbR^+ \odot \partial \calF(t); \lvert v\cdot n \lvert \, ds dt ) $. In the case of the extra regularity the trace $\gamma \rho $ is $ L^1_{loc} $ respect to the measure $ \gamma_{\tau} \lvert v\cdot n -q \lvert^{\tilde{\alpha}} \, ds dt  $ which takes in account the life time $ \gamma_{\tau} $ of the characteristic that enter or exit from $ (x,t) \in \mathbb{F} $, for more details see \cite{Boyer} Section 5.  
	
\end{remark}

\begin{proof}
	
	Although the proof is classical and follows the one of \cite{DPL} and \cite{Boyer}, some extra technicalities are needed. So let us briefly prove the result. Let us start with a technical Lemma that will be showed only after.
	
	\begin{Lemma}
		\label{Lem:5}
		Let $ p \in (1,+\infty] $ and $ 1/p + 1/q = 1 $. Let $ v \in \calL^1_{loc}(\bbR^+;W^{1,q}(\calF)) $ such that $\div(v) = 0 $ satisfying or the geometric condition or the extra regularity hypothesis. Let $ \rho^{in} \in L^p(\calF(0)) $ and let $ f \in \calL^{1}_{loc}(\bbR^+;L^p(\calF(t))) $.
		Let $ \rho \in \calL^{\infty}_{loc}(\bbR^+;L^p(\calF)) $ a solution of \eqref{weak:distr:tran}. Then there exist sequences $ \rho_{\eps} \in \calL^{\infty}(\bbR^+;C^1(\overline{\calF(t)})) \cap W^{1,1}_{loc}(\bbR^+ \odot \calF(f)) $ and $ f_{\eps} \in \calL^1_{loc}(\bbR^+;C^1(\overline{\calF(t)})) $ such that  
		\begin{equation}
		\label{trans:eps:equ}
		\partial_t \rho_{\eps} + v \cdot \rho_{\eps}+f_{\eps} = R_{\eps} \quad \text{ almost everywhere in } \bbR^{+}\odot \calF(t)
		\end{equation} 
		and 
		\begin{gather*}
		\rho_{\eps} \longrightarrow \rho \text{ in } \calL^s_{loc}(\bbR^+;L^{p}(\calF(t))) \text{ for any } s < + \infty, \quad \rho_{\eps}(0,.) \longrightarrow \rho^{in} \text{ in  } L^p(\calF(0)) \\ f_{\eps} \longrightarrow f \text{ in } \calL^1_{loc}(\bbR^+; L^p(\calF(t))) \quad \text{ and } \quad R_{\eps} \longrightarrow 0 \text{ in } L^1_{loc}(\bbR^+ \odot \calF(t)).
		\end{gather*}
		
	\end{Lemma}
	
	Let now show existence of a trace with the help of $ \rho_{\eps} $. Let $ \delta > 0 $ a small parameter and let $ \beta( x ) = x^2 / \sqrt{1+x^2} $. After taking the difference of \eqref{trans:eps:equ} satisfied by $ \rho_{\eps}$ and $ \rho_{\delta}$ and multiplying $ \beta'(\rho_{\eps}-\rho_{\delta})$ and using the fact that $ v $ is divergence free, we deduce that
	\begin{equation*}
	\partial_t \beta(\rho_{\eps}-\rho_{\delta}) + v \cdot \nabla \beta(\rho_{\eps}-\rho_{\delta}) =  (R_{\eps}-f_{\eps}-R_{\delta}+f_{\delta})\beta'(\rho_{\eps}-\rho_{\delta}) 
	\end{equation*} 
	almost everywhere in $ \bbR^{+}\odot \calF(t)$.
	In the case where the velocity field satisfies the geometric condition, let multiply the above equation by $\psi \phi $ with $ \phi \in  C^{\infty}_c((0,+\infty))$  and integrate by part to obtain
	\begin{align*}
	& \int_{\bbR^+} \int_{\partial \calF(t)}  \beta(\rho_{\eps}- \rho_{\delta}) \psi \phi (v\cdot n-q) =  \int_{\bbR^+}\int_{\calF(t)} \beta(\rho_{\eps}- \rho_{\delta}) \psi \partial_t \phi \\ & \, + \int_{\bbR^+}\int_{\calF(t)} v\cdot \nabla \psi \phi \beta(\rho_{\eps}- \rho_{\delta})   + \int_{\bbR^+}\int_{\calF(t)} (R_{\eps} - f_{\eps}- R_{\delta}-f_{\delta})\beta'(\rho_{\eps}- \rho_{\delta}) \\ & \, + \int_{\calF(0)} \beta(\rho_{\eps}(0,.)- \rho_{\delta}(0,.))\psi\phi(0)
	\end{align*}
	Recall that $ \psi = 1 $ on $ \partial \calF^+$ and $ 0 $ in the remaining boundary. Choose $ \phi \geq 0 $ such that $ \phi = 1 $ in $[0,n] $, then
	\begin{equation*}
	\int_{\partial \calF^+ \cup ([0,n]\odot \partial \calF(t))} \beta(\rho_{\eps}- \rho_{\delta}) \lvert v\cdot n-q \rvert \to 0
	\end{equation*} 
	as $ \eps, \delta $ converge to zero, in fact $ \beta(\rho_{\eps}-\rho_{\delta}) $ converges to $ 0 $ in $ L^p_{loc} $, $ \beta' $ is bounded, $ f_{\eps } $ and $ R_{\eps} $ converges to zero in $ L^1_{loc} $. Analogously by using $ 1- \psi $  we deduce
	\begin{equation*}
	\int_{\partial \calF^- \cup ([0,n]\odot \partial \calF(t))} \beta(\rho_{\eps}- \rho_{\delta}) \lvert v\cdot n-q \rvert \to 0
	\end{equation*} 
	Finally note that $ \int \beta(w)  $ is equivalent to $ \int \lvert w \rvert $ in bounded domain, we deduce that $ \rho_{\eps} $ is a Cauchy sequence in $ L^1_{loc}(\bbR^+\odot \partial \calF(t); \lvert v\cdot n - q \rvert ds dt )$. Then there exists $ \gamma \rho $ such that 
	\begin{equation*}
	\rho_{\eps}\vert_{\bbR^+\odot \partial \calF(t)} \longrightarrow \gamma \rho \quad \text{ in } L^1_{loc}(\bbR^+\odot \partial \calF(t); \lvert v\cdot n - q \rvert ds dt ).
	\end{equation*} 
	
	Let us now consider the case where $ v $ satisfies the extra regularity property. The idea, from \cite{Boyer}, is to multiply \eqref{trans:eps:equ} by test functions such that behaves as $ \lvert v\cdot n -q \rvert^{\alpha-1}(v\cdot n -q) $ in the boundary and deduce the convergence in a similar way as before. Note that $ \lvert v\cdot n -q \rvert^{\alpha-1}(v\cdot n -q) $ is not regular enough to be the trace of a smooth test function so let $ g_{m} $ a smooth sequence of functions such that  $ g_{m} $ converges to $ \lvert v\cdot n -q \rvert^{\alpha-1}(v\cdot n -q) $ in $ L^{\alpha/(\alpha-1)}_{loc}(\bbR^+\odot \partial \calF(t)) $   and $ G_{m} $ a smooth function in $ \bbR^{+} \odot \calF(t) $ with trace $ g_m $. Let $ \beta = x^2/\sqrt{1+x^2} $  an let $ \mathfrak{b}$ an odd smooth strictly increasing function such that $ \mathfrak{b}(x) $ converges to $ 1$ as $ x $ tends to  infinity. Let $ \eps, \delta > 0 $ two small parameter then from \eqref{trans:eps:equ}, we deduce 
	\begin{align*}
	\partial_t \beta(\mathfrak{b}(\rho_{\eps})-\mathfrak{b}(\rho_{\delta})) + & \, v \cdot \nabla \beta(\mathfrak{b}(\rho_{\eps})-\mathfrak{b}(\rho_{\delta})) =  \\ & \,  \left[(R_{\eps}-f_{\eps})\mathfrak{b}'(\rho_{\eps})-( R_{\delta}-f_{\delta})\mathfrak{b}'(\rho_{\delta})\right]\beta'(\mathfrak{b}(\rho_{\eps})-\mathfrak{b}(\rho_{\delta})).
	\end{align*}
	Let multiply by $ \tilde{G}_{m} = G_{m} \phi $ with $ \phi \in C^{\infty}_c([0,+\infty))$, such that $ \phi = 1 $ in $ [0, N] $ and $\phi \geq 0 $. And after some integration by parts we deduce
	\begin{align*}
	& \int_0^N\int_{\partial \calF(t)}  \beta(\mathfrak{b}(\rho_{\eps})-\mathfrak{b}(\rho_{\delta})) \lvert  v\cdot n - q \rvert^\alpha  \leq \\ & \,  \int_{\bbR^+}\int_{\partial \calF(t)} \beta(\mathfrak{b}(\rho_{\eps})-\mathfrak{b}(\rho_{\delta})) ( v\cdot n - q ) (\lvert v\cdot n - q \rvert^{\alpha-2}(v\cdot n -q ) - g_m )\phi \\  & \, +\int_{\bbR^+}\int_{\calF(t)} \beta(\mathfrak{b}(\rho_{\eps})-\mathfrak{b}(\rho_{\delta}))(\partial_t \tilde{G}_{m} + v \cdot \nabla \tilde{G}_m) \\ & \,+\int_{\bbR^+}\int_{\calF(t)} \left[(R_{\eps}-f_{\eps})\mathfrak{b}'(\rho_{\eps})-(R_{\delta}-f_{\delta})\mathfrak{b}'(\rho_{\delta})\right]\beta'(\mathfrak{b}(\rho_{\eps})-\mathfrak{b}(\rho_{\delta})) \tilde{G}_m
	\end{align*}
	Note that that the first term of the right hand side can be made small as we want by assuming $ m $ big enough independently of $ \eps $, $ \delta$. Fix now $ m $ then the last two terms converges to zero as $ \lvert \eps \rvert+ \lvert \delta \rvert $ converges to zero. We showed that  
	\begin{equation*}
	\int_0^N\int_{\partial \calF(t)}  \beta(\mathfrak{b}(\rho_{\eps})-\mathfrak{b}(\rho_{\delta})) \lvert v\cdot n - q \rvert^\alpha \longrightarrow 0 \quad \text{ as  
	} \lvert \eps \rvert+\lvert \delta \rvert \longrightarrow 0,
	\end{equation*}
	we deduce that $ \mathfrak{b}(\rho_{\eps}) $ is a Cauchy sequence in $L^1_{loc}(\bbR^+\odot \partial \calF(t), \lvert v\cdot n - q \rvert ^{\alpha} \, ds dt )$, so it converges to $ \gamma \mathfrak{b}(\rho) = \mathfrak{b}(\gamma \rho) $, where we use the invertibility of $ \mathfrak{b} $. Let conclude with showing $ \gamma\rho $ is finite almost everywhere respect to the $ \lvert v\cdot n - q \rvert ^{\alpha} \, ds dt $ measure. Suppose that $ \gamma \rho $ is infinite in a set $ E \subset [0,N] \odot \partial \calF(t) $ with 
	$$ c = \int_{E} \lvert v\cdot n - q \lvert^{\alpha} > 0. $$
	This implies that $ \lvert \mathfrak{b}(\gamma\rho) \rvert = 1 $ in $ E $. Passing to subsequence Egoroff's Theorem ensures that $ \mathfrak{b}(\rho_{\eps}) $ converges uniformly to $ \mathfrak{b}(\gamma \rho) $ in 
	$ \bar{E} $ with measure of $ \bar{E} \subset E $ greater than $ c/2 $. Introduce $ \mathfrak{b}_{M}(x) = M \mathfrak{b}(x/M) $. We have 
	\begin{align*}
	\int_0^N\int_{\partial \calF(t)} & \beta(\mathfrak{b}_{M}(\rho_{\eps})) \lvert v\cdot n - q \rvert^\alpha  \leq  \\ & \, \int_{\bbR^+}\int_{\partial \calF(t)} \beta(\mathfrak{b}_{M}(\rho_{\eps})) ( v\cdot n - q ) (\lvert v\cdot n - q \rvert^{\alpha-2}(v\cdot n -q ) - g_m )\phi \\  & \,  +\int_{\bbR^+}\int_{\calF(t)} \beta(\mathfrak{b}_{M}(\rho_{\eps}))(\partial_t \tilde{G}_{m} + v \cdot \nabla \tilde{G}_m) \\ & \, +\int_{\bbR^+}\int_{\calF(t)} (R_{\eps}-f_{\eps})\mathfrak{b}_{M}'(\rho_{\eps})\beta'(\mathfrak{b}(\rho_{\eps}))\tilde{G}_m \\
	\leq & \, (M+1)\frac{c\mathfrak{b}(1)}{4} + 1  \leq  M \frac{c\mathfrak{b}(1)}{4} + 2, 
	\end{align*}
	where we fixed $ m $ that $ \int_{\bbR^+}\int_{\partial \calF(t)}  ( v\cdot n - q ) (\lvert v\cdot n - q \rvert^{\alpha-2}(v\cdot n -q ) - g_m )\phi  \leq c\mathfrak{b}(1)/4 $, $ \beta(\mathfrak{b}_{M})(x) \leq M+1 $ and $ \beta(\mathfrak{b}_{M})(x) \leq \sup \lvert \beta'\rvert \sup \lvert \mathfrak{b}'_M\rvert \lvert x \rvert \lesssim \lvert x \rvert$. From the uniform convergence of $ \mathfrak{b}(\rho_{\eps}) $, we deduce that there exists $ \eps_{M} $ such that $ \lvert \rho_{\eps_M}\rvert \geq M $ in $ \bar{E} $. We have
	\begin{align*}
	\int_0^N\int_{\partial \calF(t)}  \beta(\mathfrak{b}_{M}(\rho_{\eps_M})) \lvert v\cdot n - q \rvert^\alpha \geq & \, \int_{\bar{E}}  \beta(\mathfrak{b}_{M}(\rho_{\eps_M})) \lvert v\cdot n - q \rvert^\alpha  \\   \geq & \, \int_{\bar{E}}  \mathfrak{b}_{M}(\rho_{\eps_M}) \lvert v\cdot n - q \rvert^\alpha \geq  \frac{M\mathfrak{b}(1)c}{2}.
	\end{align*}
	Putting all together we deduce  
	$$ \frac{M\mathfrak{b}(1)c}{2} \leq \frac{M\mathfrak{b}(1)c}{4} +2, $$
	for any $ M > 0 $ which is the desired contradiction.

	Up to subsequence we deduced that $ \rho_{\eps}$ satisfies \eqref{trans:eps:equ}. Moreover 
	\begin{gather*}
	\rho_{\eps} \longrightarrow \rho \text{ in } L^s_{loc}(\bbR^+;L^{p}(\calF(t))) \text{ for any } s < + \infty, \quad \rho_{\eps}(0,.) \longrightarrow \rho_0 \text{ in  } L^p(\calF(0)), \\ \rho_{\eps}\vert_{\bbR^+\odot \partial \calF(t)} \longrightarrow \gamma \rho  \quad \text{ almost everywhere repect to the measure }\lvert v\cdot n - q \rvert \, ds dt \\  f_{\eps} \longrightarrow f \text{ in } \calL^1_{loc}(\bbR^+;L^p( \calF(t))) \quad \text{ and } \quad R_{\eps} \longrightarrow 0 \text{ in } L^1_{loc}(\bbR^+ \odot \calF(t)).
	\end{gather*}
	Let $ \beta \in C^{\infty}_{b}(\bbR) $ and $ \psi \in C^{\infty}_{c}(\bbR^+ \odot \overline{\calF(t)}) $ we have
	\begin{align*}
	\int_{\calF(0)} \beta(\rho_{\eps})(0,.) \psi(0,.) + & \, \int_{\bbR^+}\int_{\calF(t)} \beta(\rho_{\eps})(\partial_t \psi + v\cdot \nabla \psi ) =  \\ & \, \int_{\bbR^+} \int_{\partial \calF(t)} \beta(\rho_{\eps}) \psi (v\cdot n - q) + \int_{\bbR^+} \int_{\calF(t)} (R_{\eps}-f_{\eps})\beta'(\rho_{\eps}) 
	\end{align*}
	From the above convergence is now straight-forward to pass to the limit in the formulation with the help of dominate convergence's theorem, to obtain
	\begin{align*}
	\int_{\calF(0)} \beta(\rho)(0,.) \psi(0,.) + & \,  \int_{\bbR^+}\int_{\calF(t)} \beta(\rho)(\partial_t \psi + v\cdot \nabla \psi ) = \\ & \, \int_{\bbR^+} \int_{\partial \calF(t)} \beta(\gamma \rho) \psi \lvert v\cdot n - q \lvert + \int_{\bbR^+}\int_{\calF(t)} f \beta'(\rho) \psi. 
	\end{align*}
\end{proof}

\begin{proof}[Proof of Lemma \ref{Lem:5}] The proof follows the original idea of Di Perna and Lions in \cite{DPL} with an ingredient used in \cite{Boyer} to deal with the transport equation with in out-flow. 

 Let us introduce some notations. First of all there exists an open neighbourhood $ E \subset [0,T] \odot \overline{\calF(t) } $ of  $ [0,T] \odot \partial \calF(t) $ such that the distant function $ D(t,x) = \dist\{(x,t),\{t\}\times \partial \calF(t) \} $ is $ C^{1}([0,T];C^2(\partial \calF(t))) $, see for instance \cite{dist:fun}. 
 Let $ \Psi \in C^{\infty}_c(E) $ such that $ \psi = 1 $ in an open neighbourhood of   $ [0,T] \odot \partial \calF(t) $. Then we define the normal close to the boundary by $ \nu(t,x) = \Psi \nabla_x(D(t,x)) $. Finally let $ \eta \in C^{\infty}_{c}(B_1(0)) $ such that $ \eta\geq 0 $ and $ \int \eta = 1 $. For $ \eps $ we define $ \eta_{\eps} = \eta(./\eps)/\eps $ and 
 \begin{equation}
 \label{con:boy}
 \psi \star_{\nu} \eta_{\eps}(t,y) = \int_{\calF(t)}\psi(t,x) \eta_{\eps}(y-x - 2\eps\nu(t,y)) dx. 
 \end{equation}
 Let us recall some properties.
 
 \begin{Lemma}
 	
 	For $ p,q \in [1,+\infty]$, if $ \psi \in L^p(0,T;L^q(\calF(t))) $, then $ \psi \star_{\nu} \eta_{\eps} \in L^p(0,T;C^1(\bar{\calF(t)})) $ and 
 	$$ \|\psi \star_{\nu} \eta_{\eps}\|_{L^p(0,T;L^q(\calF(t)))}\| \leq C \| \psi \|_{L^p(0,T;L^q(\calF(t))) } $$ $$ \text{ and } \quad \|\nabla(\psi \star_{\nu} \eta_{\eps})\|_{L^p(0,T;L^q(\calF(t)))}\| \leq \frac{C}{\eps} \| \psi \|_{L^p(0,T;L^q(\calF(t))) }. $$ 
 	Moreover for $ p, q < + \infty $, it holds 
 	$$ \psi \star_{\nu} \eta_{\eps} \longrightarrow \psi \quad \text{ in } L^p(0,T;L^q(\calF(t))).$$
 \end{Lemma}

 For any fixed $ y \in \mathbb{R}^d $, consider the function $ \varphi_{\eps}(t,x) = \psi(t,y)\eta_{\eps}(y-x-2\eps \nu(t,y)) $ where $ \psi \in C^{\infty}_c(\bbR^+ \odot \calF(t)) $. Note that $ \phi \in C^{\infty}_c(\bbR^+ \odot \calF(t)) $ for $ \eps $  small enough.  
 If we test the distributional formulation of the transport equation with $ \varphi_{\eps}$ and we integrate in $ y $ in all $ \bbR^d$, we deduce
 \begin{align*}
 \int_{\bbR^d}\int_{\calF(0)} \rho_0 \varphi_{\eps}(0,.) + & \, \int_{\bbR^d}\int_{\bbR^+}\int_{\calF(t)} \rho \partial_t \varphi_{\eps} \\ & \, + \int_{\bbR^d}\int_{\bbR^+}\int_{\calF(t)} \rho v \cdot \nabla \varphi_{\eps} = \int_{\bbR^d}\int_{\bbR^+}\int_{\calF(t)}  f \varphi_{\eps}
 \end{align*}
 Let rewrite any of the four terms above. The first one
 \begin{align*}
 \int_{\bbR^d}\int_{\calF(0)} \rho_0 \varphi_{\eps}(0,.) = & \,  \int_{\bbR^d}\int_{\calF(0)} \rho_0 \psi(0,y) \eta_{\eps}(y-x-2\eps \nu(0,y)) \\ =  & \,  \int_{\calF(0)} \rho_{0}\star_{\nu} \eta_{\eps} \psi(0,.).
 \end{align*}
 The second one
 \begin{align*}
 \int_{\bbR^d}\int_{\bbR^+}\int_{\calF(t)} \rho \partial_t \varphi_{\eps} = & \int_{\bbR^d}\int_{\bbR^+}\int_{\calF(t)} \rho(t,x) \partial_t \left(\psi(t,y)\eta_{\eps}(y-x-2\eps \nu(t,y)) \right) \\
 = & \, \int_{\bbR^d}\int_{\bbR^+}\int_{\calF(t)} \rho(t,x) \partial_t \left(\psi(t,y)\right)\eta_{\eps}(y-x-2\eps \nu(t,y)) \\
  \, + \int_{\bbR^d}\int_{\bbR^+}&\int_{\calF(t)}  \rho(t,x) \psi(t,y)2\eps \partial_t \nu(t,y) \cdot \nabla\eta_{\eps}(y-x-2\eps \nu(t,y)) \\
 = & \, \int_{\bbR^+}\int_{\calF(t)} \rho \star_{\nu} \eta_{\eps}(t,y) \partial_t \psi(t,y)  + \int_{\bbR^+}\int_{\calF(t)} R_{\eps}^1.
 \end{align*}
 The third one
 \begin{align*}
 \int_{\bbR^d} & \int_{\bbR^+}\int_{\calF(t)} \rho v \cdot \nabla \varphi_{\eps} = \\ & \,  \int_{\bbR^d}\int_{\bbR^+}\int_{\calF(t)} \rho(t,x) v(t,x) \cdot \nabla \left(\psi(t,y)\eta_{\eps}(y-x-2\eps \nu(t,y))  \right) \\
 = & \, \int_{\bbR^d}\int_{\bbR^+}\int_{\calF(t)} \rho(t,x) v(t,x) \cdot  \psi(t,y)\nabla_{x} \eta_{\eps}(y-x-2\eps \nu(t,y))  \\
 = & \, \int_{\bbR^+}\int_{\calF(t)} \rho \star_{\nu} \eta_{\eps}(t,y) v(t,y) \cdot \nabla \psi(t,y) - \int_{\bbR^+}\int_{\calF(t)} \rho \star_{\nu} \eta_{\eps}(t,y) v(t,y) \cdot \nabla \psi(t,y) \\
 & + \int_{\bbR^d}\int_{\bbR^+}\int_{\calF(t)} \rho(t,x) v(t,x) \cdot  \psi(t,y)\nabla_{x} \eta_{\eps}(y-x-2\eps \nu(t,y)) \\
 = & \, \int_{\bbR^+}\int_{\calF(t)} \rho \star_{\nu} \eta_{\eps}(t,y) v(t,y) \cdot \nabla \psi(t,y) \\
 & +\int_{\bbR^+}\int_{\calF(t)} \int_{\calF(t)} \rho(t,x)\Big(v(t,y) \cdot \nabla_y \eta_{\eps}(y-x-2\eps \nu(t,y))  \\
 & \quad \quad \quad \quad \quad \quad \quad \quad \quad +v(t,x) \cdot  \psi(t,y)\nabla_{x} \eta_{\eps}(y-x-2\eps \nu(t,y)) \Big) \psi(t,y) \\
 = & \, \int_{\bbR^+}\int_{\calF(t)} \rho \star_{\nu} \eta_{\eps}(t,y) v(t,y) \cdot \nabla \psi(t,y) + \int_{\bbR^+}\int_{\calF(t)} R^2_{\eps}(t,y) \psi(t,y).
 \end{align*}
 Finally
 \begin{align*}
 \int_{\bbR^d}\int_{\bbR^+}\int_{\calF(t)} f  \varphi_{\eps} =  & \int_{\bbR^d}\int_{\bbR^+}\int_{\calF(t)} f(t,x)   \psi(t,y)\eta_{\eps}(y-x-2\eps \nu(t,y)) \\
 = & \, \int_{\bbR^+}\int_{\calF(t)} f \star_{\nu} \eta_{\eps}(t,y)  \psi(t,y). 
 \end{align*}
 We deduce that 
 \begin{equation*}
 \partial_t \rho_{\eps} + v\cdot \nabla \rho_{\eps} + f_{\eps} = - R^1_{\eps}-R^2_{\eps},
 \end{equation*}
 for almost any $ (t,x) \in \bbR^+ \odot \calF(t)$, where $ \rho_{\eps} = \rho \star_{\nu} \eta_{\eps} $.
 
 It remains to show that $ R_{\eps}^1 $ and $ R_{\eps}^2 $ converge to zero in $ L^{1}((0,T) \odot \calF(t)) $. Note that the convergence of $ R_{\eps}^2 $ was showed in Lemma 3.1 of \cite{Boyer} and is a Friedrichs-type commutator lemma. So let us briefly explain how to show the convergence to zero of the term $ R^1_{\eps} $. Since $ \partial_t \nu $ is continuous in $ [0,T] \odot \calF(t) $, it is enough to show that 
 \begin{equation*}
 \int_{\calF(t)} \rho(t,x) \eps \nabla\eta_{\eps}(y-x-2\eps \nu(t,y)) \, dx \longrightarrow 0  \quad \text{ in } L^1((0,T )\otimes \calF(t)). 
 \end{equation*} 
 Note that the vector fields $ \eps \nabla \eta_{\eps} $ have integral zero and are uniformly bounded in $L^1 $. We have
 \begin{align*}
 & \left\|\int_{\calF(t)} \rho(t,x) \eps \nabla\eta_{\eps}(y-x-2\eps \nu(t,y)) \, dx \right\|_{L^1((0,T) \odot \calF(t))} \\ & \quad  =  \left\|\int_{B(y-2\eps\nu(t,y),\eps)} (\rho(t,x)- \rho(t,y - 2 \eps \nu(t,y))) \eps \nabla\eta_{\eps}(y-x-2\eps \nu(t,y)) \, dx \right\|_{L^1
 } \\
 & \quad \leq \|\eps \nabla \eta_{\eps} \|_{L^1} \sup_{\lvert h \rvert \leq \eps} \|\rho(t, y - 2 \eps \nu(t,y) + h )- \rho(t, y - 2 \eps \nu(t,y))\|_{L^1((0,T) \odot \calF(t))} \longrightarrow 0,  
 \end{align*}   
 because  $ \rho(t, y - 2 \eps \nu(t,y) ) $ is $ L^1 $, which implies that its $ L^1$ modulus of continuity converge to zero.
 	
\end{proof}

\begin{remark}
	
	Lemma \ref{Lem:5} holds true also for less regular geometry, in particular it is enough that the boundaries satisfy the so called cone condition (see Chapter III of \cite{BF}) where the cones have the axis of rotation symmetry parallel to the plane $ \{ t = 0 \} $. In this case the convolution \eqref{con:boy} take the form (III.18) of \cite{BF}.

\end{remark}

\subsection{Existence and uniqueness of distributional solutions}

Let us now prove existence of distributional solution in the case where $ v $ is a divergence free vector field in $ L^1_{loc}(\bbR^+;W^{1,q}(\calF(t)))$, $ (\rho^{in}, \rho^+) $ is in $ L^{p}(\calF_0) \times L^{p}_{loc}(\bbR^+\odot \partial \calF^+(t); \lvert g-q \rvert \, dsdt) $, $ f $ is in $\calL_{loc}^1(\bbR+;L^p(\calF(t))) $ and $ p $ and $ q $ are conjugate, i.e. $ 1/p +1/q = 1 $.

\begin{proposition}
\label{Prop:1:Last}
Let $ p \in (1,+\infty] $ and $ 1/p + 1/q = 1 $. Let $ v \in \calL^1_{loc}(\bbR^+;W^{1,q}(\calF(t))) $ such that $\div(v) = 0 $. Let $ \rho^{in} $ and $ \rho^{+} $ respectively in $ L^{p}(\calF(0)) $ and $ L^p_{loc}(\bbR^+ \odot \partial \calF^+(t),  \lvert v\cdot n - q \rvert \, dt ds ) $. Let $ f $ in $\calL_{loc}^1(\bbR^+;L^p(\calF(t)))$. Then there exists a distributional solution to \eqref{DP:theo:tdep}. More precisely there exists $ (\rho, \rho^-) $  in $ \calL^{\infty}_{loc}(\bbR^+;L^p(\calF)) \times \calL^p_{loc}(\bbR^+ \odot \partial \calF^-(t),  \lvert v\cdot n - q \rvert \, dt ds ) $ solution of
\begin{align}
\label{tr:1}
\int_{\calF(0)} \rho_0 \varphi(0,.) \, dx + \int_{\bbR^+}\int_{\calF(t)} \rho (\partial_t \varphi + v\cdot \nabla \varphi) \, dx dt & \, + \int_{\bbR^+}\int_{\calF(t)} f \varphi \, dx dt  \\  = & \, \sum_{i = +,-}\int_{\partial \calF^i} \rho^i \varphi (v\cdot n - q) \, ds dt, \nonumber
\end{align}	
for any $ \varphi \in C^{\infty}_{c}(\bbR^+ \odot \overline{\calF(t)} )$. 
	
\end{proposition}	

Let us also notice that distributional solutions are also unique.

\begin{proposition}
\label{Prop:uniq:distr}
Let $ p \in (1,+\infty] $ and $ 1/p + 1/q = 1 $. Let $ v \in \calL^1_{loc}(\bbR^+;W^{1,q}(\calF(t))) $ such that $\div(v) = 0 $. Let $ \rho^{in} $ and $ \rho^{+} $ respectively in $ L^{p}(\calF(0)) $ and $ L^p_{loc}(\bbR^+ \odot \partial \calF^+(t) , \lvert v\cdot n - q \rvert \, dt ds ) $. Let $ f $ in $\calL_{loc}^1(\bbR^+;L^p(\calF(t)))$.Then distributional solutions  $ (\rho, \rho^- )$  in $ \calL^{\infty}_{loc}(\bbR^+;L^p(\calF)) \times L^p_{loc}(\bbR^+ \odot \partial \calF^-(t) ,  \lvert v\cdot n - q \rvert \, dt ds ) $ of \eqref{tr:1} are unique.
\end{proposition}

In the next subsection we present an heuristic proof of Proposition \ref{Prop:1:Last}.

\subsubsection{Heuristic proof of Proposition \ref{Prop:1:Last}} 

We show existence by a vanishing viscosity method similar to \cite{Boyer}. More precisely we add to the  transport equation \eqref{DP:theo:tdep} a viscous term multiply by a parameter $ \nu $ and we study the limit when $ \nu $ tends to zero.  The idea is to consider the viscous system 
\begin{align}
\partial_t \rho_{\nu} + v_{\nu} \cdot \nabla \rho_{\nu} - \nu \Delta \rho_{\nu} = & \, f_{\nu} \quad && \text{ for } x \in \calF(t), \nonumber \\  
\nu \nabla \rho_{\nu} \cdot n  - (\rho_{\nu}-\rho^+_{\nu})(v_{\nu}-q)\mathds{1}_{\partial \calF^+} = & \, 0 \quad && \text{ for } x \in \partial \calF(t), \label{TE:vis}  \\
\rho_{\nu}(0,.) = & \, \rho^{in}_{\nu} \quad && \text{ for } x \in \calF_0. \nonumber
\end{align}
and to show that  $ \rho_{\nu} $ converges weakly to $ \rho $ a solution of \eqref{tr:1} as $ \nu $ tends to zero. The system \eqref{TE:vis} is parabolic and strong solutions in the class $ L^2(H^2) \cap C^0(H^1) $ can be constructed via a Galerkin method for $ H^1 $ initial datum. These solutions satisfy the a-priori bounds 
\begin{align}
\label{a:priori:est}
\int_{\calF(T)} G(\rho_{\nu}) \, dx + & \int_{0}^T \int_{\partial \calF^{-}(t)} G(\rho_{\nu})(v_{\nu} \cdot n - q) \, ds dt \nonumber \\ & \, +  \int_{0}^T \int_{\calF(t)} \lvert \nabla \rho_{\nu}\rvert^2 G''(\rho_{\nu}) \, dx dt \nonumber \\   \leq & \, \int_{\calF(0)} G(\rho_{\nu}^{in}) \, dx - \int_{0}^T \int_{\partial \calF^{+}(t)} G(\rho_{\nu}^+)(v_{\nu} \cdot n - q) \, ds dt \\ & \, + \int_{0}^{T}\int_{\calF(t)} f_{\nu} G'(\rho_{\nu})  \nonumber
\end{align}
for any smooth positive and convex function $ G $. The weak formulation of \eqref{TE:vis} reads
\begin{align}
\label{TE:vis:weak:for}
\int_{0}^t \int_{\calF(t)} & \rho_{\nu}\left(\partial_t \varphi + v_{\nu} \cdot \nabla \varphi\right) - \nu \nabla \rho_{\nu} \cdot \nabla \varphi \, dx dt + \int_{0}^t \int_{\calF^{-} (t)} \rho_{\nu} \varphi  (v_{\nu} \cdot n - q ) \, ds dt \nonumber  \\ & \,  + \int_{0}^{T}\int_{\calF(t)} f_{\nu} \varphi \, dx dt  \nonumber \\  = & \, \int_{\calF(t)} \rho_{\nu}(t,.) \varphi(t,.) \, dx + \int_{\calF(0)} \rho_{\nu}^{in} \varphi(0,.) \, dx  \\ & \, - \int_{0}^t \int_{\calF^{+}(t)} \rho_{\nu}^+ \varphi  (v_{\nu} \cdot n - q ) \, ds dt  \nonumber 
\end{align}
With the help of \eqref{a:priori:est} for $ G(x) = \lvert x \rvert^p $, it is easy to pass to the limit in \eqref{TE:vis:weak:for} as $ \nu $ tends to $ 0 $ and deduce Proposition \ref{Prop:1:Last}. 

Let start by showing existence and uniqueness of strong solutions to the system \eqref{TE:vis}. Notice that for $ \rho_{\nu} \in L^2(H^2) $, which implies that $ \nabla \rho_{\nu} $ has $ L^2(H^{1/2}) $ trace on the boundary $ [0,T] \odot \partial \calF (t) $. We need then to consider approximate data $ v_{\nu}$ and $ \rho_{\nu}^+ $ that are compatible with the regularity of $ \nabla \rho_{\nu} \cdot n $.   

Moreover equations \eqref{TE:vis} are satisfied in a time dependent subset of $ \bbR^2 $. So let start by considering a change of variables that makes the domain time independent.

\subsubsection{A time independent frame}

Note that the transport equation \eqref{DP:theo:tdep} and its viscous approximation \eqref{TE:vis} are satisfied in the prescribed time-dependent domain $ \bbR^+ \odot \calF(t) $. To tackle existence, we introduce a change of variables that translates the equations in the time independent domain $ [0,T] \times \calF(0) $ for any fixed $ T > 0 $, we show that the two formulations are equivalent and we prove existence by a vanishing viscosity method. 

Let us introduce the change of variables.

\begin{Lemma}
For any $ T >  0 $ there exists a map $ (id, X) : [0,T] \times \calF(0) \to  [0,T] \odot \calF(t) $ such that
	
\begin{itemize}
		
\item it holds  $ (id, X)(t,y) = (t, X(t,y))$ in $ [0,T] \times \calF(0) $, $ X(0,y) = y $ in $\calF(0) $ and $(id,X) = (id,b) $ on $ [0,T] \times \partial \calF(0) $,
		
\item the function $ (id, X) $ is $ C^{1,\alpha}([0,T];C^{2,\alpha}(\calF(0)))$ and admits an inverse 	$ (id, Y) $ with the same regularity,

\item the gradient $ \nabla X $ is  strictly positive in the sense that $ \xi \cdot \nabla X  \xi \geq C \lvert \xi \rvert^2 $ for $ C > 0 $ and $ \xi \in \bbR^2 \setminus \{0\} $,
		
\item for $ N(x) $ and $ n(t,x)$ the normal respectively to $ \partial \calF(0) $ and $  \partial \calF(t) $ exiting from the domain, it holds
		
$$ \nabla X N = n, \quad  (\nabla X)^{T} n = N, \quad \nabla Y n = N \text{ and } \quad (\nabla Y)^{T} N = n $$   	
on respectively $ [0,T] \times \partial \calF(0) $ and $ [0,T] \odot \calF(t) $.

	\end{itemize}

\end{Lemma}

Let now rewrite the equations \eqref{DP:theo:tdep} in the time independent domain $ [0,T] \times \calF(0 ) $ with the help of the change of variables $ X $. Let 
\begin{gather*} \bar{\rho}(t,y) = \rho(t, X(t,y)), \quad \bar{\rho}^{+}(t,y) = \rho+(t,X(t,y)), \quad \bar{\rho}^{in}(y) = \rho^{in}(0,X(0,y)), \\ \bar{V}(t,y) = \partial_t Y(t,X(t,y)), \quad  \bar{v}(t,y) = \nabla Y(t,X(t,y))v(t,X(t,y)) \\ \quad \text{ and } \quad \bar{f}(t,y) = f(t,X(t,y)).
\end{gather*}
The equations \eqref{DP:theo:tdep} rewrite 
\begin{align}
\partial_ t \bar{\rho} + (\bar{V} + \bar{v}) \cdot \nabla \bar{\rho} = & \, \bar{f} \quad && \text{ for } x \in \calF(0), \nonumber \\
\bar{\rho} =  & \, \bar{\rho}^{+} \quad && \text{ for } x \in \partial \calF^{+}(0), \label{DP:theo:tind}\\
\bar{\rho}(0,.) = & \, \bar{\rho}^{in} \quad && \text{ for } \calF(0).   \nonumber
\end{align} 
We show existence of solutions for the system \eqref{DP:theo:tind}. In the following, let use the notation $ \bar{q}(t,y) = \bar{V}(t,y) \cdot N(y) = q(t,X(t,y)) $.
\begin{proposition}
	\label{Prop:2:Last}
	Let $ p \in (1,+\infty] $ and $ 1/p + 1/q = 1 $. Let $ \bar{v} \in L^1_{loc}(\bbR^+;W^{1,q}(\calF(0))) $ such that $\div_x(\nabla X(t, Y(t,x))\bar{v}(t,Y(t,x))) = 0 $. Let $ \bar{\rho}^{in} $ and $ \bar{\rho}^{+} $ respectively in $ L^{p}(\calF(0)) $ and $ L^p_{loc}(\partial \calF^+; \, \lvert \bar{v}\cdot N - \bar{q} \rvert dt ds ) $. Let $ \bar{f} $ in $L^1_{loc}(\bbR^+;L^p(\calF(0)))$Then there exists a distributional solution to \eqref{DP:theo:tind}. More precisely there exists $ (\bar{\rho}, \bar{\rho}^-) $  in $ L^{\infty}_{loc}(\bbR^+;L^p(\calF(0))) \times L^p_{loc}(\partial \calF^-; \, \lvert \bar{v}\cdot N - \bar{q} \rvert dt ds ) $ solution of
	\begin{align}
	\label{tr:2}
	\int_{\calF(0)} \bar{\rho}^{in} \varphi(0,.) \, dx + \int_{\bbR^+}\int_{\calF(0)} \bar{\rho} (\partial_t \varphi + & \div((\bar{V}+\bar{v}) \varphi)) \, dx dt + \int_{\bbR^+}\int_{\calF(0)} \bar{f} \varphi \, dx dt \\ & = \sum_{i = +,-}\int_{\partial \calF^i} \bar{\rho}^i \varphi (\bar{v}\cdot n - \bar{q}) \, ds dt, \nonumber  
	\end{align}	
	for any $ \varphi \in C^{\infty}_{c}(\bbR^+ \times \overline{\calF(0)} )$. 
	
\end{proposition}	
Let us notice that Proposition \ref{Prop:1:Last} and \ref{Prop:2:Last} are equivalent in the following sense.

\begin{proposition}
	\label{Prop:equivalence}
	The couple 	$ (\rho, \rho^-) $ is a $ L^{\infty}_{loc}(\bbR^+;L^p(\calF)) \times \calL^p_{loc}(\bbR^+; \partial \calF^-(t); \, \lvert v\cdot n - q \rvert dtds ) $ solution of \eqref{tr:1} if and only if  $ (\bar{\rho},\bar{\rho}^-) = (\rho(t,X(t,y)), \rho^{-}(t,X(t,y)))$
	is a $ L^{\infty}_{loc}(\bbR^+;L^p(\calF(0))) \times L^p_{loc}(\partial \calF^- ; \, \lvert \bar{v}\cdot N - \bar{q} \rvert dt ds ) $ solution of \eqref{tr:2}.
\end{proposition}

\begin{proof}
	
	For $ \psi \in C^{\infty}_{c}(\bbR^+ \times \overline{\calF(0)}) $, use $ \varphi(t,x) = \psi(t, Y(t,x)) \det(\nabla Y(t,x)) $ as test function in \eqref{tr:1}. After some computations the equivalence follows easily. Let us only recall that for a smooth square matrix $ A(t,x) $, it holds
	$$ \partial_z \det{A} = \tr ( \text{adj}(A)\partial_z A ) \quad \text{ for } z \in \{t, x_1,\dots, x_d\}.$$ 
	where Adj$(A)$ is the adjugate matrix of $ A $. 	
\end{proof}

To show Proposition \ref{Prop:2:Last} we use a vanishing viscosity method. Let consider the following viscous approximation of the equation \eqref{DP:theo:tind}. Let denote by $ \nu > 0 $ the viscous parameter, then 

\begin{align}
\partial_ t \bar{\rho}_{\nu} + (\bar{V} + \bar{v}_{\nu}) \cdot \nabla \bar{\rho}_{\nu} - \nu \Delta Y\cdot \nabla \bar{\rho}_{\nu} - \nu \nabla Y (\nabla Y)^T: \nabla^2 \bar{\rho}_{\nu}= & \, \bar{f}_{\nu} \quad && x \in \calF(0), \nonumber \\
\nu ((\nabla Y)^T N) \cdot ((\nabla Y)^T \nabla \bar{\rho}_{\nu}) -  (\bar{\rho}_{\nu}-\bar{\rho}^{+}_{\nu})(\bar{v}_{\nu} \cdot N - \bar{q})\mathds{1}_{\calF^+}   = & \, 0  \quad && x \in \partial \calF^{+}, \label{DP:theo:tind:vis}\\
\bar{\rho}(0,.) = & \, \bar{\rho}^{in} \quad && x \in  \calF(0).   \nonumber
\end{align} 

\begin{remark} First of all let us notice that the equation \eqref{DP:theo:tind:vis} is of parabolic type, in fact $ \nabla Y (\nabla Y)^T > c \Id  $ for some small $ c > 0 $.	
\end{remark}

\begin{remark} The viscous approximation \eqref{DP:theo:tind:vis} in the variables $ \rho_{\nu}(t,x) = \bar{\rho}_{\nu}(t, Y(t,x)) $ is the system \eqref{TE:vis}.	
\end{remark}

\subsubsection{Existence of strong solutions for the viscous system}

Let show existence of weak solutions for the system \eqref{DP:theo:tind:vis}.

\begin{proposition}
	\label{exi:vis:sys}
	Let $  \nu > 0 $. Let $ \bar{v}_{\nu} \in L^2([0,T];H^1(\calF(0))) \cap L^2([0,T];L^{\infty}(\calF(0))) $ such that $\div_x(\nabla X(t, Y(t,x)) \bar{v}_{\nu}(t,Y(t,x))) = 0 $ and $ \bar{v} \cdot N \in C^{1}([0,T] \times \partial \calF(0))$.  Let $ \bar{\rho}_{\nu} \in H^{1}(\calF(0)) $ and $ \rho_{\nu}^+ \in  C^1([0,T] \times \partial \calF(0)) $. Let $ \bar{f} \in C^0([0,T]\times \calF(0))$. Then there exists a strong solution $ \bar{\rho}_{\nu} \in W^{1,2}([0,T];L^2(\calF(0))) \cap L^2(0,T;H^2(\calF(0)))$ of the system \eqref{DP:theo:tind:vis}.

	Moreover if $ p \in (1,+\infty) $, it holds 
	\begin{gather}
	\left(\int_{\calF(0)} \lvert \bar{\rho}_{\nu}\rvert^p \det(\nabla X(t,y)) dy 
	+  \int_{0}^{T} \int_{\partial \calF^-}  \lvert \bar{\rho}_{\nu}^-\rvert^p (\bar{v}\cdot n- \bar{q}) ds dt \right)^{1/p} \nonumber \\ 
 \quad \quad \quad \quad \leq \left( \int_{\calF(0)} \lvert \bar{\rho}^{in}_{\nu}\rvert^p \det(\nabla X(0,y)) dy -  \int_0^T \int_{\partial \calF^+}  \lvert \bar{\rho}_{\nu}^+\rvert^p (\bar{v}\cdot n- \bar{q})  \,  ds dt \right)^{1/p} \nonumber \\  \quad\quad\quad\quad\quad\quad\quad\quad 
+ \int_0^T \left( \int_{\calF(0)} \lvert \bar{f}_{\nu}\rvert^p \det(\nabla X(t,y)) \, dy   \right)^{1/p} dt   \label{est:conv}\\ 
\text{ and } \quad \sqrt{\nu} \left(\int_{0}^{T}\int_{\calF(0)} \lvert \nabla \bar{\rho}_{\nu} \rvert^2  \, dy dt \right)^{1/2} \leq \nonumber \\ \left(\int_{\calF(0)}\lvert \bar{\rho}^{in}_{\nu} \rvert^2 \, dy -  \int_0^T \int_{\partial \calF^+_t}  \lvert \bar{\rho}_{\nu}^+\rvert^2 (\bar{v}\cdot n- \bar{q}) ds dt \right)^{1/2} +   \int_0^T \left( \int_{\calF(0)} \lvert \bar{f}_{\nu}\rvert^2 \, dy   \right)^{1/2} dt \nonumber
	\end{gather}

\end{proposition} 

\begin{proof}
	
The proof of this result follows from a Galerkin method. Let us show only the a-priori estimates. Without loss of generality let us fixed $ \nu = 1 $ and let us write $ \bar{\rho} $ instead of $ \bar{\rho}_1 $ for simplicity.
If we multiply \eqref{DP:theo:tind:vis} by $ \bar{\rho} $ and we integrate in $ \calF(0) $, we deduce that
\begin{align*}
\partial_{t} \int_{\calF(0)}\frac{\bar{\rho}^2}{2} &  + \int_{\partial \calF^-_t} (\bar{V}+\bar{v})\cdot N \frac{\bar{\rho}^2}{2} +  \int_{\calF(0)} \lvert \nabla Y \nabla \bar{\rho}\rvert^2 = \\
 & - \int_{\partial \calF^+_t}(\bar{V}+\bar{v})\cdot N \bar{\rho}\left(\frac{\bar{\rho}}{2}- (\bar{\rho}-\bar{\rho}^+) \right) + \int_{\calF(0)} \div( \bar{V} +\bar{v}) \frac{\bar{\rho}^2}{2} \\
& \,   +  \int_{\calF(0)} \Delta Y \cdot \nabla \bar{\rho} \bar{\rho} - \int_{\calF(0)} \div \left(\nabla Y (\nabla Y)^{T} \right)\cdot \nabla \bar{\rho} \bar{\rho} - \int_{\calF(0)} \bar{f} \bar{\rho} \\
\leq & \, - \int_{\partial \calF^-_t} (\bar{V}+\bar{v})\cdot N \frac{\lvert \bar{\rho}^+\rvert^2}{2} + \| \div( \bar{V} +\bar{v}) \|_{L^{\infty}}\|\bar{\rho}\|_{L^2}^2 + \eps  \| \nabla \bar{\rho} \|_{L^2} \\
& \, C_{\eps} \left(\|\Delta Y\|_{L^{\infty}} \|\bar{\rho}\|_{L^2}^2 + \|\div \left(\nabla Y (\nabla Y)^{T} \right) \|_{L^{\infty}} \|\bar{\rho}\|_{L^2}^2 \right) + \|\bar{f}\|_{L^2}^2 +\|\bar{\rho}\|_{L^2}^2. 
\end{align*}
We absorb the term $ \|\nabla \bar{\rho}\|_{L^2} $ in the left hand side by choosing $ \eps $  small enough and the a-priori estimates follows from Gr\"omwall's Lemma, in particular we show that $ \bar{\rho} $ is a-priori bounded in $ L^{\infty}(0,T;L^2(\calF(0))) \cap L^2(0,T;H^1(\calF(0))) $ and that the second inequality of \eqref{est:conv} holds true.

To show the $ H^{1}(0,T;L^2(\calF(0))) \cap L^{\infty}(0,T;H^1(\calF(0))) $ a-priori bound, let us multiply \eqref{DP:theo:tind:vis} by $ \partial_{t} \bar{\rho} $ and integrate in $ \calF(0) $. We deduce 
\begin{align*}
\int_{\calF(0)} \lvert \partial_t \bar{\rho} \rvert^2 + \int_{\calF(0)} \nabla Y (\nabla Y)^T\nabla \bar{\rho} \cdot \partial_t \nabla \bar{\rho} & \, - \int_{\partial \calF^+_t} (\bar{\rho}- \bar{\rho}^+) \partial_t \bar{\rho} (\bar{V}+\bar{v})\cdot N = \\ - \int_{\calF(0)} \bar{f}  \partial_t \bar{\rho}  - \int_{\calF(0)}(\bar{V}-\bar{v}) \cdot \bar{\rho} \partial_t \bar{\rho}  & \, + \int_{\calF(0)} \Delta Y \bar{\rho} \partial_t \bar{\rho} \\ & \, - \int_{\calF(0)} \div \left(\nabla Y (\nabla Y)^{T} \right) \cdot \nabla \bar{\rho} \partial_t \bar{\rho}.   
\end{align*} 
We notice that 
\begin{align*}
\int_{\calF(0)} \nabla Y (\nabla Y)^T\nabla \bar{\rho} \cdot \partial_t \nabla \bar{\rho} = & \,\partial_{t} \int_{\calF(0)} \frac{\lvert \nabla Y \nabla \bar{\rho} \rvert^2}{2} \\ & \, - \frac{1}{2}\int_{\calF(0)} \partial_t \left(\nabla Y (\nabla Y)^T \right): \nabla \bar{\rho} \otimes \nabla \bar{\rho} 
\end{align*}
and 
\begin{align*}
- \int_{\partial \calF^+_t} & (\bar{\rho}- \bar{\rho}^+) \partial_t \bar{\rho} (\bar{V}+\bar{v})\cdot N = \\ & \,  - \int_{\partial \calF^+_t} (\bar{\rho}- \bar{\rho}^+) \partial_t (\bar{\rho}- \bar{\rho}^+ )(\bar{V}+\bar{v})\cdot N - \int_{\partial \calF^+_t} (\bar{\rho}- \bar{\rho}^+) \partial_t \bar{\rho}^+ (\bar{V}+\bar{v})\cdot N  \\
= & \,- \partial_t  \int_{\partial \calF^+_t} \frac{(\bar{\rho}- \bar{\rho}^+)^2}{2} (\bar{V}+\bar{v})\cdot N  + \int_{\partial \calF^+_t} \frac{(\bar{\rho}- \bar{\rho}^+)^2}{2} \partial_t (\bar{V}+\bar{v})\cdot N 
\\ & \, - \int_{\partial \calF^+_t} (\bar{\rho} - \bar{\rho}^+) \partial_t \bar{\rho}^+ (\bar{V}+\bar{v})\cdot N,
\end{align*}
where we use that $ (\bar{V} + \bar{v})\cdot N \mathds{1}_{\calF_t^+} $ is a $ W^{1,\infty} $ function.

Putting all together we deduce that
\begin{align*}
\int_{\calF(0)} & \lvert \partial_t \bar{\rho}\lvert^2 + \partial_{t} \int_{\calF(0)} \frac{\lvert \nabla Y \nabla \bar{\rho} \rvert^2}{2} - \partial_t  \int_{\partial \calF^+_t} \frac{(\bar{\rho}- \bar{\rho}^+)^2}{2} (\bar{V}+\bar{v})\cdot N  \\
&  =  \, - \int_{\calF(0)} \bar{f}  \partial_t \bar{\rho}  - \int_{\calF(0)}(\bar{V}-\bar{v}) \cdot \bar{\rho} \partial_t \bar{\rho}  + \int_{\calF(0)} \Delta Y \bar{\rho} \partial_t \bar{\rho}
\\ & \, - \int_{\calF(0)} \div \left(\nabla Y (\nabla Y)^{T} \right) \cdot \nabla \bar{\rho} \partial_t \bar{\rho} + \frac{1}{2}\int_{\calF(0)} \partial_t \left(\nabla Y (\nabla Y)^T \right): \nabla \bar{\rho} \otimes \nabla \bar{\rho} 
\\ & \, - \int_{\partial \calF^+_t} \frac{(\bar{\rho}- \bar{\rho}^+)^2}{2} \partial_t (\bar{V}+\bar{v})\cdot N 
+ \int_{\partial \calF^+_t} (\bar{\rho}
- \bar{\rho}^+) \partial_t \bar{\rho}^+ (\bar{V}+\bar{v})\cdot N \\
\leq & \, \eps\|\partial_t \bar{\rho} \|_{L^2}^2 + C_{\eps} \left(\|\bar{f}\|_{L^2}^2+ \|(\bar{V}+\bar{v})\|_{L^{\infty}}\|\bar{\rho}\|_{L^2}^2 + \|\Delta Y\|_{L^{\infty}}\|\bar{\rho}\|_{L^2}^2\right) \\
& \, + C_{\eps} \left(\| \div \left(\nabla Y (\nabla Y)^{T} \right)\|_{L^{\infty}}\|\nabla \bar{\rho}\|_{L^2}^2  \right)
 +  \| \partial_t \left(\nabla Y (\nabla Y)^{T}\right) \|_{L^{\infty}}\|\nabla \bar{\rho}\|_{L^2}^2 \\
& \, + \|\bar{\rho}- \bar{\rho}^{+}\|_{H^1}^2\|\partial_t (\bar{V}+\bar{v})\cdot N \|_{L^{\infty}} +  \|\bar{\rho}- \bar{\rho}^{+}\|_{H^1} \|\partial_t \bar{\rho}^+\|_{L^2} \|(\bar{V}+\bar{v})\cdot N \|_{L^{\infty}}.
\end{align*} 
We absorb the term $ \|\partial_t \bar{\rho}\|_{L^2} $ in the left hand side by choosing $ \eps $ small enough and the a-priori estimates follows from Gr\"omwall's Lemma, in particular we show that $ \bar{\rho} $ is a-priori bounded in $ H^{1}(0,T;L^2(\calF(0))) \cap L^{\infty}(0,T;H^1(\calF(0))) $. Finally to show the $ L^2(0,T;H^2(\calF(0))) $ a-priori bound is enough to use the equation.

It remains only to show the first inequality of \eqref{est:conv}. To do that we multiply \eqref{DP:theo:tind:vis} by $ p \bar{\rho}_{\nu} \lvert \bar{\rho}_{\nu}\rvert^{p-2} \det(\nabla X) $ and integrate by parts. The estimate then follows from the Bihari-LaSalle inequality which is a generalization of the Gr\"omwall's lemma. 
	
\end{proof}

\subsubsection{Existence of approximate data}

In this subsection we show that there exists data $ \bar{v}_{\nu} $, $ \bar{\rho}_{\nu} $ and $ \bar{\rho}_{\nu}^+ $ that satisfy the hypothesis of Proposition \ref{exi:vis:sys} that approximate $ \bar{v} $, $ \bar{\rho}^{in} $ and $ \bar{\rho}^+ $ that satisfy the hypothesis of Proposition \ref{Prop:2:Last} in some strong enough norm.

\begin{Lemma}
\label{approx:lem}
Let $ p \in (1,+\infty] $ and $ q $ such that $ 1/p +1/q = 1 $. Let $ \bar{v} \in L^1_{loc}(\bbR^+;W^{1,q}(\calF(0))) $ such that $\div_x(\bar{v}(t,Y(t,x))) = 0 $. Let $ \bar{\rho}^{in} $ and $ \bar{\rho}^{+} $ respectively in $ L^{p}(\calF(0)) $ and $ L^p_{loc}(\partial \calF^+; \,  \lvert \bar{v}\cdot N - \bar{q} \lvert \, dt ds ) $. Let $ \bar{f} $ in $ L^1_{loc}(\bbR^+;L^p(\calF(0))) $. Then there exist  $ \bar{v}_{\nu} $, $ \bar{\rho}^{in}_{\nu} $, $ \bar{\rho}_{\nu}^+ $ and $ \bar{f}_{\nu} $ such that $ \bar{v}_{\nu} \in L^2([0,T];H^1(\calF(0))) \cap L^2([0,T];L^{\infty}(\calF(0))) $ with $\div_x(\nabla X(t, Y(t,x))\bar{v}_{\nu}(t,Y(t,x))) = 0 $ and with $ \bar{v} \cdot N \in C^{1}([0,T] \times \partial \calF(0))$. Such that $ \bar{\rho}^{in}_{\nu} \in H^{1}(\calF(0)) $ and such that $ \rho_{\nu}^+ \in  C^1([0,T] \times \partial \calF(0)) $.  Such that $ \bar{f}_{\nu} \in C^0([0,T]\times \calF(0)) $. Moreover
\begin{align*}
\bar{v}_{\nu} \longrightarrow & \, \bar{v}  \quad && \text{ in } L^1(0,T; W^{1,q}(\calF(0))), 
\end{align*}
for $ p < + \infty $  
\begin{align}
\bar{\rho}^{in}_{\nu} \longrightarrow & \, \bar{\rho}^{in} \quad && \text{ in } L^p(\calF(0)), \nonumber  \\
\bar{\rho}^{+}_{\nu}(\bar{v}_{\nu}\cdot N - \bar{q} )^{1/p} \mathds{1}_{\partial \calF^+_{\nu}}  \longrightarrow  & \, \bar{\rho}^{+}(\bar{v} \cdot N - \bar{q} )^{1/p} \mathds{1}_{\partial \calF^+} \quad && \text{ in } L^{p}((0,T)\times \partial \calF(0)),\label{conv:ieva}   \\
\bar{f}_{\nu} \longrightarrow & \, \bar{f} \quad && \text{ in } L^1(0,T;L^p(\calF(0))),  \nonumber . 
\end{align}
And  
\begin{gather}
\label{grad:conv:ieva}
\int_{\calF(0)}\lvert \bar{\rho}^{in}_{\nu} \rvert^2 \, dy -  \iint_{\partial \calF^+_{\nu}}  \lvert \bar{\rho}_{\nu}^+\rvert^2 (\bar{v}_{\nu}\cdot n- \bar{q}) ds dt \leq \nonumber \\  \frac{1}{\sqrt{\nu}} \left( \int_{\calF(0)}\lvert \bar{\rho}^{in}_{\nu} \rvert^p \, dy -   \iint_{\partial \calF^+}  \lvert \bar{\rho}_{\nu}^+ \rvert^p (\bar{v}_{\nu}\cdot n- \bar{q}) ds dt \right) \\
\text{ and } \quad  \left( \int_0^T \left(\int_{\calF(0)}\lvert \bar{f}_{\nu}\rvert^2  \, dy  \right)^{1/2} dt  \right)^2 \leq \frac{1}{\sqrt{\nu}}   \left( \int_0^T \left(\int_{\calF(0)}\lvert \bar{f}_{\nu}\rvert^p  \, dy \right)^{1/p} dt  \right)^p .                             \nonumber
\end{gather} 
For $ p = + \infty $ the convergences \eqref{conv:ieva} hold in any $ L^s $ with $ s < +\infty $ with exponent $ 1/s $ on $ \bar{v}_{\nu} \cdot N - \bar{q} $ and $ \bar{v} \cdot N - \bar{q} $.

\end{Lemma}

\begin{proof}

Consider the case $ p < \infty $. It is not difficult do define functions $ \bar{v}_n $ that converge to $ \bar{v} $ in $ L^1(0,T; W^{1,q}(\calF(0))) $ as $ n $ goes to $ + \infty $ and that satisfy the required properties and regularities. 

For $ M > 0 $, let denote by $ f \wedge M ( x ) = f(x) $ if $ \lvert f(x)\rvert \leq M  $ and $ f \wedge M ( x ) = f(x) $ if $ \lvert(x)\rvert > M  $. Let consider the following inequality
\begin{align*} 
\|\bar{\rho}^+ (\bar{v} \cdot N & - \bar{q} )^{1/p}   \mathds{1}_{\partial \calF^+} - \bar{\rho}_{\nu}^+ (\bar{v}_{n} \cdot N - \bar{q} )^{1/p} \mathds{1}_{\partial \calF^+}\|_{L^p} \\ 
 \leq & \, \|\bar{\rho}^+ (\bar{v} \cdot N - \bar{q} )^{1/p} \mathds{1}_{\partial \calF^+} - (\bar{\rho}^+ \wedge M_{\nu}) (\bar{v} \cdot N - \bar{q} )^{1/p} \mathds{1}_{\partial \calF^+} \|_{L^p}  \\
 & \, + \|(\bar{\rho}^+ \wedge M_{\nu}) (\bar{v} \cdot N - \bar{q} )^{1/p} \mathds{1}_{\partial \calF^+}- (\bar{\rho}^+ \wedge M_{\nu}) (\bar{v}_n \cdot N - \bar{q} )^{1/p} \mathds{1}_{\partial \calF^+_n} \|_{L^p}
 \\
& \,  + \| (\bar{\rho}^+ \wedge M_{\nu}) (\bar{v}_n \cdot N - \bar{q} )^{1/p} \mathds{1}_{\partial \calF^+_n} - \bar{\rho}_{\nu}^+ (\bar{v}_n \cdot N - \bar{q} )^{1/p} \mathds{1}_{\partial \calF^+_n}\|_{L^p}
\end{align*}
By dominate convergence, we fix $ M_{\nu} $ such that the first term on the right hand side is less or equal to $ \nu/3 $. Then we denote by $ v_{\nu} = v_{n_{\nu}}$ such that the second term is less the  $ \nu/3 $. Finally $ \rho^+ \wedge M_{\nu} $ is $ L^p $ so by density of smooth functions we choose one such that the last term of the right hand side is smaller then $ \nu/3 $. With this choices the second convergence of \eqref{conv:ieva} holds true and the approximation is smooth. Regarding $ \bar{\rho}^{in}_{\nu} $ and $ \bar{f}_{\nu} $ we use the density of smooth functions in $ L^p $. Finally \eqref{grad:conv:ieva} holds true after relabelling the sequences.

\end{proof}

\subsubsection{Proof of Proposition \ref{Prop:1:Last}}

We are now able to prove Proposition \ref{Prop:1:Last}. 

\begin{proof}[Proof of Proposition \ref{Prop:1:Last} and \ref{Prop:2:Last}]
	
Let us recall that Proposition \ref{Prop:equivalence} shows that Proposition \ref{Prop:1:Last} and \ref{Prop:2:Last} are equivalent so let us prove the second one.

Let us consider the case $ p < + \infty $ and let $ \bar{v} $, $ \bar{\rho}^{in} $ and $ \bar{\rho}^+ $ the given vector field and data. Lemma \ref{approx:lem} ensures that there exist vector fields $ v_{\nu} $ and data $ \bar{\rho}^{in}_{\nu} $, $ \bar{\rho}^+_{\nu} $ that satisfy the hypothesis of Proposition \ref{exi:vis:sys}, in particular there exist solutions $ \bar{\rho}_{\nu} $ that satisfy the a-priori bounds \eqref{est:conv}. After noticing that $ \det(\nabla X) \geq c > 0  $ in $ [0,T] \times \calF(0) $ and using \eqref{est:conv} and \eqref{conv:ieva}, we deduce that up to subsequence 
\begin{gather*}
\bar{\rho}_{\nu} \cvwstar \bar{\rho} \quad \text{ in } L^{\infty}(0,T;L^p(\calF(t))) \quad \text{ and } \quad \\ \bar{\rho}^-_{\nu}(\bar{v}_{\nu}\cdot N - \bar{q})^{1/p}\mathds{1}_{\partial \calF^-_{\nu}} \cv \bar{\rho}^-(\bar{v} \cdot N - \bar{q})^{1/p}\mathds{1}_{\partial \calF^-} \quad \text{ in } L^{p}([0,T]\time \partial \calF(0)).
\end{gather*}
The functions $ \bar{\rho}_{\nu} $ satisfy the equation \eqref{DP:theo:tind:vis}, in particular they satisfy the weak formulation 
\begin{align*}
\int_{\calF(0)} \bar{\rho}^{in}_{\nu} & \varphi(0,.) \, dy +  \int_{\bbR^+}\int_{\calF(0)} \bar{\rho}_{\nu} (\partial_t \varphi  + \div((\bar{V}_{\nu}+\bar{v}_{\nu}) \varphi)) \, dy dt \\ & \, + \nu \int_{\bbR+} \int_{\calF(0)} (\Delta Y - \div(\nabla Y (\nabla Y)^T))\cdot \nabla \bar{\rho}_{\nu} \varphi \, dy dt \\ &  -\nu \int_{\bbR+} \int_{\calF(0)} \nabla Y (\nabla Y)^T : \nabla \bar{\rho}_{\nu} \otimes \nabla \varphi \, dy dt  \\ & + \int_{\bbR+} \int_{\calF(0)} \bar{f}_{\nu} \varphi \, dy dt= \sum_{i = +,-}\int_{\partial \calF^i_{\nu}} \bar{\rho}^i_{\nu} \varphi (\bar{v}_{\nu} \cdot n - \bar{q}) \, ds dt. 
\end{align*}
Passing to the limit in $ \nu $ we deduce that $ (\bar{\rho}, \bar{\rho}^-) $  satisfies \eqref{tr:2} in fact the terms in which appears $ \nabla \bar{\rho}_{\nu} $ converges to zero due to the second bound of \eqref{est:conv} together with \eqref{grad:conv:ieva}.

The case $ p = + \infty $ can be treated in a similar way. It is enough to consider $ s = n/(n-1) $ in the Lemma \ref{approx:lem} and to notice that $ \div_y(\bar{v}_{\nu}) = \sum_{i,k} \partial_i \partial_k X_i \bar{v}_{\nu,k} $ from the hypothesis $\div_x(\nabla X(t, Y(t,x))\bar{v}_{\nu}(t,Y(t,x))) = 0 $, in particular no derivatives appear on $ \bar{v}_{\nu} $ in the expression $\div_y(\bar{v}_{\nu})$.
 
\end{proof}

\subsubsection{Proof of Proposition \ref{Prop:uniq:distr}}

Let us now prove uniqueness of distributional solutions to the transport equation.

\begin{proof}[Proof of Proposition \ref{Prop:uniq:distr}]

Let $ (\rho, \rho^-) $ a distributional solution to the transport equation in the sense of \eqref{tr:1}. Then by Theorem \ref{theo:dist:equiv:ren} there exists a trace $ \gamma \rho $ on $ \bbR^+ \odot \partial \calF(t) $ such that $ (\rho, \gamma \rho \mathds{1}_{\partial \calF^-} ) $ is a renormalized solution to the transport with entering data $ \gamma \rho \mathds{1}_{\partial \calF^+} $.  

Let now show that $ \rho^i = \gamma \rho \mathds{1}_{\partial \calF^i} $ for $ i = +,- $. This will implies uniqueness because renormalized solutions to the transport are unique. 

Consider a smooth positive function $ \beta $ such that $ \beta(s) = 1/|s| $ for $ |s| \geq 1 $ and $ \bar{\beta}(s) = s \beta(s) $. The function $ (\beta(\rho), \beta(\gamma \rho)\mathds{1}_{\partial \calF^-}) $ is a distributional solution of the transport equation with entering data $ \beta(\gamma \rho)\mathds{1}_{\partial \calF^+} $ due to Theorem \ref{theo:dist:equiv:ren}. If we define $ \beta(\rho)_{\eps} = \eta_{\eps} \star_{\nu} \beta(\rho) $, we have
\begin{equation*}
\partial_t \beta(\rho)_{\eps} + v \cdot \nabla \beta(\rho)_{\eps} = R_{\eps} -(f\beta'(\rho))_{\eps} 
\end{equation*} 
with $ R_{\eps} $ converging to zero in $ \calL^1(0,T; L^q(\calF(t))) $. Moreover $ \beta(\gamma \rho) = \gamma \beta(\rho) $.

Test \eqref{tr:1} with the admissible test function $ \psi = \beta(\rho)_{\eps} \varphi $ and after letting $ \eps $ to zero we deduce
\begin{align*}
 \int_{\calF(0)} \rho^{in} \beta(\rho^{in})& \varphi(0,.) \, dx +   \int_{\bbR^+}\int_{\calF(t)} \rho \beta{\rho} (\partial_t \varphi + v\cdot \nabla \varphi) \, dx dt = \\ & \, \sum_{i = +,-}\int_{\partial \calF^i} \rho^i\beta(\gamma \rho) \varphi (v\cdot n - q) \, ds dt + \int_{\bbR^+}\int_{\calF(t)} f\beta(\rho)  \varphi \, dx dt \\ & \,  + \int_{\bbR^+}\int_{\calF(t)} f\beta'(\rho)\rho  \varphi \, dx dt.
\end{align*}
Note that also $ \bar{\beta} $ is smooth bounded with bounded derivative. So if we use it in \eqref{ren:equ} we deduce
\begin{align*}
\int_{\calF(0)} &  \rho^{in} \beta(\rho^{in}) \varphi(0,.) \, dx +  \int_{\bbR^+}\int_{\calF(t)} \rho \beta{\rho} (\partial_t \varphi + v\cdot \nabla \varphi) \, dx dt = \\ &  \, \int_{\bbR+} \int_{\partial \calF} \gamma\rho \beta(\gamma \rho) \varphi (v\cdot n - q) \, ds dt  + \int_{\bbR^+}\int_{\calF(t)} f(\beta(\rho)+\beta'(\rho)\rho)  \varphi \, dx dt.
\end{align*}
We deduce that 
\begin{equation*}
 \sum_{i = +,-}\int_{\partial \calF^i} \rho^i\beta(\gamma \rho) \varphi (v\cdot n - q) \, ds dt  = \int_{\bbR+} \int_{\partial \calF} \gamma\rho \beta(\gamma \rho) \varphi (v\cdot n - q) \, ds dt
\end{equation*}
for any $ \varphi $, which implies
\begin{equation*}
\sum_{i = +,-} \rho^i\beta(\gamma \rho) (v\cdot n - q)  = \gamma\rho \beta(\gamma \rho) (v\cdot n - q),
\end{equation*}
Recall that $ \beta > 0 $, we conclude  
\begin{equation*}
\sum_{i = +,-} \rho^i \mathds{1}_{\partial \calF^{i}} = \gamma \rho \quad  \quad \quad \, (v\cdot n - q) \, ds dt  \text{-almost everywhere.}
\end{equation*}
\end{proof}

\subsection{Final step of the proof of Theorem \ref{exi:transport:ppppppppp} and Theorem \ref{duality:formula}} 

We have all the ingredients to prove Theorem \ref{exi:transport:ppppppppp}. 

\begin{proof}[Proof of Theorem \ref{exi:transport:ppppppppp}]
	
Let $ \beta \in C^1_b(\bbR) $ a strictly increasing function, in particular it is invertible in its image. It holds that $ \beta(\rho^{in}) $ and $ \beta(\rho^{+}) $ are respectively in $ L^{\infty}(\calF(0)) $ and $ L^{\infty}((0,T)\odot \calF(t) )$.  From Proposition \ref{Prop:1:Last} and \ref{Prop:uniq:distr} there exists a unique distributional solution $ (\bar{\rho}, \bar{\rho}^-)$ to the transport equation associated with the vector field $ v $ and the data $ \beta(\rho^{in}) $ and $ \beta(\rho^{+}) $. Moreover it is the renormalized solution to the transport by Theorem \ref{theo:dist:equiv:ren}. Notice that $ (\bar{\rho}, \bar{\rho}^+ ) $ is in the image of $ \beta $ from the $ L^{\infty } $ a-priori bounds. The couple $ (\rho, \rho^- ) = (\beta^{-1}{\rho}, \beta^{-1}(\rho^-)) $ is then the desired solution.
	
\end{proof}

Let now conclude with the proof of Theorem \ref{duality:formula}

\begin{proof}[Proof of Theorem \ref{duality:formula}]
	
For solutions regular enough the proof is an integration by parts. For $ \beta \in C^1_b(\bbR) $ and $ \eps > 0 $ small we consider $ \tilde{\rho}_{\eps} = \eta_{\eps} \star_{\nu} \beta(\rho) $ and $ \tilde{r}_{\eps} = \eta_{\eps} \star_{\nu} \beta(r) $. From Lemma \ref{Lem:5}, we have that  $ \tilde{\rho}_{\eps} $ and $ \tilde{r}_{\eps} $ 
satisfy the equations
\begin{equation*}
\partial_t \tilde{\rho}_{\eps} + v \cdot \tilde{\rho}_{\eps} = R_{\eps}  \quad  \text{ and } \quad 
\partial_t \tilde{r}_{\eps} + v \cdot \tilde{r}_{\eps} -\tilde{f}_{\eps} = S_{\eps}, 
\end{equation*}
with $ R_{\eps}, S_{\eps} $ converging to zero in $ \calL^1(0,T;L^q(\calF(t))) $. 
Multiply the equation satisfied by $ \tilde{\rho}_{\eps} $ by $ \tilde{r}_{\eps} $ and integrate in $ [0,T]\odot \calF(t) $. After some integrations by parts we deduce
\begin{align*}
\int_{\calF(T)} \tilde{\rho}_{\eps}(T,.) \tilde{r}_{\eps}(T,.) \, dx & \, -  \int_{\calF(0)} \tilde{\rho}_{\eps}(0,.) \tilde{r}_{\eps}(0,.) \, dx   + \sum_{i = +,-} \tilde{\rho}_{\eps}^i \tilde{r}_{\eps}^i(v \cdot n - q ) \, ds dt  \\ = &\, \int_0^T \int_{\calF(t)} \tilde{f}_{\eps} \tilde{\rho}_{\eps} + \int_0^T \int_{\calF(t)} S_{\eps} \tilde{\rho}_{\eps} + \int_0^T \int_{\calF(t)} \tilde{r}_{\eps} R_{\eps}.
\end{align*}
Passing to the limit in $ \eps $ we deduce 
\begin{align*}
\int_{\calF(T)} \beta(\rho)(T,.)&  \beta(r)(T,.) \, dx  -  \int_{\calF(0)} \beta(\rho)(0,.) \beta(r)(0,.) \, dx  \\ & \,  + \sum_{i = +,-} \beta(\rho)^i \beta(r)^i(v \cdot n - q ) \, ds dt =  \int_0^T \int_{\calF(t)} \beta'(\rho)f \beta(\rho).
\end{align*}
This holds for any $ \beta $, in particular by an approximation argument the equality holds for $ \beta(z) = z $.

\end{proof}

\section{The Div-Curl system in time dependent domain}

Before moving to the proof of Theorem \ref{exi:Lp:Nempty}, \ref{exi:L1:ss} and \ref{app:theo:theo} let us study some properties of the Div-Curl system. In the first subsection we study well-posedness and we show some estimates with constants independent of time. In the second one we look for uniform estimates when the size of the holes are small enough. 

\subsection{Well-posedness for the Div-Curl system in time-depending domain}

In this subsection we study the Div-Curl system in the case the holes are not allowed to become points. More precisely the fluid domain $ \calF(t) $ is give as in the beginning of Section \ref{section:setting} and we suppose $ r_i(t) \geq c > 0 $ for any $ t $. In particular we are looking for estimates independent of the shape of $ \calF(t) $. All the results are taken from \cite{Jud}.

The Div-Curl system we are interested in is 
\begin{align}
\div u = & \, 0 && \text{ for } x \in \calF(t), \nonumber \\
\curl u = & \, \omega && \text{ for } x \in \calF(t), \nonumber \\
u\cdot n = & \, g && \text{ for } x \in \partial  \calF(t), \label{div:curl:est:ell}\\
\oint_{\pS^{i}(t)} u \cdot \tau  = & \, \calC_{i}(t). &&  \nonumber
\end{align}

The following lemmas holds.

\begin{Lemma}
	\label{ell:est:1}
	Let $ l \geq 1 $, let $ p \in (1, + \infty) $ and let $ \partial \calF $ of class $ C^{l+1} $. If $ \omega(t,.) \in W^{l-1,p}(\calF(t)) $, $ g \in W^{l-1/p,p}(\partial \calF(t)) $ such that $ \oint g = 0 $ and $ \calC_i(t) \in \bbR $. Then there exists a unique solution $ u(t,.) \in W^{l,p}(\calF(t)) $ of \eqref{div:curl:est:ell} such that 
	\begin{align*}
	\| u(t,.) & \|_{W^{l,p}(\calF(t))} \leq \\ & C\frac{p^2}{p-1}\left(\|\omega(t,.)\|_{W^{l-1,p}(\calF(t))} + \|g(t,.)\|_{W^{l-1/p,p}(\partial \calF(t))} + \sum_{i = +, -}\lvert \calC_i(t)\rvert \right),
	\end{align*}	 
	where $ C $ does not depend on time. 
	
\end{Lemma}

Let denote by $ \bar{D}_t $ the convective derivative following the boundary of the domain $ \partial  \calF(t) $, i.e. $ \bar{D}_t g = \partial_t g + \nabla Y \cdot \nabla g $. The following lemmas holds.

\begin{Lemma}
	\label{ell:est:2}
	Let $ p \in (1, +\infty) $. Let $ \omega(t,.) \in W^{1,p}(\calF(t)) $ such that $ \partial_t \omega(t,.) \in L^p(\calF(t)) $, let $ g \in W^{2-1/p,p}(\partial \calF(t)) $ with $ \bar{D}_t g \in W^{1-1/p,p}(\calF(t)) $ and let $ \calC_i(t) \in \bbR $ with $ \dot{\calC}_i(t) \in \bbR $. Then there exists a unique solution $ u $ of \eqref{div:curl:est:ell} such that $ u(t,.) \in W^{2,p}(\calF(t)) $ and $ \partial_t u(t,.) \in W^{1,p}(\calF(t)) $ such that 
	\begin{align*}
	\| u(t,.)& \|_{W^{2,p}(\calF(t))} +  \| \partial_t u(t,.) \|_{W^{1,p}(\calF(t))} \leq \\  C\frac{p^2}{p-1}\Bigg(&\|\omega(t,.)\|_{W^{1,p}(\calF(t))} + \|\partial_t \omega(t,.)\|_{L^{p}(\calF(t))} + \|g(t,.)\|_{W^{2-1/p,p}(\partial \calF(t))} \\ & \quad + \|\bar{D}_t g(t,.)\|_{W^{1-1/p,p}(\partial \calF(t))} + \sum_{i = +, -}\left(\lvert\calC_i(t)\rvert+ \lvert \dot{\calC_i}(t)\rvert \right) \Bigg),
	\end{align*}	 
	where $ C $ does not depend on time. 	
\end{Lemma}

For the proof we refer to \cite{Jud}.

\subsection{Reflection method for the Div-Curl system}

In the previous section, we show uniform estimates in the case where $ r_i(t) \geq c > 0 $. In this subsection we are interested in the other scenario, i.e. in the case where the size of the holes can be arbitrarily small. 

Let us introduce some notations. Given $ i \in \{ +, - \} $, we denote by $ {}^{op}i $ the only element of $ \{ +, - \} \setminus i $. The fluid domain $ \calF(t) = \Omega \setminus \overline{\text{int}(\calS^+(t)\cup \calS^-(t))} $ and $ \calS^i(t) $ is the bounded subset of $ \bbR^2 $ with boundary  $ \partial \calS^i(t)  = \{ r^i(t)S^i(t,x) + h^i(t) $ $\text{ such that } $ $ x \in \partial B_{1}(0) \}. $
Les us denote by $ \calF_{r^+, r^-}(t) = \calF(t) $ and $ \calS^i_{r^i}(t) = \calS^i(t) $ to emphasize the presence of the parameter $r^+ $ and $r^- $ in the definition of $ \calF(t)$ and $ \calS^i(t)$. 

In the following we show some estimates uniform in $ r^+ $ and $ r^- $ for the Div-Curl system \eqref{div:curl:est:ell} with $ \omega = 0 $ and $ \calC_i(t) = 0$, in the case where $ 0 < r^+$, $ r_{-} < \kappa $ for $ \kappa $ small enough or for $ 0 < r^i < \kappa $, $ r^{{}^{op}i} \geq c > 0 $ and $ \kappa $ small enough that depends on $ c $.

\begin{Lemma}
	\label{lemma:ref:meth}
	Let $ T > 0 $ and let $ S^i$ and $ h^i$ as in the beginning of Section \ref{section:setting}. Then there exists $ \kappa_T > 0 $ such that for $ 0 < r^{+}$, $ r^{-} < \kappa_{T} $, for $ \calF_{r^{+},r^{-}}(t) = \Omega \setminus(\calS^{+}_{r_{+}}(t) \cup \calS^{-}_{r_{-}}(t))$ and $ g^i \in \tilde{L}^{q}(\pS^{i}_{r^i}(t)) $, where $ \sim $ means that $ \int_{\pS^{i}_{r^{i}}(t)} g^i = 0 $, the following estimate holds
	\begin{equation*}
	\| w\|_{L^{q}(\calF_{r^{+},r^{-}}(t))} \leq C_{T}\sum_{i \in \{+,-\}}\left(r^i\right)^{1/q} \| g^{i}\|_{\tilde{L}^{q}(\pS^{i}_{r^{i}})}
	\end{equation*} 
	for any $ w $ solution of $ \div w = 0 $ and $ \curl w = 0 $ on $ \calF_{r^+, r^-}(t) $, $ w\cdot n = g^i $ on $ \pS^{i}_{r^i} $, $ w \cdot n = 0 $ on $\partial \Omega $ and $ \int_{\pS^{i}_{r^i}} w \cdot \tau = 0 $ for $ i \in \{+,- \}$.

	Moreover for $ r^{{}^{op}i} \in C^{1,\alpha}(0,T) $ with $ r^{{}^{op}i} \geq \bar{c} > 0 $, there exist $ \kappa^{{}^{op}i}_{T}$ such that for $ r^i  < \kappa^{{}^{op}i}_{T} $ and for $ g^{i} \in \tilde{L}^{q}(\pS^{i}_{r^i}(t)) $, the following estimate holds 
	\begin{equation*}
	\| w\|_{L^{q}(\calF_{r^{+},r^{-}}(t))} \leq C_{T}(r^i)^{1/q} \| g^{i}\|_{\tilde{L}^{q}(\pS^{i}_{r^{i}})}
	\end{equation*}   
	for any $ w $ solution of $ \div w = 0 $ and $ \curl w = 0 $ on $ \calF_{r^+, r^-}(t) $, $ w\cdot n = g^i $ on $ \pS^{i}_{r^i} $,  $ w\cdot n = 0 $ on $ \pS^{{}^{op}i}_{r^{{}^{op}i}} $, $ w \cdot n = 0 $ on $\partial \Omega $ and $ \int_{\pS^{j}_{r^j}} w \cdot \tau = 0 $ for $ j \in \{+,- \}$.
	
\end{Lemma}

The proof of the above lemma is a consequence of the reflection method. The only delicate part is the fact that the domains are not only shrinking but also changing shape in time. 

For time independent domain Lemma \ref{lemma:ref:meth} was shown in Corollary 5.3.1 of \cite{IOT} and the proof relies on the Lemma 5.3.1 and 5.3.2. Due to the fact that the proof of Lemma \ref{lemma:ref:meth} is exactly the same of Corollary 5.3.1, we only present here the proof of the analogous version of Lemma 5.3.1 and 5.3.2 for time independent domain.

\begin{Lemma}[Analogous of Lemma 5.3.1]
For $ t \in [0,T] $ let $ \calF_{r^+,r^-}(t) $ such that
\begin{enumerate}
	\item or $ r^{+} = r^- = 0 $, i.e. $ \calF_{r^+,r^-}(t) = \Omega $,
	
	\item or $ 0 < c \leq r^i $ and $ r^i \in C^{1,\alpha}([0,T]) $ and $ r^{{}^{op}i } = 0 $.
\end{enumerate}
Let $ g \in L^{q}(\partial \calF_{r^+,r^-}(t) )$ such that $ \int_{\partial \calF_{r^+,r^-}(t) }  g = 0 $. Then there exists a unique solution $ h[g] \in L^q(\calF_{r^+,r^-}(t))$ such that $ - \Delta h[g] = 0 $ in $ \calF_{r^+,r^-}(t) $ and $ \nabla h[g] \cdot n = g $ on $ \partial \calF_{r^+,r^-}(t) $. Moreover for any compact subset $ K $ of $ [0,T] \odot \calF_{r^+,r^-}(t) $, we denote by $ K_t = \{ x \in \Omega $ such that $ (t,x) \in K\cap {t}\times \Omega \} $ and we have that $ h[g] $ is of class $ C^{k}(K_t) $ and 
$$ \|h[g]\|_{C^k(K_t)} \leq C \|g\|_{L^q}(\partial \calF_{r^+,r^-}(t)), $$	
with $ C $ independent of time.	
		
\end{Lemma}

\begin{proof} The proof is base on improved interior estimates that are well-know from Lemma \ref{ell:est:1}.
\end{proof}

Let introduce the space 
\begin{align*}
\mathcal{W}^{1,\beta}(\bbR^2 \setminus  \calS^i_{r^i}) = \big\{  u \in \mathcal{D}'(\bbR^2 \setminus  \calS^i_{r^i}) \text{ such that } & 
\| \lvert 1+ \lvert . \rvert^2 \rvert^{-1/2} \lvert u(.) \rvert \|_{L^{\beta}(\bbR^2 \setminus  \calS^i_{r^i})} \\ 
& \,  \text{ and } \|\nabla u \|_{L^{\beta}(\bbR^2 \setminus  \calS^i_{r^i})} < \infty    \big\}
\end{align*}
with norm given by
$$ \| u \|_{\mathcal{W}^{1,\beta}(\bbR^2 \setminus  \calS^i_{r^i})} = \| \lvert 1+ \lvert.\rvert^2\rvert ^{-1/2} \lvert u(.) \rvert  \|_{L^{\beta}(\bbR^2 \setminus  \calS^i_{r^i})} + \|\nabla u \|_{L^{\beta}(\bbR^2 \setminus  \calS^i_{r^i})} $$

\begin{Lemma}[Analogous of Lemma 5.3.2]
	
Let $ E(\calS^i_{r^i}) = \bbR^2 \setminus  \calS^i_{r^i}  $, let $ g \in L^q(\partial  \calS^i_{r_i} )$ such that $ \oint_{ \partial  \calS^i_{r^i} } g = 0 $, let $ \beta > 2 $, let $ \hat{f}[g] 
$ to be the unique solution in $ \mathcal{W}^{1,\beta}(E(\calS^i_{r^i})) $ of $ - \Delta \hat{f}[g]  = 0 $ in $ E(\calS^i_{r^i})  $, $ \nabla \hat{f}[g] \cdot n = 0 $ on $ \partial \calS^i_{r^i} $ and $\lvert \hat{f}[g]\rvert \to 0 $ as $ \lvert x \rvert \to + \infty $. Then it holds
\begin{gather*}
\|\nabla \hat{f}[g] \|_{L^{q}( E(\calS^i_{r^i}))} \leq C(r^i)^{1/q}\|g\|_{L^{q}( \partial \calS^i_{r^i})} \quad \text{and} \quad \\ \lvert \nabla \hat{f}[g] \rvert \leq C\frac{(r^i)^{2-1/p}}{\lvert x-h(t)\rvert^2}\|g\|_{L^{q}( \partial \calS^i_{r^i})} \text{ for } x \text{ s. t. } \lvert x-h(t) \rvert \geq \bar{C} r^i,
\end{gather*}	
where $ C $ and $ \bar{C}$  do not depend on $ t $ and $ r^i $. 	

\end{Lemma}

\begin{proof}

To show the result we use scaling estimates, so it is enough to show that in the case $ r^i = 1 $ there exist a unique solution in $ \mathcal{W}^{1,\beta}(E(\calS^i_{1})) $ of $ - \Delta \hat{f}[g]  = 0 $ in $ E(\calS^i_{1})  $, $ \nabla \hat{f}[g] \cdot n = 0 $ on $ \partial \calS^i_{1} $ and $\lvert \hat{f}[g]\rvert \to 0 $ as $ \lvert x \rvert \to + \infty $ such that 
\begin{gather*}
\|\nabla \hat{f}[g] \|_{L^{q}( E(\calS^i_{1}))} \leq C \|g\|_{L^{q}( \partial \calS^i_{1})} \quad \text{and} \quad \\ \lvert \nabla \hat{f}[g] \rvert \leq C\frac{1}{\lvert x-h(t)\rvert^2}\|g\|_{L^{q}( \partial \calS^i_{1})} \text{ for } x \text{ s. t. } \lvert x-h(t)\rvert \geq \bar{C},
\end{gather*}	
where $ C $ and $ \bar{C}$  do not depend on $ t $.

In the case $ q \geq 2 $ then is enough to consider $ \beta = p $ and existence and uniqueness in $ \mathcal{W}^{1,q}$ follows from Theorem 3.1 and Proposition 3.3 of \cite{AGG}. Moreover the estimates 
\begin{equation*}
\|\nabla \hat{f}[g] \|_{L^{q}( E(\calS^i_{1}))} \leq C \|g\|_{L^{q}( \partial \calS^i_{1})}
\end{equation*}    
holds with a constant independent of time because the proof is based on a Poincar\'e inequality that is deduce to the case of exterior domain via a partition of unity which can be chosen smooth in the time variables.

Regarding the decay at infinity, let $ B_R(h(t)) $ a ball such that $ \calS^i \subset   B_{R/2}(h(t))$ for $ t \in [0,T] $. Let $ w $ the solution of 
$$ - \Delta w = 0 \text{ in } \bbR^2 \setminus B_{R}(h(t)), \quad \nabla w \cdot n = \nabla \hat{f}[g] \cdot n \text{ on } \partial B_{R}(h(t)) \quad \text{ and } \quad $$$$ \lvert w \rvert \to 0 \text{ for } \lvert x \rvert \to +\infty .$$
First of all notice that the only solution in $ \mathcal{W}^{1,q} $ is given by $ \hat{f}[g] $. By interior estimates we have 
$ \|\nabla \hat{f}[g] \cdot n \|_{L^1(\partial B_{R}(h(t)))} \leq C\| g \|_{L^q(\partial \calS^i_1)} $
and $ w(t,x) =  \int_{\partial B_R(h(t))} G(x-h(t),y-h(t)) g $, where $ G $ are the Green function associated with the Newman Laplacian on the exterior of a ball. It is possible to compute explicitly $ G $ and deduce that
$$ \lvert \nabla \hat{f}[g] \rvert = \lvert w \rvert \leq C\frac{1}{\lvert x-h(t)\rvert^2}\|g\|_{L^{q}( \partial \calS^i_{1})} \text{ for } x \text{ s. t. } \lvert x-h(t)\rvert \geq \bar{C}.$$ 
Let conclude with the interesting case $ q \in (1,2) $. First of all let us notice that $ L^q(\partial \calS^i_1) \subset W^{-1/\beta,\beta} $ for any $ \beta > 2 $. Theorem 3.1 and Proposition 3.3 of \cite{AGG} imply existence an uniqueness of solutions in $ \mathcal{W}^{1,\beta}$. As in the case $ q \geq 2 $, we deduce that for a big enough $ C $ independent of time, it holds   
$$ \lvert \nabla \hat{f}[g] \rvert  \leq C\frac{1}{\lvert x-h(t)\rvert^2}\|g\|_{L^{q}( \partial \calS^i_{1})} \text{ for } x \text{ s. t. } \lvert x-h(t)\rvert \geq \bar{C}.$$ 
We deduce that $ \|\nabla \hat{f}[g]\|_{L^q(\bbR^2 \setminus B_R(h(t)))} \leq C \|g\|_{L^q(\partial \calS^i_1)} $. Let now notice that $ B_R(h(t)) \setminus \calS^i_1 $ is bounded so the solution of a Laplace problem with Newman boundary condition satisfy
$$  \|\nabla \hat{f}[g]\|_{L^q( B_R(h(t)) \setminus \calS^i_1)} \leq C \|g\|_{L^q(\partial \calS^i_1)} + \|\nabla \hat{f}[g]\|_{L^q(\partial B_R(h(t)))} \leq (C+\bar{C}) \|g\|_{L^q(\partial \calS^i_1)}. $$	
\end{proof}

\section{Proof of Theorem \ref{exi:Lp:Nempty} and \ref{exi:L1:ss}}

In this section we prove Theorem \ref{exi:Lp:Nempty} and \ref{exi:L1:ss}. The idea is similar to the one used for the transport equation. More precisely we will consider a viscous approximation of the Euler system, which is a Navier-Stokes system with non-physical boundary conditions, and we pass to the limit as the viscosity parameter $ \nu $ tends to zero. The main difference is how to pass to the limit in the non-linear term. To do that we will take advantage of the duality formula. The viscous system reads
\begin{align}
\partial_t \omega_{\nu} + v_{\nu}\cdot \nabla \omega_{\nu} - \nu \Delta \omega_{\nu} = & \, 0 \quad && \text{ for } x \in \calF(t), \nonumber \\    
\nu \partial_n \omega_{\nu} -(\omega_{\nu} -\omega^{+}_{\nu})(g_{\nu}-q) \mathds{1}_{\partial F_{+}} = &  \, 0 \quad && \text{ for } x \in \partial \calF(t), \nonumber \\
\div v_{\nu} = & \, 0 \quad && \text{ for } x \in \calF(t), \nonumber \\
\curl v_{\nu} = & \, \omega_{\nu} \quad && \text{ for } x \in \calF(t), \label{equ:app:ss} \\
v_{\nu} \cdot n = & \, g_{\nu} \quad && \text{ for } x \in \partial \calF(t), \nonumber \\
\oint_{\pS_i} v_{\nu} \cdot \tau =  \calC^{in}_{i} - \int_0^{t}\int_{\pS_i} (\omega^{+}_{\nu} \mathds{1}_{\partial \calF^{+}} + & \omega_{\nu}\mathds{1}_{\partial \calF^{-}})(g_{\nu}-q). && \nonumber 
\end{align} 

For smooth enough data system \eqref{equ:app:ss} admits regular solutions in the sense Lemma \ref{exi:Lama:timedep}.

\subsection{Existence and uniqueness of solutions for the viscous approximation}

We show well-posedness for the system \eqref{equ:app:ss}.
\begin{Lemma}
	\label{exi:Lama:timedep}
	Let $ \omega^{in}_{\nu} \in H^1(\calF(0)) $, $ \omega_{\nu}^{+} \in C^1([0,T] \odot \calF(t)) $ and $ g_{\nu} \in C^1([0,T] \odot \calF(t)) $, the system admits a unique solution of \eqref{equ:app:ss} with $ \omega_{\nu} \in L^{2}_{loc}(\bbR^{+};H^{2}(\calF(t)))\cap C^{0}_{loc}(\bbR^{+}; H^{1}(\calF(t)))$ with $ \partial_t  \omega_{\nu} \in L^{2}_{loc}(\bbR^{+}; L^2(\calF(t)))$. Moreover the following estimate holds. For any positive convex even function $ G $,
    \begin{align}
    \label{est:con:Gun}
    \int_{\calF(t)} G( \lvert \omega_{\nu}\rvert) + \int_0^{t}\int_{\partial \calF^{-}(t)} (g_{\nu}-q) G(\lvert \omega_{\nu}\rvert)   \leq  & \, \int_{\calF(t)} G(\lvert \omega_{\nu}^{in}\rvert) \\ & \, - \int_0^{t} \int_{\partial \calF^{+}(t)} (g_{\nu}-q) G(\lvert \omega_{\nu}^{+} \rvert). \nonumber
    \end{align}
    In particular for $ p > 1 $ and $ G(x) = \lvert x \rvert^p $ the above inequality reads		
	\begin{align}
	\label{est:con:lone}
	\frac{d}{dt}\int_{\calF(t)} \lvert \omega_{\nu}\rvert^{p} + \int_{\partial \calF^{-}(t)} (g_{\nu}-q) \lvert \omega_{\nu} \rvert^{p} + \nu \frac{4(p-1)}{p} \int_{\calF(t)}  & \left(\nabla \lvert \omega_{\nu}\rvert^{\frac{p}{2}}\right)^2  \leq \\ & \,  -\int_{\partial \calF^{+}(t)} (g_{\nu}-q) \lvert \omega_{\nu}^{+}\rvert^{p}. \nonumber 
	\end{align}
	
\end{Lemma}

The existence of regular solutions for the system \eqref{equ:app:ss} was done in \cite{IO3} Lemma 4, in the case the domain $ \calF(t) = \calF(0) $ does not depend on time. In our setting the extra difficulty is to deal with the fact that $ \calF(t) $ is time dependent.

\begin{proof} Existence and uniqueness are shown as in Lemma 4 of \cite{IO3} and the proof is based on a Schauder fixed point argument. Regarding the estimate \eqref{est:con:Gun}, it follows formally by multiply the first equation of \eqref{equ:app:ss} by $ G'(\lvert \omega_{\nu}\rvert)  \omega_{\nu} / \lvert \omega_{\nu} \rvert $, integrate in $ (0,T) \odot \calF(t)  $ and some integrations by part. For a rigorous proof see Proposition 2 of \cite{IO3}.

Let us recall that  the Schauder fixed point Theorem asserts that if  $\mathcal{Z}$ is a non-empty convex closed subset of a 
normed space $\mathcal{X}$ and $F:\mathcal{Z}\mapsto\mathcal{Z}$ is a continuous mapping 
such that $ F(\mathcal{Z})$ is contained in a compact subset of $ \mathcal{B}$, then
$F$  has a fixed point.

In our setting the space $ \mathcal{Z} $ is 

\begin{align*}
\mathcal{Z} = \Bigg\{ \omega \in \bigcap_{i=0}^{1} H^{i}(0,T;H^{2-2i}(\calF(t))) \quad \text{ such that } \quad 
\| \omega \|_Z := \| \omega \|_{\bigcap_{i=0}^{1} H^{i}(0,T;H^{2-2i}(\calF(t)))} \leq R \Bigg\}, 
\end{align*}
for an $ R > 0 $ big enough. The map $ F:\mathcal{Z} \longrightarrow \bigcap_{i=0}^{1} H^{i}(0,T;H^{2-2i}(\calF(t))) $ is defined by $ F(\omega) = \bar{\omega} $ which is the solution of 
\begin{align*}
\partial_t \bar{\omega} + u_{\omega} \cdot \nabla \bar{\omega}- \nu \Delta \bar{\omega} = & \,  0  \quad && \text{ for  } x \in \calF(t), \\
\partial_n \bar{\omega} =  & \, (\bar{\omega} -\omega^+_{\nu})g_{\nu} \quad && \text{ for } x \in \calF(t) \\
\text{div } u_{\omega} =  &\, 0  \quad && \text{ for } x \in \calF(t),  \\
\text{curl } u_{\omega} = &\, \omega   \quad && \text{ for  } x \in \calF(t),  \\
u_{\omega} \cdot n = & \, g_{\nu}   \quad && \text{ for } x \in \partial \calF(t),  \\
\oint_{\partial \calS^{i}} v_{\omega} \cdot  \tau =   \calC^{in}_{i} - \int_0^{t}\int_{\pS_i} (\omega^{+}_{\nu} & \mathds{1}_{\partial \calF^{+}} + \omega \mathds{1}_{\partial \calF^{-}})(g_{\nu}-q).  && \nonumber 
\end{align*}
For any $\omega$ in $ \mathcal{Z} $, we have that $ v_{\omega} $ is in $\bigcap_{i=0}^{1}H^i(0,T;H^{2-2i+1}(\calF(t)))$ from Lemma \ref{ell:est:1} and \ref{ell:est:2}. 
Then, by using the  \textit{a priori} estimates  from Proposition \ref{exi:vis:sys}, we observe that the $ \mathcal{Z} $-norm of $ \bar{\omega} $ depends on the $ \calL^2(0,T;H^{1}(\calF(t)))\cup \calL^2(0,T;L^{\infty}(\calF(t))) $ of $ v_{\omega} $, which converges,  as $ t $ tend to zero, to zero uniformly with respect to $ \omega \in \mathcal{Z} $.
Therefore for  $ T $ small  enough, we conclude that $ F(\mathcal{Z}) \subset \mathcal{Z} $.

Let us now prove that  $ F $ is relatively compact.  
Let $ (\omega_{j} )_j$ a  bounded sequence in $ \mathcal{Z} $. Then, up to a subsequence, $ \omega_{j} \cv \omega $ in $ \mathcal{Z} $. By Rellich's theorem the convergence is  strong in $ \calL^2(0,T; H^{1 }(\calF(t))) .$
We deduce that the corresponding velocity $ v_{j} = v_{\omega_{j}} $ converges to $v $ in 
$ \calL^2(0,T; H^{2}(\calF(t))) .$
Moreover for any $j$, the function $ w_j = \bar{\omega}-\bar{\omega}_{j}$ satisfies the system 
\begin{align*}
\partial_{t} w_j + v\cdot \nabla w_j - \Delta w_j = & \, -(v-v_{j})\cdot \nabla \bar{\omega}_{j}  \quad && \text{ for  } x \in \calF(t),\\
\partial_n w_j = & \, w_j g   \quad &&  \text{ for } x \in \partial \calF(t),
\end{align*} 
with zero initial data. 
We observe that 
\begin{align*}
\left\| (v-v_j)\cdot \nabla \bar{\omega}_j \right\|_{\calL^2(0,T;L^2(\calF(t)))} 
\leq & \, \left\| (v-v_j )\|_{L^{2}(0,T;L^{\infty}(\calF))} \| \nabla \bar{\omega}_{j}\right\|_{L^{\infty}(0,T;L^2(\calF))}
\end{align*}
which converges to zero.
Then by using the \textit{a priori} estimates, we deduce that $ w_j $ converges to $ 0 $ in $ \mathcal{Z} $. Thus $ F $ is relatively compact. 
The continuity of $F$ can be proved along the same lines. 
Thus Schauder's fixed point theorem can be applied. It implies that $F$ has a fixed point in $\mathcal{Z}$. This has proved the local in time existence of strong solutions.
Moreover  the existence to all $  [0,T] $  can be deduced from the \textit{a priori} estimates.  In fact if we suppose by contradiction that there exists a maximal time of existence $ t_{*} < T $, the  \textit{a priori}  estimates ensure that $ \omega(t_{*}) $ is enough regular to apply again the local existence result and we obtain a contradiction. Uniqueness follows from the energy estimate and Gr\"onwall's lemma.

\end{proof}

Before going in the proof of Theorem \ref{exi:Lp:Nempty} and \ref{exi:L1:ss}, let us explain how to regularize the initial data to be able to apply Lemma \ref{exi:Lama:timedep}.

\subsection{Regularization of the initial data}

Under the hypothesis of Theorem \ref{exi:Lp:Nempty} and \ref{exi:L1:ss} the initial data are not regular enough to apply Lemma \ref{exi:Lama:timedep}. To avoid this issue we consider a sequence of initial data with the following properties.

\begin{Lemma}
	\label{approx:lem:2}
	Let $ p \in (1,+\infty) $. Let $ g \in L^1_{loc}(\bbR^+;W^{1-1/p,p}(\calF(0))) $ such that $\div_x(\bar{v}(t,Y(t,x))) = 0 $. Let $ \omega^{in} $ and $ \omega^{+} $ respectively in $ L^{p}(\calF(0)) $ and $ L^p_{loc}(\partial \calF^+; \,\lvert g - q \rvert \, dt ds ) $. Then there exist  $ g_{\nu} $, $ \omega^{in}_{\nu} $ and $ \omega_{\nu}^+ $ such that $ g_{\nu} \in C^{1,\alpha}([0,T];C^{2,\alpha}(\partial \calF(t))) $. Such that $ \omega^{in}_{\nu} \in H^{1}(\calF(0)) $ and such that $ \omega_{\nu}^+ \in  C^1([0,T] \times \partial \calF(0)) $. Moreover
	\begin{align*}
	g_{\nu} \longrightarrow & \, g  \quad && \text{ in } \calL^1(0,T; W^{1-1/p,p}(\calF(t))), \nonumber \\
	\omega^{in}_{\nu} \longrightarrow & \, \omega^{in} \quad && \text{ in } L^p(\calF(0)),  
	\\
	\omega^{+}_{\nu}(g_{\nu} - q )^{1/p} \mathds{1}_{\partial \calF^+_{\nu}}  \longrightarrow  & \, \omega^{+}(g - q )^{1/p} \mathds{1}_{\partial \calF^+} \quad && \text{ in } L^{p}((0,T)\odot \partial \calF(t)). \nonumber 
	\end{align*}
	And  
	\begin{gather*}
	\int_{\calF(0)}\lvert \bar{\omega}^{in}_{\nu} \rvert^2 \, dy -  \iint_{\partial \calF^+_{\nu}}  \lvert \bar{\omega}_{\nu}^+\rvert^2 (\bar{v}_{\nu}\cdot n- \bar{q}) ds dt \leq \\ \frac{1}{\sqrt{\nu}} \left( \int_{\calF(0)}\lvert \bar{\omega}^{in}_{\nu} \rvert^p \, dy -   \iint_{\partial \calF^+}  \lvert\bar{\omega}_{\nu}^+\rvert^p (\bar{v}_{\nu}\cdot n- \bar{q}) ds dt \right).
	\end{gather*} 

\end{Lemma}

\begin{proof}

The proof follow exactly the one of Lemma \ref{approx:lem}.

\end{proof}

We are now able to prove Theorem \ref{exi:Lp:Nempty} and \ref{exi:L1:ss}. We start from the first one.

\subsection{Proof of Theorem  \ref{exi:Lp:Nempty}}

\begin{proof}[Proof of Theorem \ref{exi:Lp:Nempty}]
	
Let $ \omega^{in} $, $ \calC^{in}_{i} $, $ g $ and $ \omega^{+} $ the given data. Then from Lemma \ref{approx:lem:2}, there exist regular approximate data $ \omega^{in}_{\nu} \in H^1(\calF(0))$, $ g_{\nu}\in C^{1,\alpha}_{loc}(\bbR^{+}; C^{2,\alpha} \partial \calF(t)) $ and $ \omega^{+}_{\nu} \in C^{2}_{loc}(\bbR^{+} \odot \partial \calF^{+}_{\nu}(t))  $, where $ \partial \calF^{+}_{\nu}(t) = \{ x \in \partial \calF(t) $ such that $ g_{\nu} - q < 0 \}$, such that 
\begin{gather*}
\omega^{in}_{\nu} \longrightarrow \omega^{in} \text{ in } L^{p}(\calF(0)), \quad g_{\nu} \longrightarrow g \text{ in } \calL_{loc}^{r}(\bbR^{+}; W^{1-1/p,p}(\partial \calF(t))) \quad \text{ and } \\ \omega^{+}_{\nu}\mathds{1}_{\partial \calF^{+}_{\nu}(t)}(g_{\nu}- q)^{1/p} \longrightarrow \omega^{+} \mathds{1}_{\partial \calF^{+}(t)}(g- q)^{1/p} \text{ in } \calL^{r}_{loc}(\bbR^{+}; L^{p}(\partial \calF(t))).
\end{gather*}
	%
	
For the data $ \omega^{in}_{\nu} $, $ \calC^{in}_{i} $, $ g_{\nu} $ and $ \omega^{+}_{\nu} $ Lemma \ref{exi:Lama:timedep} applies. In particular we deduce the existence of a solutions $ (\omega_{\nu}, v_{\nu}) $ such that	\eqref{est:con:lone} holds true. We deduce that $ \omega_{\nu} $ is uniformly bounded in $ \calL^{\infty}_{loc}(\bbR^{+};L^{p}(\calF(t)))$ and in $ L^{p}_{loc}(\bbR^{+}; L^{p}(\partial \calF^{-}_{\nu}(t); (g_{\nu}-q) ds )))$ and up to subsequence 
\begin{align}
\label{conv:visc:app:meri}
\omega_{\nu} \cvwstar \omega \text{ in } \calL^{\infty}_{loc}(\bbR^{+};L^{p}(\calF(t))) \quad \text{ and } \quad  \quad \quad \quad \quad \quad \quad \quad \\ \omega_{\nu}\mathds{1}_{\partial \calF^{-}_{\nu}(t)}(g_{\nu}- q)^{1/p} \cv \omega^{-} \mathds{1}_{\partial \calF^{-}(t)}(g- q)^{1/p} \text{ in } \calL^{r}_{loc}(\bbR^{+}; L^{p}(\partial \calF(t))). \nonumber 
\end{align}
Note that in the only non trivial part it the fact that $ \partial \calF^{-} \neq \partial \calF^{-}_{\nu}$. The next step is to notice that $ \partial_t \omega_{\nu} $ is uniformly bounded in $ \calL^{r'}_{loc}(\bbR+; H^{-s}_0(\calF(t))) $ for some $ s $ big enough, where $ H^{-s} $ is the dual of $ H^s_0$. To see this let decompose the velocity field as in \eqref{u:dec:ss}, i.e.  
$$ v_{\nu} = v_{g_{\nu}} + \calK^{0}_{\calF(t)}[\omega_{\nu}]+\sum_{i} \left[ \int_{\calF} \Psi_i \omega_{\nu} + \calC_i(t)  \right] X_i $$  
\begin{align*}
\int_{0}^{T}& \langle \partial_t \omega_{\nu}, \varphi \rangle \, dt =   \int_{0}^T\int_{\calF(t)} \omega_{\nu} v_{\nu} \cdot \nabla \varphi \, dx dt -  \nu \int_0^T \int_{\calF(t) } \lvert \nabla \omega_{\nu} \rvert^2 \, dx dt  \\  = & \, \int_{0}^T\int_{\calF(t)} \omega_{\nu} \left( v_{g_{\nu}} + \calK^{0}_{\calF(t)}[\omega_{\nu}]+\sum_{i} \left[ \int_{\calF} \Psi_i \omega_{\nu} + \calC_i(t)  \right] X_i \right) \cdot \nabla \varphi \, dx dt \\ & \, -  \nu \int_0^T \int_{\calF(t) } \lvert \nabla \omega_{\nu} \rvert^2 \, dx dt .
\end{align*}
Let us estimate the four terms of the right-hand side separately. Regarding the first one we denote by $ \eta_{\nu} $ the solution of $ \Delta \eta_{\nu} = \omega_{\nu} $ in $ \calF(t) $ and $ \eta_{\nu} = 0 $ on $ \partial \calF(t) $. We have 
\begin{align*}
\int_{0}^T\int_{\calF(t)} \omega_{\nu} v_{g_{\nu}} \cdot \nabla \varphi  \, dx dt = & \, \int_{0}^T\int_{\calF(t)} \Delta \eta_{\nu} v_{g_{\nu}} \cdot \nabla \varphi \, dx dt \\
= & \, \int_{0}^T\int_{\calF(t)}  \eta_{\nu} \Delta\left( v_{g_{\nu}} \cdot \nabla \varphi  \right) \, dx dt \\
= & \, 2 \int_{0}^T\int_{\calF(t)}  \eta_{\nu} \nabla v_{g_{\nu}} : \nabla^2 \varphi  \, dx dt  \\ & \, + \int_{0}^T\int_{\calF(t)}  \eta_{\nu}  v_{g_{\nu}} \cdot \nabla \Delta \varphi  \, dx dt,
\end{align*}
where we use that $ v_{g_{\nu}} = \nabla \psi $ with $ - \Delta \psi = 0 $ in $ \calF(t) $ and $ \psi = 0 $ on $ \partial \calF(t) $. It is now easy to see that
\begin{align*}
\lvert \int_{0}^T\int_{\calF(t)} & \omega_{\nu} v_{g_{\nu}} \cdot \nabla \varphi  \, dx dt \rvert  \\ \leq  & \,   \|\eta_{\nu}\|_{\calL^{\infty}(0,T;L^{p'}(\calF(t))} \|v_{\gamma_{\nu}}\|_{\calL^r(0,T;W^{1,p}(\calF(t)))} \| \varphi \|_{\calL^{r'}(0,T;W^{3\infty}(\calF(t)))}  \\
\leq & \, C \|\eta_{\nu}\|_{\calL^{\infty}(0,T;W^{2,p}(\calF(t))} \|v_{\gamma_{\nu}}\|_{\calL^r(0,T;W^{1,p}(\calF(t)))} \| \varphi \|_{\calL^{r'}(0,T;W^{3\infty}(\calF(t)))},  \\
\leq & \, C \|\omega_{\nu}\|_{\calL^{\infty}(0,T;L^{p}(\calF(t))} \|v_{\gamma_{\nu}}\|_{\calL^r(0,T;W^{1,p}(\calF(t)))} \| \varphi \|_{\calL^{r'}(0,T;W^{3\infty}(\calF(t)))}, 
\end{align*}
where we use $ p > 1 $. Regarding the second term we recall that it can be rewritten as 
$$ \int_0^T \int_{\calF(t)} \omega_{\nu} \calK^0_{\calF(t)}[\omega_{\nu}] \cdot \nabla \varphi \, dx dt = \int_0^T \int_{\calF(t)} \int_{\calF(t)} H_{\varphi}(t,x,y) \omega_{\nu}(t,x) \omega_{\nu}(t,y) \, dx dy dt, $$
where $ H_{\varphi} $ is defined in \eqref{H:varphi} and satisfy $ \|H_{\varphi}\|_{L^{\infty}(\calF(t))} \leq \| \varphi \|_{W^{2,\infty}(\calF(t))}$. We deduce that 
$$ \lvert \int_0^T \int_{\calF(t)} \omega_{\nu} \calK^0_{\calF(t)}[\omega_{\nu}] \cdot \nabla \varphi \, dx dt \rvert \leq C \|\varphi \|_{\calL^1(0,T; W^{2,\infty}(\calF(t)))} \|\omega_{\nu} \|^2_{\calL{\infty}(0,T;L^{p}(\calF(t)))}. $$
In the third term the velocity is the linear combination of finitely many vector fields which are bounded. This implies that this term can be tackle easily. Finally the last term is bounded by energy estimates.   

We showed that $ \partial_t \omega_{\nu } $ is uniformly bounded in $ \calL^{r'}_{loc}( \bbR+;H^{-s}(\calF(t))) $, which implies together with 
\eqref{conv:visc:app:meri} that 
$$ \omega_{\nu} \longrightarrow \omega \quad \text{ in } C^{0}_{loc}(\bbR^+; L^p-w(\calF(t)) ),$$
where $ L^p-w(\calF(t)) $ denotes the $ L^{p}(\calF(t)) $ space endowed with the weak topology. This implies that
\begin{equation*}
\calK^0_{\calF(t)}[\omega_{\nu}] \longrightarrow \calK^{0}_{\calF(t)}[\omega] \quad \text{ in } \calL^{r}_{loc}(\bbR+;L^p(\calF(t))),
\end{equation*}
as a consequence of Lemma 6.4 of \cite{NS}. Due to the $ \calL^{r}_{loc}(0,T;W^{1-1/p,p}(\calF(t))) $ convergence of $ g_{\nu}$  towards $ g  $, we have that $ v_{g_{\nu}} $ converges to $ v_{g} $ in $ \calL^{r}_{loc}(\bbR+;L^p(\calF(t)))
$. Finally the convergence 
$$ \sum_{i} \left[ \int_{\calF} \Psi_i \omega_{\nu} + \calC_i(t)  \right] X_i \longrightarrow \sum_{i} \left[ \int_{\calF} \Psi_i \omega + \calC_i(t)  \right] X_i  \quad \text{ in } \calL^{r}_{loc}(\bbR+;L^p(\calF(t))) $$
is straight-forward. We deduce that 
\begin{equation*}
v_{\nu} \longrightarrow v \text{ in } \calL^{r}_{loc}(\bbR^{+}; L^{p}(\calF(t))).
\end{equation*}
	and by linearity of Div-Curl system it holds \eqref{div_curl:blabla}.

	Theorem \ref{exi:transport:ppppppppp} ensures the existence of a solution $ (\bar{\omega}, \bar{\omega}^{-}) $ associated with the velocity field $ v $ and the data $ (\omega^{in}, \omega^{+}) $. It remains to show that $ (\omega, \omega^{-}) = (\bar{\omega}, \bar{\omega}^{-}) $. To do that we use a backward flow $ (\phi_{\nu}, \phi^{+}_{\nu}) $, see for instance \cite{Cri:Spi} or \cite{IO3}, solution of

	\begin{align}
	-\partial_t \phi_{\nu} -v_{\nu}\cdot \nabla \phi_{\nu} - \nu \Delta \phi_{\nu} = & \, \chi \quad && \text{ for } x \in \calF(t), \nonumber \\
	\nabla \phi_{\nu} \cdot n = & \, -( \phi_{\nu} - \Psi)(g_{\nu}-q) \mathds{1}_{\partial \calF_{\nu}^{-}(t)} \quad && \text{ for } x \in \partial \calF(t), \label{vis:back:trans:equ}   \\
	\phi_{\nu}(T,x) = & \, 0, &&  \nonumber
	\end{align}
	
	where $ \chi $ and $ \Psi $ are smooth functions. And 
	\begin{align}
	-\partial_t \phi -v \cdot \nabla \phi = & \, \chi \quad && \text{ for } x \in \calF(t), \nonumber \\
	\phi = & \, \Psi  \quad && \text{ for } x \in \partial \calF^{-}(x),  \label{back:trans:equ} \\
	\phi(T,x) = & \, 0, &&  \nonumber
	\end{align}
	Then there exists a solution $ \phi_{\nu} $ of \eqref{vis:back:trans:equ}  in the sense of Lemma \ref{exi:Lama:timedep} and a distributional solution of \eqref{back:trans:equ}.
	
	The functions $ \phi_{\nu} $ satisfy estimates of the type \eqref{est:con:lone}, which imply a uniform bound in $ L^{\infty}(0,T;L^{p}(\calF)) $, moreover due to \eqref{vis:back:trans:equ} $ \phi_{\nu} $ is uniformly continuous in some $ W^{-1, r}(\calF) $ for some $ r $. It follows that for a subsequence
	\begin{gather*}
	\phi_{\nu} \longrightarrow \bar{\phi} = \phi \quad \text{ in } C^0([0,T];L^{q}-w(\calF)) \quad \text{ and } \quad \\ (g_{\nu}-q)^{1/q}\mathds{1}_{\partial \calF^{+}_{\nu}} \phi_{\nu} \cv \bar{\psi} = \lvert g-q \rvert^{1/q} \mathds{1}_{\partial \calF^{+}}\phi^{+} \text{ in } L^{q}((0,T)\times \partial \mathcal{F})
	\end{gather*}
	and the identification $ (\bar{\phi}, \lvert g-q \rvert^{-1/q}\bar{\psi})  = (\phi, \phi^{+}) $ comes from the fact that we can pass to the limit in the weak formulation satisfied by $ \phi_{\nu } $ to show that $ (\bar{\phi}, (g-q)^{-1/q}\bar{\psi}) $ is a weak solution of \eqref{back:trans:equ} and conclude by uniqueness. Note that here we use the fact that $ \phi v \in L^{1}(\calF) $.

	Using the duality formula from Theorem \ref{duality:formula} with $ u =  \bar{\omega} $ and $ r(t,.) = \phi(T-t) $, we have
	%
	\begin{equation*}
	\int_0^{T}\int_{\calF} \bar{\omega} \chi + \int_0^{T} \int_{\partial \calF^{-}} (g-q) \bar{\omega}^{-} \Psi = \int_{\calF} \omega^{in} \phi(0,.) - \int_{0}^{T} \int_{\partial \calF^{+}} (g -q)\omega^{+} \phi^{+} 
	\end{equation*}
	
	Consider now the equation satisfy by $ \omega_{\nu} $ tested with $ \phi_{\nu} $. Due to the convergences previously showed we deduce 
	\begin{equation*}
	\int_0^{T}\int_{\calF} \omega \chi + \int_0^{T} \int_{\partial \calF^{-}} (g-q) \omega^{-} \Psi = \int_{\calF} \omega^{in} \phi(0,.)  - \int_{0}^{T} \int_{\partial \calF^{+}} (g-q) \omega^{+} \phi^{+}.
	\end{equation*}
	The right hand side of the two above equalities is the same. We deduce that  
	\begin{equation*}
	\int_0^{T}\int_{\calF} (\bar{\omega} - \omega)\chi + \int_0^{T} \int_{\partial \calF^{-}} (g-q) (\omega^{-} - \bar{\omega}^{-}) \Psi = 0
	\end{equation*}
	for any smooth function $ \chi $ and $ \Psi $. This shows $ (\omega, \omega^{-}) = (\bar{\omega}, \bar{\omega}^{-}) $.
	
\end{proof}

We move to the case of $ L^1 $ vorticity.

\begin{proof}[Proof of Theorem \ref{exi:L1:ss}] The proof of this result is similar to Theorem 4 of \cite{IO3}. Regarding the convergence of $ \omega_{\nu}(t,x)\omega_{\nu}(t,y) $, it enough to extend $ \omega_{\nu} $ by zero. It follows that $ \omega_{\nu}(t,x)\omega_{\nu}(t,y) \cv \omega(t,x)\omega(t,y)$ in $ L^{\infty}(0,T,\mathcal{M}(\bbR^{2}))$, where $ \mathcal{M}(\bbR^{2}) $ is the space of Random measures. 
	
	Note now that $ \omega(t,x)\omega(t,y) $ are zero in the exterior of the closure of $ \bbR^{+} \odot \calF(t) $ and $ H_{\varphi, \delta} $ which is $ H_{\varphi} $ multiply by a cut-off along the diagonal (see \cite{Scho}) is continuous up to the boundary due to its regularity. To conclude it is enough to show that $ \omega_{\nu}(t,x)\omega_{\nu}(t,y) $ do not concentrate along the diagonal, but this follow from \eqref{est:con:Gun}.
	
\end{proof}

\section{Proof of Theorem \ref{app:theo:theo}}

In this section we prove Theorem \ref{app:theo:theo}. The proof is dived in four steps. In the first one we show a priori bounds for $ (\omega_{\eps}, \omega_{\eps}^{-}) $ from which we deduce weak convergence of the vorticity in step two. In the third one we show strong convergence of the velocity and in the last one we explain how to pass to the limit and derive the system \eqref{main:sys}.

\subsection{Uniform bounds}
\label{unif:bound}

In this subsection we prove uniform bounds for $ \omega_{\eps} $, $\omega_{\eps}^{-}$ and $ v_{\eps} $ in appropriate spaces.

\begin{Lemma}
	\label{firs_apriori_est}
	Let $ (\omega_{\eps}, \omega_{\eps}^{-}, v_{\eps}) $ a weak solution of \eqref{main:sys} satisfying the hypothesis of Theorem \ref{app:theo:theo} and let $ T > 0 $. Then $\| \omega_{\eps} \|_{\mathcal{L}^{\infty}(0,T; L^{p}(\calF_{\eps}(t)))}$ and $ \| \lvert g_{\eps}-q_{\eps}\rvert^{1/p} \omega_{\eps}^{-} \|_{\mathcal{L}^{p}(0,T;L^{p}(\pS^i_{\eps}(t)))}$ are uniformly bounded respect to the parameter $ \eps $. 
\end{Lemma}

\begin{proof} The proof is a direct consequence of the fact that  $ (\omega_{\eps}, \omega_{\eps}^{-}, v_{\eps}) $  are renormalized solution to the transport equation of the vorticity.		
\end{proof}

We now prove a bound for the velocity field. To do that we decompose the time $ \bbR^+ $ and  the velocity field in dependence of the size of the holes $ \calS^i $.

\subsubsection{Decomposition of the time and of the velocity field}

Given $ ( \Omega, \calS^+, \calS^-) $ we decompose the time $ \bbR^+ $ in dependence of the size of the holes $ \calS^i $. We have already introduced $ \calT^i = \{ t \in \bbR^+ $ s.t. $ r^i(t) = 0 \}$, $ \calT^i_{NP} = \{ t \in \bbR^+ $ s.t. $ r^i(t) > 0 \} $, $ \calT^i_{NP,\delta} = \{ t \in \bbR^+ $ s.t. $ r^i(t) \geq \delta \}$. We now define the ``transition times'' $ \calT^i_{TR,\delta} = \{ t \in \bbR^+ $ s.t. $ 0 < r^i(t) < \delta \} $ and $ \calT_{TR,\delta} = \calT^+_{TR,\delta} \cup \calT^-_{TR,\delta}$. Finally the ``no-transition times'' are $ \calT_{NTR,\delta} = \bbR^+ \setminus \calT_{TR,\delta} $. All this informations are resumed in Figure \ref{Tab}, where $ A = \calT^+ \cap \calT^- $, $ B^i = \calT^i_{TR,\delta} \cap \calT^{{}^{op}i} $, $ C^i = \calT^i_{NP,\delta} \cap \calT^{{}^{op}i} $, $ D = \calT^+_{TR,\delta} \cap \calT^-_{TR,\delta} $, $ E^i = \calT^{{}^{op}i}_{TR,\delta} \cap \calT^{i}_{NP,\delta} $, $ F = \calT^{+}_{NP,\delta} \cap \calT^{-}_{NP,\delta} $ and for $ i \in \{+, - \}$, we denote $ {}^{op}i $ the unique element of $ \{+,-\} \setminus \{i\}$

\begin{figure}
	\begin{tabular}{|c|c|c|c|c|} 
	\hline 
		\multicolumn{2}{|c|}{\multirow{2}{*}{}} & \multirow{2}{*}{$ \calT^+ \Leftrightarrow r^+ = 0  $ } & \multicolumn{2}{|c|}{$ \calT^+_{NP} \Leftrightarrow  r^+ > 0 $} \\ \cline{4-1} \cline{5-1} \multicolumn{2}{|c|}{} &  & $ \calT^+_{TR,\delta} \Leftrightarrow  0 < r^+ < \delta $ & $ \calT^+_{NP,\delta} \Leftrightarrow  r^+ \geq  \delta $ \\
		\hline
		\multicolumn{2}{|c|}{$ \calT^- \Leftrightarrow  r^- = 0 $} & $A $ & $ B^+ $ & $ C^+$ \\
		\hline
		\multirow{2}{*}{$ \calT^-_{NP} \Leftrightarrow  r^- > 0 $} & $ \calT^-_{TR,\delta} \Leftrightarrow  0 < r^- < \delta $ & $ B^- $ & $ D $ & $ E^+ $  \\ \cline{2-1} \cline{3 -1} \cline{4-1} \cline{5-1} & $ \calT^-_{NP,\delta} \Leftrightarrow  r^- \geq  \delta $ &  $ C^-$ & $ E^- $ &  $ F $ \\
		\hline 
		\end{tabular}	
		\caption{Table of times.}	
	\label{Tab}
\end{figure}

In the non transition times $  \calT_{NTR,\delta} = A \cup C^+ \cup C^- \cup F $ three situations can occur  
\begin{itemize}

\item[A.] both the holes are points, in particular $ t \in \calT^{+}\cap \calT^{-} = A $.

\item[C.] one hole is a point the other has size greater or equal to $  \delta $. In this case for $ i \in \{+, - \}$, we denote $ {}^{op}i $ the unique element of $ \{+,-\} \setminus \{i\}$ and $ t \in \bigcup_{i \in\{+,-\}} \calT^{i} \cap \calT^{{}^{op}i}_{NP, \delta} = C^+ \cup C^- = \calT_{PNP, \delta} $. 

\item[F.] both the holes has size greater of equal to $  \delta $, in particular $ t \in \calT^{+}_{NP,\delta} \cap \calT^{-}_{NP,  \delta} = F $.

\end{itemize}

We are now able to rewrite the velocity field $ v_{\eps} $ in dependence of the time interval.

\paragraph{Decomposition of $ v_{\eps} $ in $ A $, $ B^+$, $ B^-$ and $ D $.}

 For $ t \in A \cup B^+ \cup B^- \cup D $ we write 
\begin{equation}
\label{12}
v_{\eps} = \calK^0_{\Omega}[\omega_{\eps}]  + L_{\eps}  + \calB_{\Omega}[-L_{\eps}\cdot n] + w^{1}_{\eps},
\end{equation}
with 
\begin{align}
\label{112}
L_{\eps}(t,x) = & \,  \sum_{i \in \{+,-\}} \left( \oint_{\pS^{i}_{\eps}(t)} g_{\eps}^i \right) \frac{x-h^{i}(t)}{2\pi \lvert x-h^{i}(t)\rvert^2} \\ & \, + \sum_{i \in \{+, -\}} \left[ \calC_{i, \eps}^{in} - \int_{0}^{t}\oint_{\pS^{i}(t)} \omega^{i}_{\eps} (g^{i}_{\eps} -q_{\eps}) \right] \frac{(x-h^i(t))^{\perp}}{2\pi \lvert x-h^{i}(t)\rvert ^2}. \nonumber
\end{align}
The vector field $ \calB_{\Omega}[-L_{\eps}\cdot n] $ satisfies
\begin{gather*}
\div \left( \calB_{\Omega}[-L_{\eps}\cdot n]  \right) = 0 \text{ and } \curl \left( \calB_{\Omega}[-L_{\eps}\cdot n] \right) = 0 \text{ in } \Omega \\  \quad \text{ and }  \quad \calB_{\Omega}[-L_{\eps}\cdot n] \cdot n = - L_{\eps} \cdot n \text{ on } \partial \Omega.  
\end{gather*}
Finally the vector field $ w_{\eps} $ satisfies  
\begin{gather*}
\div w^{1}_{\eps} = 0 \text{ and } \curl w^{1}_{\eps} = 0 \text{ in } \calF_{\eps}(t), \quad \\ w^{1}_{\eps} \cdot n =g_{\eps} - (\calK^0_{\Omega}[\omega_{\eps}] +L_{\eps}+ \calB_{\Omega}[-L_{\eps}\cdot n]) \cdot n \text{ on } \partial \calF_{\eps}(t)  \\ \quad \text{ and } \quad \oint_{\pS_{\eps}^{i}(t)}  w^{1}_{\eps} \cdot \tau = 0. 
\end{gather*}

\paragraph{Decomposition of $ v_{\eps} $ in $ C^+$, $ C^-$, $ E^+ $ and $ E^- $.}

For $ i \in \{+,-\}$ and  $ t \in C^i \cup E^i $, we set $ \tilde{\calF}(t) = \Omega \setminus \overline{\calS^i(t) } $. Notice that for $ t \in C^i $ we have $ \tilde{\calF}(t) = \calF(t) $. Then the velocity field 
\begin{equation}
\label{13}
v_{\eps} = \calK_{\tilde{\calF}(t)}^{0}[\omega_{\eps}] + L^{{}^{op}i}_{\eps} + \calB_{\tilde{\calF}(t)}[g_{\eps}-L^{{}^{op}i}_{\eps}\cdot n] + w^{2}_{\eps}
\end{equation}
where
\begin{align*}
L^{{}^{op}i}_{\eps}(t,x) = & \, \left( \oint_{\pS^{{}^{op}i}_{\eps}(t)} g_{\eps} \right) \frac{x-h^{{}^{op}i}(t)}{2\pi \lvert x-h^{{}^{op}i}(t)\rvert^2} \\ & \, + \left[ \calC_{{}^{op}i, \eps}^{in} - \int_{0}^{t}\oint_{\pS^{{}^{op}i}(t)} \omega^{{}^{op}i}_{\eps} (g_{\eps} -q_{\eps}) \right] \frac{(x-h^{{}^{op}i}(t))^{\perp}}{2\pi \lvert x-h^{{}^{op}i}(t) \rvert^2}.
\end{align*}
The velocity field $ \calB_{\tilde{\calF}(t)}[g_{\eps}-L^{{}^{op}i}_{\eps}\cdot n] $ satisfies
\begin{gather*}
\div \calB_{\tilde{\calF}(t)}[g_{\eps}-L^{{}^{op}i}_{\eps}\cdot n] = 0, \quad \curl \calB_{\tilde{\calF}}[g_{\eps}-L^{{}^{op}i}_{\eps}\cdot n] = 0  \text{ in } \tilde{\calF}(t), \quad \\  \calB_{\tilde{\calF}(t)}[g_{\eps}-L^{{}^{op}i}_{\eps}\cdot n] \cdot n = g_{\eps}- L_{\eps}\cdot n \text{ on } \partial \tilde{\calF}(t) \\ \quad \text{ and } \quad \int_{\pS^{i}(t)} \calB_{\tilde{\calF}}[g_{\eps}-L^{{}^{op}i}_{\eps} \cdot n] \cdot \tau = \calC_{i, \eps}(t).
\end{gather*}
and finally 
\begin{gather*}
\div w^{2}_{\eps} = 0 \text{ and } \curl w^{2}_{\eps} = 0 \text{ in } \calF_{\eps}(t), \quad w^{2}_{\eps} \cdot n = 0 \text{ on } \pS^{i}(t) \\
w^{2}_{\eps} \cdot n = g_{\eps} - (\calK^0_{\tilde{\calF}(t)}[\omega_{\eps}]+ L^{{}^{op}i}_{\eps} + \calB_{\tilde{\calF(t)}}[g_{\eps}-L^{{}^{op}i}_{\eps}\cdot n]) \cdot n \text{ on } \partial \calS^{^{{}^{op}i}}_{\eps}(t)\\  \quad \text{ and } \quad \oint_{\pS_{\eps}^{j}(t)}  w^{1}_{\eps} \cdot \tau = 0, 
\end{gather*}
for $ j = \{+,-\}$.

\paragraph{Decomposition of $ v_{\eps} $ in $F$. }

For $ t \in \calT^{+}_{NP,\delta} \cap \calT^{-}_{NP, \delta} = F $ and small enough $ \eps $ the velocity field 
\begin{equation}
\label{11}
v_{\eps} = \calK_{\calF(t)}^{0}[\omega_{\eps}] + \calB_{\calF(t)}[g_{\eps}],
\end{equation}
where $ \calK_{\calF(t)}^{0} $ is the Biot-Savart operator with $ 0 $ component on the boundaries and $ 0 $ circulation on $ \pS^i(t) $ and $ \calB_{\calF(t)} $ is the solution of
\begin{gather*}
\div \calB_{\calF(t)}[g_{\eps}] = 0, \quad \curl \calB_{\calF(t)}[g_{\eps}] = 0  \text{ in } \calF(t), \quad \calB_{\calF(t)}[g_{\eps}] \cdot n = g_{\eps} \text{ on } \partial \calF(t) \\ \quad \text{ and } \quad \int_{\pS^{i}(t)} \calB_{\calF(t)}[g_{\eps}] \cdot \tau = \calC_{i, \eps}(t).
\end{gather*}

In figure \ref{Tab:2}, we resume the decomposition of the velocity field.

\begin{figure}
	\begin{tabular}{|c|c|c|c|c|} 
		\hline
		\multicolumn{2}{|c|}{\multirow{2}{*}{}} & \multirow{2}{*}{$ \calT^+ \Leftrightarrow r^+ = 0  $ } & \multicolumn{2}{|c|}{$ \calT^+_{NP} \Leftrightarrow  r^+ > 0 $} \\ \cline{4-1} \cline{5-1} \multicolumn{2}{|c|}{} &  & $ \calT^+_{TR,\delta} \Leftrightarrow  0 < r^+ < \delta $ & $ \calT^+_{NP,\delta} \Leftrightarrow  r^+ \geq  \delta $ \\
		\hline
		\multicolumn{2}{|c|}{$ \calT^- \Leftrightarrow  r^- = 0 $} &  \eqref{12}  &  \eqref{12}  & \eqref{13} \\
		\hline
		\multirow{2}{*}{$ \calT^-_{NP} \Leftrightarrow  r^- > 0 $} & $ \calT^-_{TR,\delta} \Leftrightarrow  0 < r^- < \delta $ & $ \eqref{12} $ & $ \eqref{12} $ & \eqref{13}  \\ \cline{2-1} \cline{3 -1} \cline{4-1} \cline{5-1} & $ \calT^-_{NP,\delta} \Leftrightarrow  r^- \geq  \delta $ &   \eqref{13}  &  \eqref{13}  &  \eqref{11} \\
		\hline 
	\end{tabular}
	\caption{Decomposition of the velocity field.}	
	\label{Tab:2}
\end{figure}

Now we will show that the velocity field $ v_{\eps} $ is uniformly bounded in $ \calL_{loc}^{r}(\bbR^{+}; L^{q}(\calF(t)))  $ and if we restrict to the ``transition times'' $ \calT_{TR;\delta} $ then $ \calL_{loc}^{r}(\calT_{TR;\delta}; L^{q}(\calF(t)))   $ norm of $ v_{\eps} $ converges to zero uniformly in $ \eps $ as $ \delta $ goes to zero.

\begin{proposition}
\label{prop:A:Priori:est}

Let $ \delta $ small enough, it holds that $ v_{\eps} $ is uniformly bounded in $ \calL_{loc}^{r}(\bbR^{+}; L^{q}(\calF(t)))$. Moreover $ w_{\eps}^{1} \longrightarrow 0 $ in $ \calL^{r}_{loc}(\calT^{+}\cap \calT^{-}; L^{q}(\calF(t))$, $ w_{\eps}^{2} \longrightarrow 0 $ in $ \calL^{r}_{loc}(\calT_{PNP,\delta}; L^{q}(\calF(t))$ as $ \eps$ converges to zero and 
\begin{equation*}
\| v_{\eps}\|_{\calL^{r}_{loc}(\calT_{TR,\delta}, L^{q}(\calF(t)))} \longrightarrow 0 
\quad \text{ as } \delta \longrightarrow 0 
\end{equation*}
uniformly in $ \eps $.

\end{proposition}

\begin{proof} Fix a time $ T > 0 $, in the following we prove the estimates in the compact time interval $ [0,T] $. To show the uniform bound we show the bound for the right hand sides of \eqref{11}-\eqref{12}-\eqref{13}. Fix $ 2\delta \leq \kappa_{T} $ with $ \kappa_{T} $ from Lemma \ref{lemma:ref:meth}.
	
\paragraph{Case $ t \in F $.} For $ t \in \calT^{+}_{NP,\delta} \cap \calT^{-}_{NP,\delta} \cap [0,T] $ the decomposition \eqref{11} holds, in particular
\begin{align*}
\|v_{\eps}\|_{L^{q}(\calF(t))} = & \,  \|\calK_{\calF(t)}^{0}[\omega_{\eps}] + \calB_{\calF(t)}[g_{\eps}]\|_{L^{q}(\calF(t))} \\ \leq  & \, C \left(\|\omega_{\eps}\|_{L^{q}(\calF(t))} +\| g_{\eps}\|_{L^{q}(\partial \calF(t))} + \lvert \calC_{+,\eps} \rvert + \lvert \calC_{-,\eps} \rvert \right)
\end{align*}
and after taking the $L^{r}$-norm in time and using Lemma \ref{firs_apriori_est} and the hypothesis on $ g_{\eps} $, we deduce the desired estimates in this zone. 

\paragraph{Case $ t \in A $}

For $ t \in \calT^{+} \cap \calT^{-}  $ and for $ \eps < \delta $, we use the decomposition \eqref{12} in the following way. For $ \calK^{0}_{\Omega}[\omega_{\eps}] $ the estimates is clear, the uniform bound of $ L_{\eps} $ comes from the fact that it is the sum of the products of an $ L^{r} $ in time functions for example $ \int_{\pS^i_{\eps}(t)} g_{\eps} $ and vector fields of the type $ (x-h^i(t))/\lvert x-h^i(t)\rvert^2 $ which are bounded in $ L^{q}(\Omega) $ uniformly in $ h^{i}(t) \in \Omega $. The vector field $ \calB_{\calF(t)}[-L_{\eps} \cdot n] $ is associated to $ L_{\eps} $ throw its values on $ \partial \Omega $, in particular they are bounded in $L^{r}_{loc}(L^{q}(\partial \Omega)) $ because $ h^i(t) \in K $ a set compactly contained in $\Omega$. We are left with the uniform bound for $ w_{\eps}^{1} $. In this case we actually prove that  $ \omega_{\eps}^{1} \longrightarrow 0 $ in $ \calL^{r}_{loc}(\calT^{+}\cap \calT^{-}; L^{q}(\calF(t))$. To prove this convergence we notice that for $ \eps < 2 \delta $ we are in the setting to apply Lemma \ref{lemma:ref:meth}, in particular it holds
\begin{align*}
\| w_{\eps}^1\|_{L^{q}(\calF_{\eps}(t))} \leq & \, C \eps^{1/q}\| g^{\eps} - (\calK^0_{\calF(t)}[\omega_{\eps}]  + L_{\eps}  + \calB_{\calF(t)}[-L_{\eps} \cdot n] )\cdot n \|_{L^{q}(\partial \calF_{\eps}(t))}  \\
\leq & \, C \eps^{1/q}\big(\| g^{\eps} \|_{L^{q}(\partial \calF_{\eps}(t))} + \| \calK^0_{\Omega}[\omega_{\eps}]\cdot n\|_{L^{q}(\partial \calF_{\eps}(t))} \\ &  \quad\quad\quad\quad + \| L_{\eps}\cdot n\|_{L^{q}(\partial \calF_{\eps}(t))}  + \|\calB_{\calF(t)}[-L_{\eps} \cdot n] \cdot n \|_{L^{q}(\partial \calF_{\eps}(t))}\big)
\end{align*}  
and the right hand side tends to zero in $ L^{r} $, in fact $ \eps^{1/q} \| g^{\eps} \|_{L^{q}(\partial \calF_{\eps}(t))} $ tends to zero by hypothesis \eqref{3}, $ \calK^0_{\Omega}[\omega_{\eps}] $ and $ \calB_{\calF(t)}[-L_{\eps} \cdot n] $ are uniformly bounded in $ L^{\infty} $ in any compact subset of $ \Omega $  which implies that $ \eps^{1/q}( \| \calK^0_{\Omega}[\omega_{\eps}]\cdot n\|_{L^{q}(\partial \calF_{\eps}(t))} + \|\calB_{\calF(t)}[-L_{\eps} \cdot n] \cdot n \|_{L^{q}(\partial \calF_{\eps}(t))}) $ tends to zero and the estimates for $ L_{\eps} $ follows from 
\begin{equation*}
\eps^{1/q}\left\| \frac{x-h^i(t)}{2\pi \lvert x-h^i(t) \rvert^2}\right\|_{L^{q}(\pS^{i}_{\eps}(t))} = \eps^{1/q}\eps^{1/q -1}\left\| \frac{x-h^i(t)}{2\pi \lvert x-h^i(t) \rvert^2}\right\|_{L^{q}(\pS^{i}_{1}(t))} \longrightarrow 0.
\end{equation*}

\paragraph{Case $ t \in C^+ \cup C^-$.}

For $ t \in \calT_{PNP, \delta} = C^+ \cup C^-  $ we assume, without loss of generality, that the i-th hole is a point and we assume $ \eps $ sufficiently small to apply the second part of Lemma \ref{lemma:ref:meth}. As usual we rewrite $ v_{\eps }$ with the help of \eqref{13}. The estimate of the right hand side of \eqref{13} can be performed in the similar way to the case $ t \in \calT^{+} \cap \calT^{-} = A $, in particular we use the estimates
\begin{equation*}
\| \calB_{\calF(t)}[g_{\eps}-L^i_{\eps}\cdot n] \|_{L^{q}(\calF(t))} \leq C \| g_{\eps}-L^i_{\eps}\cdot n \|_{L^{q}(\partial \calF(t))}
\end{equation*}
and for $ w^{2}_{\eps} $ we apply the second part of Lemma \ref{lemma:ref:meth}.

\paragraph{Case $ t \in B^+ \cup B^- \cup C \cup E^+ \cup E^- $}

For $ t \in B^+ \cup B^- \cup C$, we use the decomposition \eqref{12} and note that we are in the setting to use Lemma \ref{lemma:ref:meth}. We estimate the right hand side of \eqref{12} as follow.
\begin{align*}
\| \calK^{0}_{\Omega}[\omega_{\eps}]\|_{L^{r}(\calT_{TR, \delta}, L^{q}(\Omega)) }   \leq & \, C \| \omega_{\eps} \|_{L^{r}(\calT_{TR, \delta}, L^{q}(\Omega)) } \\ \leq & \, C \left(Leb^{1}(\calT_{TR, \delta})\right)^{1-1/r} \| \omega_{\eps} \|_{L^{\infty}(\calT_{TR, \delta}, L^{q}(\Omega)) }. 
\end{align*}
Note that the last term in the series of inequalities is bounded uniformly in $ \eps $ due to Lemma \ref{firs_apriori_est} and $ Leb^{1}(\calT_{TR, \delta}) $, which denotes the Lebesgue measure of $\calT_{TR, \delta} $, converges to zero. To estimate $ L_{\eps} $ and $ \calB_{\calF(t)}[-L_{\eps} \cdot n] $ we notice that 
\begin{align*}
\| L_{\eps}+ & \calB_{\calF(t)}[-L_{\eps} \cdot n] \|_{L^{r}(\calT_{TR, \delta}; L^{q}(\Omega))} \\ \leq & \, C \sum_{i \in \{+,-\} } \left(\left\| \oint_{\pS^{i}_{\eps}(t)} g_{\eps}\right\|_{L^{r}(\calT_{TR, \delta})} + \left\| \calC_{i, \eps}^{in} - \int_{0}^{t}\oint_{\pS^{i}(t)} \omega^{i}_{\eps} (g_{\eps}-q_{\eps} )\right\|_{L^{r}(\calT_{TR, \delta})} \right)  \\
\leq & \, C\sum_{i \in \{+,-\} } \left(\left\|(r_{\eps}^{i}(t))^{1/p}\| g_{\eps}\|_{L^{q}(\pS^{i}_{\eps}(t))}\right\|_{L^{r}(\calT_{TR, \delta})} \right) \\
& \, + C \left(Leb^{1}(\calT_{TR, \delta})\right)^{1/r}\left[\lvert \calC_{i, \eps}^{in} \rvert \right] \\  & \, + C \left(Leb^{1}(\calT_{TR, \delta})\right)^{1/r} \left[ \lvert \int_0^T \oint_{\pS^{i}_{\eps}(t)} \lvert \omega^i \rvert^{p} \lvert g_{\eps}-q_{\eps}\rvert \rvert^{1/p} \lvert \int_0^{T} \oint_{\pS^{i}_{\eps}(t)} \lvert g_{\eps}-q_{\eps}\rvert  \rvert^{1/q} \right], 
\end{align*}
which converges uniformly in $ \eps $ to zero due to hypothesis \eqref{5}, Lemma \ref{firs_apriori_est} and the fact that the measure of $ \calT_{TR, \delta} $ converges to zero.

We are left with the reminder $ w_{\eps} $ for which the uniform convergence to zero follow from Lemma \ref{lemma:ref:meth}, together with Lemma \ref{firs_apriori_est} and hypothesis \eqref{5} applied in the same style as above to estimate 
\begin{equation*}
\sum_{i\in\{+,-\}} r_{\eps}^i(t)^{1/q} \left\| g_{\eps} - \left( \calK^{0}_{\Omega}[\omega_{\eps}] + L_{\eps} + \calB_{\calF(t)}[-L_{\eps} \cdot n] \right)\cdot n \right\|_{L^{q}(\pS^{i}_{\eps}(t))}.
\end{equation*} 
For $ t \in E^+ \cup E^- $ the estimates are similar.

\end{proof}

We now present an a priori bound for the time derivative of the vorticity.

\begin{Lemma}
\label{lem:unif:est:dt}
Let $ (\omega_{\eps}, \omega_{\eps}^{-}, v_{\eps}) $ a weak solution of \eqref{app:judv:sys} satisfying the hypothesis of Theorem \ref{app:theo:theo} and let $ T > 0 $. Then $ \| \partial_t \omega_{\eps} \|_{L^{l}(0,T; W^{-2,s}(\bbR^{2}))} $ is uniformly bounded in $ \eps $ for $ l, s $ big enough.

\end{Lemma} 

\begin{proof}

This is a consequence of equation \eqref{app:judv:sys}, Lemma \ref{firs_apriori_est} and Proposition \ref{prop:A:Priori:est}. The triple $ (\omega_{\eps}, \omega_{\eps}^{-}, v_{\eps}) $ is a weak solution of the system \eqref{app:judv:sys}, we deduce that 

\begin{align*}
\lvert \left \langle \partial_t \omega_{\varepsilon}, \varphi \right\rangle \rvert
&=\lvert \int_{0}^{T}\int_{\mathcal{F}_{\varepsilon}} \omega_{\varepsilon} \partial_t\varphi dx dt -\int_{\mathcal{F}_{\varepsilon}(t)} \omega_{\varepsilon}(T,.) \varphi(T,.) dx +\int_{\mathcal{F}_{\varepsilon}(t)} \omega_{\varepsilon}^{in} \varphi(0,.) dx \rvert \\ 
\leq &  \lvert \int_{0}^{T} \int_{\mathcal{F}_{\varepsilon}(t)} v_{\varepsilon}\cdot\nabla \varphi \omega_{\varepsilon} dx dt \rvert + \sum_{i \in \{+,-\}} \lvert \int_{0}^{T}\oint_{\pS^{i}_{\varepsilon}(t)} (g_{\varepsilon}-q_{\eps}) \omega^{i}_{\varepsilon} \varphi ds dt \rvert,
\end{align*}
The terms in the right hand side above can be estimated as follows. 
\begin{gather*}
\lvert \int_{0}^{T} \int_{\mathcal{F}_{\varepsilon}(t)} v_{\varepsilon}\cdot \nabla \varphi \omega_{\varepsilon} \rvert \leq \\ C  \| \omega_{\varepsilon}\|_{{\calL^{\infty}(0,T;L^{p}(\calF_{\eps}(t)))}} \|v_{\varepsilon}\|_{\calL^{r}(0,T; L^{q}(\calF_{\eps}(t)))}\|\nabla \varphi \|_{\calL^{r'}(0,T;L^{\infty}(\calF_{\eps}(t)))}  ,
\\
\text{ and }
\\  
\lvert \int_{0}^{T}\oint_{\pS^{i}_{\varepsilon}(t)} (g_{\varepsilon}-q_{\eps})\omega^i_{\varepsilon}\varphi \rvert \leq  \\ \| \lvert g_{\varepsilon}-q_{\eps}\rvert^{1/p} \omega^i_{\varepsilon}\|_{L^{p}([0,T]\times \pS^{i}_{\varepsilon}(t))} \| g_{\varepsilon}-q_{\eps}\|^{1/p'}_{\calL^{r}(0,T ; L^{1}(\pS^{i}_{\varepsilon}(t) ))} \| \varphi \|_{\calL^{p'r'}(0,T; L^{\infty}(\pS^{i}(t)) )}.
\end{gather*}
The uniform boundedness of the above quantities follows from Lemma \ref{firs_apriori_est}, Proposition \ref{prop:A:Priori:est} and hypothesis \eqref{2}-\eqref{3}-\eqref{5}.

\end{proof}

\begin{remark}

Let us notice that in contrast to the proof of Theorem  \ref{exi:Lp:Nempty}, the uniform bound of the time derivative of the vorticity is much easier because the term $ \omega_{\eps} v_{\eps} $ is uniformly bounded in $ L^1 $ which was not the case in the proof of Theorem  \ref{exi:Lp:Nempty}. 
	
\end{remark}

\subsection{Weak and strong convergence of the vorticities}

From the uniform bounds of Lemma \ref{firs_apriori_est} we derive weak convergences of the vorticity in the following sense.

\begin{proposition}
 
Let $ (\omega_{\eps}, \omega_{\eps}^{-}, v_{\eps}) $ a weak solution of \eqref{app:judv:sys} satisfying the hypothesis of Theorem \ref{app:theo:theo}. Then up to subsequence $ \omega_{\eps} $ converges weakly-star to $ \omega $ in $ \calL^{\infty}_{loc}(\bbR^{+}; L^{p}(\calF(t))) $ and strongly in $ C^0_{loc}(\bbR^+;L^p-w(\calF(t)))$ and for any small enough $ \delta$
the function $ (g_{\eps}-q_{\eps})^{1/p} \omega^{-}_{\eps} $ converges weakly to $ (g-q)^{1/p} \omega^{-} $ in $ \calL^{p}_{loc}(\calT_{NP,\delta}^{-}; L^{p}(\pS^-(t) ) $.

\end{proposition}

\begin{proof}

The weak-star convergence of $ \omega_{\eps} $ is a direct consequence of Lemma \ref{firs_apriori_est}, after noticing that
\begin{equation*}
\calL^{\infty}_{loc}(\bbR^{+};L^{p}(\calF(t))) \cong \left(\calL^{1}_{loc}(\bbR^{+};L^{p/(p-1)}(\calF(t)))\right)^{*}.
\end{equation*}
The strong convergence in $ C^0_{loc}(\bbR^+;L^p-w(\calF(t)))$ is a consequence of Lemma \ref{firs_apriori_est} and \ref{lem:unif:est:dt} together with Lemma C.1 of \cite{lions}.

\end{proof}

\subsection{Strong convergence for the velocity}

This subsection is dedicated to the proof of strong convergence for the the velocity field $ v_{\eps}$. 

\begin{proposition}
\label{strong:con:vel}
Let $ (\omega_{\eps}, \omega_{\eps}^{-}, v_{\eps}) $ a weak solution of \eqref{app:judv:sys} satisfying the hypothesis of Theorem \ref{app:theo:theo}. Then up to subsequence $ v_{\eps} $ converges strongly to $v$ in $ \calL^{r}_{loc}(\bbR^{+};L^{q}(\calF(t))) $. Moreover 
\begin{align*}
v = & \,  \calJ_{\calF(t)} \left[\mu \mathds{1}_{\calT^{+}}\delta_{x_{+}(t)} - \left(\oint_{\pS^{+}(t)} g \mathds{1}_{\calT^{+}_{NP}} + \mu \mathds{1}_{\calT^{+}} \right) \mathds{1}_{\calT^{-}}\delta_{x_{-}(t)} \right] \\ & \, + \calK_{\calF(t)} \left[ \omega + \calC_{+} \mathds{1}_{\calT^{+}} \delta_{x_{+}(t)} + \calC_{-} \mathds{1}_{\calT^{-}} \delta_{x_{-}(t)} \right].
\end{align*}

\end{proposition}

\begin{proof}

Proposition \ref{prop:A:Priori:est} shows that 
\begin{equation*}
\| v_{\eps}\|_{\calL^{r}_{loc}(\calT_{TR,\delta}, L^{q}(\calF(t)))} \longrightarrow 0 
\quad \text{ as } \delta \longrightarrow 0 
\end{equation*}
uniformly in $ \eps $. This implies that it is enough to show the convergence of $ v_{\eps} $ in the non-transition zone $ \calT_{NTR,\delta} $ for any small enough fixed $ \delta $. We divide the non-transition zone in three parts as in the beginning of Subsection \ref{unif:bound}. 

\paragraph{Case $ t \in F $.}

In $ \calT^{+}_{NP,\delta} \cap \calT^{-}_{NP,\delta} = F $, the velocity field $ v_{\eps} = \calK_{\calF(t)}^{0}[\omega_{\eps}] + \calB_{\calF(t)}[g_{\eps}] $, moreover   
\begin{equation*}
\left\| \calB_{\calF(t)}[g_{\eps}] - \calB_{\calF(t)}[g]   \right\|_{L^{q}(\calF(t))} = \left\| \calB_{\calF(t)}[g_{\eps} - g]   \right\|_{L^{q}(\calF(t))} \leq C \| g_{\eps}-g \|_{L^{q}(\partial \calF(t))} \to 0,
\end{equation*} 
in $ L^{r}_{loc}(\calT^{+}_{NP,\delta} \cap \calT^{-}_{NP,\delta} )$ by hypothesis \eqref{2}. 

Regarding $ \calK_{\calF(t)}^{0}[\omega_{\eps}] $, the strong convergence in $ \calL^r_{loc}(\bbR^+;L^{q}(\calF(t)))$ follows from the  $ C^0_{loc}(\bbR^+;L^p-w(\calF(t)))$ convergence of $ \omega_{\eps}$ together with Lemma 6.4 of \cite{NS}.

\paragraph{Case $ t \in A $.}

In $ \calT^{+} \cap \calT^{-} $, the velocity field $ v_{\eps} = \calK^0_{\Omega}[\omega_{\eps}]  + L_{\eps}  + \calB_{\Omega}[-L_{\eps}\cdot n] + w^{1}_{\eps} $. We have already shown that $  \calK^0_{\Omega}[\omega_{\eps}] \to  \calK^0_{\Omega}[\omega]$ and $ w^{1}_{\eps} \to 0 $ in $ L^{r}_{loc}(\calT^{+} \cap \calT^{-}, L^{q}(\Omega))$. It remains to show the strong convergence of $ L_{\eps} $ and $ \calB_{\Omega}[-L_{\eps}\cdot n] $, in particular the strong convergence of $ L_{\eps} $ implies also the one of $ \calB_{\Omega}[-L_{\eps}\cdot n] $. The vector field $ L_{\eps} $ is defined in \eqref{112}. From hypothesis \eqref{4}, we have    
\begin{equation*}
\sum_{i \in \{+,-\}} \left( \oint_{\pS^{i}_{\eps}(t)} g_{\eps} \right) \frac{x-h^{i}(t)}{2\pi \lvert x-h^{i}(t) \rvert^2} \longrightarrow  \sum_{i \in \{+,-\}} i \mu \frac{x-h^{i}(t)}{2\pi \lvert x-h^{i}(t)\rvert^2}
\end{equation*}
in $ L^{r}_{loc}(\calT^{+} \cap \calT^{-}, L^{q}(\Omega))$, it remains to show that
\begin{equation*}
\calC_{+, \eps}^{in} - \int_{0}^{t}\oint_{\pS^{+}(t)} \omega^{+}_{\eps} (g_{\eps}-q_{\eps}) \longrightarrow \calC_{+}^{in} - \int_{0}^{t} \left( \oint_{\partial S^{+}(t)} \omega^{+}(g - q)\mathds{1}_{\calT^{+}_{NP}} + j \mathds{1}_{\calT^{+}}  \right)
\end{equation*}
and
\begin{gather*}
\calC_{-, \eps}^{in} - \int_{0}^{t}\oint_{\pS^{-}_{\eps}(t)} \omega^{-}_{\eps} (g_{\eps}-q_{\eps}) \longrightarrow  \\ \calC_{-}^{in} + \int_{\calF(t)} \omega - \int_{\calF_0} \omega^{in} + \int_{0}^{t} \left( \oint_{\partial S^{+}(t)} \omega^{+}(g - q)\mathds{1}_{\calT^{+}_{NP}} + j \mathds{1}_{\calT^{+}}  \right)
\end{gather*}
in $ L^{r}_{loc}(\calT^{+}\cap \calT^{-}) $. Note that the second convergence follows from from the first one plus the weak formulation of the first equation of \eqref{app:judv:sys} tested with the constant function $ 1 $. For $ \mathfrak{d} > 0 $, we have 
\begin{align*}
\int_{0}^{t}\oint_{\pS^{+}(t)} \omega^{+}_{\eps} g_{\eps} = \int_{0}^{t}\Bigg(& \oint_{\pS^{+}_{\eps}(t)} \omega^{+}_{\eps} (g_{\eps}-q_{\eps}) \mathds{1}_{\calT^{+}} + \oint_{\pS^{+}_{\eps}(t)} \omega^{+}_{\eps} (g_{\eps}-q_{\eps}) \mathds{1}_{\calT^{+}_{NP, 2\mathfrak{d}}} \\ & + \oint_{\pS^{+}_{\eps}(t)} \omega^{+}_{\eps} (g_{\eps}-q_{\eps}) \mathds{1}_{\calT^{+}_{TR, 2\mathfrak{d}}} \Bigg).
\end{align*}
From hypothesis \eqref{4} and \eqref{2}, it holds
\begin{gather*}
\int_{0}^{t} \oint_{\pS^{+}_{\eps}(t)} \omega^{+}_{\eps}(g_{\eps}-q_{\eps}) \mathds{1}_{\calT^{+}} \longrightarrow \int_{0}^t j \mathds{1}_{\calT^{+}} \quad \text{ and } \quad  \\ \int_{0}^{t}  \oint_{\pS^{+}_{\eps}(t)} \omega^{+}_{\eps} (g_{\eps}-q_{\eps}) \mathds{1}_{\calT^{+}_{NP, 2\mathfrak{d}}} \longrightarrow \int_{0}^{t}  \oint_{\pS^{+}(t)} \omega^{+} (g-q) \mathds{1}_{\calT^{+}_{NP, 2\mathfrak{d}}}.
\end{gather*}
Finally we use \eqref{3}-\eqref{5} to show that the remaining term converge to zero in $ L^{r} $. It holds
\begin{align*}
\int_{\calT_{TR, 2\mathfrak{d}}} & \lvert \int_{0}^{t} \oint_{\pS^{+}_{\eps}} (g_{\eps}-q_{\eps}) \omega_{\eps}^{+} \rvert^r \\  \leq & \,   C \int_{\calT_{TR, 2\mathfrak{d}}} \left(\int_{0}^{T}\oint_{\pS^{+}_{\eps}} \lvert (g_{\eps}-q_{\eps}) \omega_{\eps}^+	\rvert^{p} \right)^{r} + \left( \int_{0}^{T} \oint_{\pS^{+}_{\eps}} \lvert g_{\eps}-q_{\eps} \rvert \right)^{r} \\
 \leq & \, C Leb^1(\calT_{TR, 2\mathfrak{d}})\left(\int_{0}^{T}\oint_{\pS^{+}_{\eps}} \lvert g_{\eps}-q_{\eps} \lvert \lvert \omega_{\eps}^+ \rvert^{p} \right)^{r}  \\ & \, + C T^{(r-1)/r} \int_{\calT_{TR, 2\mathfrak{d}}} r_{\eps}^+(t)^{r/p} \left( \oint_{\pS^{+}_{\eps}} \lvert g_{\eps}-q_{\eps} \lvert^{q} \right)^{r/q},
\end{align*}
we deduce that the right hand side converge to zero as $ \mathfrak{d} $ converges to zero uniformly in $ \eps $. This allows us to conclude that the convergence holds.

\paragraph{Case $ t \in C^{+} \cup C^- $.}

In $ \calT_{PNP, \delta} $, we recall from \eqref{13} that $ v_{\eps} = \calK_{\calF(t)}^{0}[\omega_{\eps}] + L_{\eps} + \calB_{\calF(t)}[g_{\eps}-L_{\eps}\cdot n] + w^{2}_{\eps} $. We start by noticing that the convergence of $ L_{\eps} $ and $ \calB_{\calF(t)}[-L_{\eps}\cdot n] $ follows analogously to the time $ \calT^{+} \cap \calT^{-} $, then the convergence of  $ \calK_{\calF(t)}^{0}[\omega_{\eps}] $ and $ \calB_{\calF(t)}[g_{\eps}] $ is shown as for $ t \in \calT_{NP,\delta} $. The term $ w_2 $ converges to zero due to Proposition \ref{prop:A:Priori:est}. 


\end{proof}

\subsection{Passing to the limit in the weak formulation}

Consider a sequence of approximate solutions $ (\omega_{\eps}, \omega_{\eps}^{-}, v_{\eps}) $ satisfying the hypothesis of Theorem \ref{app:theo:theo}. We use the weak convergence of the vorticities and the strong one of the velocity to pass to the limit in the weak formation. Recall that a weak solution of \eqref{app:judv:sys} satisfies the integral equation
\begin{align*}
\int_{\calF_{\eps}(0)} \omega_{\eps}^{in} \varphi(0,.) + \int_{\bbR^{+}} \int_{\calF_{\eps}(t)} \omega_{\eps} \partial_t \varphi + & \, \int_{\bbR^{+}} \int_{\calF_{\eps}(t)} v_{\eps}\cdot \nabla \varphi \omega_{\eps} = \\  \int_{\bbR^{+}} \oint_{\partial \calS_{\eps}^{+}(t)} (g_{\eps}-q_{\eps}) \omega_{\eps}^{+} \varphi & \, + \int_{\bbR^{+}} \oint_{\partial \calS_{\eps}^{-}(t)} (g_{\eps}-q_{\eps}) \omega_{\eps}^{-} \varphi.
\end{align*}    
From the weak convergence of the vorticity $ \omega_{\eps}$ together with the strong convergence of the velocity proved in Proposition \ref{strong:con:vel}, we can pass to the limit in the left hand side. Regarding the right hand side the first term passes to the limit by hypothesis, in fact $ \omega_{\eps}^{+} $ is not an unknown of the problem so we need only to show that
\begin{align*}
\int_{\bbR^{+}} \oint_{\partial \calS_{\eps}^{-}(t)}& \, (g_{\eps}-q_{\eps}) \omega_{\eps}^{-} \varphi \longrightarrow   \int_{\calT^{-}_{NP}} \oint_{\pS^{-}(t)} \omega^{-}(g-q)\varphi \\ & \, + \int_{\calT^{-}} \left( \frac{d}{dt} \int \omega +  \oint_{\partial S^{+}(t)} \omega^{+}(g - q)\mathds{1}_{\calT^{+}_{NP}} + j \mathds{1}_{\calT^{+}} \right) \varphi(t,x_{-}(t))
\end{align*}
First of all for transition times we have 
\begin{align*}
& \lvert \int_{\calT^{-}_{TR,\delta}} \oint_{\partial \calS_{\eps}^{-}(t)}  (g_{\eps}-q_{\eps}) \omega_{\eps}^{-} \varphi \rvert \\ \leq & \,  \left\| (g_{\eps}-q_{\eps})^{1/p}  \omega_{\eps}^{-}  \right\|_{\calL^p_{loc}(\calT^{-}_{TR,\delta};L^{p}(\pS_{\eps}(t)))} \left\| (g_{\eps}-q_{\eps})  \right\|_{\calL^1_{loc}(\calT^{-}_{TR,\delta};L^{1}(\pS_{\eps}(t)))}^{1/q} \\
\leq & \,  \left\| (g_{\eps}-q_{\eps})^{1/p}  \omega_{\eps}^{-}  \right\|_{\calL^p_{loc}(\calT^{-}_{TR,\delta};L^{p}(\pS_{\eps}(t)))} \left\| (r_{\eps}^{-})^{1/p}\|g_{\eps}-q_{\eps} \|_{L^{q}(\partial \calS_{\eps}^{-}(t))}  \right\|_{\calL^r_{loc}(\calT^{-}_{TR,\delta})} \\
& \, \longrightarrow 0,
\end{align*}
uniformly in $ \delta $ where we use hypothesis \eqref{5}. For time in $ \calT^{-}_{NP,\delta} $, it holds
\begin{align*}
\int_{\calT^{-}_{NP,\delta}} \oint_{\partial \calS^{-}(t)}  (g_{\eps}-q_{\eps}) \omega_{\eps}^{-} \varphi \longrightarrow \int_{\calT^{-}_{NP,\delta}} \oint_{\partial \calS^{-}(t)}  (g-q) \omega^{-} \varphi,
\end{align*}
from weak convergence. We are left with the times in $ \calT^{-} $. We have 
\begin{align*}
\int_{\calT^{-}} \oint_{\partial \calS^{-}(t)}  (g_{\eps}-q_{\eps}) \omega_{\eps}^{-} \varphi = & \, 
\int_{\calT^{-}} \oint_{\partial \calS^{-}(t)}  (g_{\eps}-q_{\eps}) \omega_{\eps}^{-} \varphi(.,x_{-}(.)) \\ & \, +
\int_{\calT^{-}} \oint_{\partial \calS^{-}(t)}  (g_{\eps}-q_{\eps}) \omega_{\eps}^{-} (\varphi-\varphi(.,x_{-}(.))).
\end{align*}
The last term converge to zero, in fact
\begin{align*}
\lvert \int_{\calT^{-}} \oint_{\partial \calS^{-}(t)} &, (g_{\eps}-q_{\eps}) \omega_{\eps}^{-} (\varphi-\varphi(.,x_{-}(.))) \rvert \\ \leq & \, \left\| (g_{\eps}-q_{\eps})^{1/p}  \omega_{\eps}^{-}  \right\|_{\calL^p_{loc}(\calT^{-}_{TR,\delta};L^{p}(\pS_{\eps}(t)))} \\ & \, \times \left\| (g_{\eps}-q_{\eps}) (\varphi-\varphi(.,x_{-}(.))) \right\|_{\calL^1(\calT^{-}_{TR,\delta};L^{1}(\pS_{\eps}(t)))}^{1/q} \\
\leq & \, \left\| (g_{\eps}-q_{\eps})^{1/p}  \omega_{\eps}^{-}  \right\|_{\calL^p_{loc}(\calT^{-}_{TR,\delta};L^{p}(\pS_{\eps}(t)))} \\ & \, \times \left\| (g_{\eps}-q_{\eps})  \right\|_{\calL^1_{loc}(\calT^{-}_{TR,\delta};L^{1}(\pS_{\eps}(t)))}^{1/q} L \eps \\
& \, \longrightarrow 0,
\end{align*}
where $ L $ is the Lipschitz constant of $ \varphi $. For the remaining term we have
\begin{equation*}
\oint_{\partial \calS^{-}(t)}  (g_{\eps}-q_{\eps}) \omega_{\eps}^{-}
\end{equation*}
is uniformly bounded in $ L^{p}_{loc}(\calT^{-})$, so
\begin{equation*}
\int_{\calT^{-}} \oint_{\partial \calS^{-}(t)}  (g_{\eps}-q_{\eps}) \omega_{\eps}^{-} \varphi(.,x_{-}(.)) \longrightarrow \int_{\calT^{-}} j^{-} \varphi(.,x_{-}(.))
\end{equation*} 
up to subsequence. We have now to identify the limit. Note that 
\begin{equation*}
\frac{d}{dt} \int_{\calF(t)} \omega \in L_{loc}^{s}(\bbR^{+}),
\end{equation*} 
for some $ s > 1 $. Moreover using the equation for any $ \psi \in C^{\infty}_{c}(\text{int}(\calT^{-})) $, it holds
\begin{equation*}
\int_{\bbR^{+}} j \mathds{1}_{\calT^{-}} \psi  = - \int_{\bbR^{+}} \left( \frac{d}{dt} \int_{\calF(t)} \omega +  \oint_{\partial S^{+}(t)} \omega^{+}(g - q)\mathds{1}_{\calT^{+}_{NP}} + j \mathds{1}_{\calT^{+}} \right) \psi.
\end{equation*}
Finally $ v $ satisfies the Div-Curl system in a weak form due to its explicit expression in Proposition \ref{strong:con:vel}. 

This conclude the proof of Theorem \ref{app:theo:theo}.

\appendix

\section{Proof of Lemma \ref{H:phi:bounded}}

\label{app:H:phi:bounded}

The proof of Lemma \ref{H:phi:bounded} is based on the following estimate.

\begin{Lemma}
\label{Lemma:Green:fond:EST}	
Let $ (\Omega, \calS^, \calS^-) $ a regular and compatible geometry and  let  $ \calF(t) = \Omega \setminus \overline{\calS^+(t) \cup \calS^-(t)} $. For $ t \in [0,T] $, let $ G(t,x,y) $ the Green function in $ \calF(t) $ with Dirichlet boundary condition, i.e. $ G $ is the solution of $ - \Delta_{x} G(t,.,y) = \delta_0(x-y) $ in $ \calF(t)$ and $ G(t,.,y) = 0 $ on $ \partial \calF(t) $. Then 
\begin{equation}
\label{Green:fond:EST}
\lvert G(t,x,y) \rvert \leq M \left(1+ \lvert log \lvert x-y \rvert \rvert\right) \quad \text{and} \quad \lvert \nabla_x G(t,x,y) \rvert \leq \frac{M}{\lvert x-y \rvert},
\end{equation} 		
for any $ x, y \in \calF(t) $ and $ M $ independent of time.	
\end{Lemma}

Let us briefly show Lemma \ref{H:phi:bounded}.
\begin{proof}[Proof of Lemma \ref{H:phi:bounded}] 

The proof is a consequence of the estimates \eqref{Green:fond:EST} and the fact that $ \varphi $ is constant in any connected component of the boundary in such a way that $ \nabla \varphi $ is a multiple of the normal to boundary in any point of $ \partial \calF(t) $. We refer to \cite{ILN} for a complete proof of the result.
 	
\end{proof}
We conclude the section with the proof of Lemma \ref{Lemma:Green:fond:EST}. First of all let us recall that Lemma \ref{Lemma:Green:fond:EST} has been shown in \cite{Lic} for any fixed time and \eqref{Green:fond:EST} holds true with constant $ M = M(t) $. In the case of a domain with smooth boundary a proof is available in the appendix of \cite{ILN} and relies on the Riemann mapping Theorem. Let us now explain how to get \eqref{Green:fond:EST} with a constant independent of time.

\begin{proof}[Proof of Lemma \ref{Lemma:Green:fond:EST}]

In the study of Green's functions for the Dirichlet Laplacian, the case of dimension two is special in the sense that the fundamental solution in $ \bbR^2 $ is given by $ \Phi_2(x) = -\log \lvert x \rvert/2\pi $ while in dimension $ n \geq 3 $ is $\Phi_n(x) = \lvert x \rvert^{2-n}/w(n) $ where $ w( n ) $ is $ n(n-2) $ times the volume of the unit ball in $ \bbR^n $. For this reason often the authors restrict their proof to the case $ n \geq 3 $. 

In the case $ n = 3 $, Theorem 1 of \cite{DM} states that given $ \mathcal{O} \subset \bbR^3$  a bounded domain with $ C^{2,\alpha}$ boundary 
 and given $ G^{\mathcal{O}} $ the Green's function associated with the Dirichlet Laplacian on $ \mathcal{O} $,
 then   
\begin{equation*}
\lvert G^{\mathcal{O}}(t,\bar{x},\bar{y}) \rvert \leq \frac{M}{\lvert \bar{x}-\bar{y}\rvert} \quad \text{and} \quad \lvert \nabla_x G^{\mathcal{O}}(t,\bar{x},\bar{y}) \rvert \leq \frac{M}{\lvert \bar{x}-\bar{y}\rvert^2},
\end{equation*}
where $ \bar{x}, \bar{y} \in \mathcal{O} \subset \bbR^3 $. The proof is based on interior and boundary estimates. This implies that the proof can be extended to Green function associated with the three dimensional domain $ \calF(t) \times S^1 $ where $ S^1 $ is the circle with length $ 1 $, in fact the boundary is in $ C^{1,\alpha}_{loc}(\bbR^+;C^{2,\alpha}) $ so the constant $ M $ can be chosen  independent of the time.

Let us now recall that from the proof of Theorem 13 part \textit{ii.} of \cite{AL}
the Green's function $ G(t,.,.) $ associated with the two dimensional domain $ \calF(t) $ can be obtain from $ G^{\calF(t) \times S^1 } $ through the formula
\begin{equation*}
G(t,x,y) = \int_{0}^1 G^{\calF(t) \times S^1 }((x,0),(y,y_{3})) \, dy_{3}.
\end{equation*} 
We have
\begin{align*}
\lvert \nabla_{x} G(t,x,y)\rvert = & \, \lvert \int_{0}^1  G^{\calF(t) \times S^1 }((x,0),(y,y_{3})) \, dy_{3} \rvert \leq \int_{0}^1 \lvert G^{\calF(t) \times S^1 }((x,0),(y,y_{3}))\rvert \, dy_{3} \\
\leq & \,  \int_{0}^1 \frac{M}{\lvert x-y \rvert^2 + \lvert y_3 \rvert^2} \, dy_3 \leq \frac{1}{\lvert x-y \rvert} \int_{0}^{\frac{1}{\lvert x-y \rvert}} \frac{M}{1+\lvert s \rvert^2} \, ds \leq C \frac{M}{\lvert x-y \rvert}.
\end{align*}
For the estimates of $ \lvert G \rvert $, we follow the proof of Theorem 13 part \textit{ii.} of \cite{AL}.
\end{proof}  

\paragraph{Acknowledgements.} We want to thank Prof. F. Sueur for very fruitful discussions and comments. 
Marco Bravin  is supported by ERC-2014-ADG project HADE Id. 669689 (European Research Council), MINECO grant BERC 2018-2021 and  SEV-2017-0718 (Spain).


\begin{thebibliography}{50}




\bibitem{Ale} Alekseev, G. V. (1976). On solvability of the nonhomogeneous boundary value problem for two-dimensional nonsteady equations of ideal fluid dynamics. Dinamika Sploshnoy Sredy, Novosibirsk, Lavrentyev Institute of Hydrodynamics, (24), 15-35.


\bibitem{AGG} Amrouche, C., Girault, V., Giroire, J. (1997). Dirichlet and neumann exterior problems for the n-dimensional laplace operator an approach in weighted sobolev spaces. Journal de mathématiques pures et appliqu\'ees, 76(1), 55-81.

\bibitem{AL} Avellaneda, M., Lin, F. H. (1987). Compactness methods in the theory of homogenization. Communications on Pure and Applied Mathematics, 40(6), 803-847.


\bibitem{Bardos} Bardos, C. (1970). Problèmes aux limites pour les \'equations aux d\'eriv\'ees partielles du premier ordre \'a coefficients r\'eels; th\'eor\`emes d'approximation; application \'a l'\'equation de transport. In Annales scientifiques de l'\'Ecole Normale Sup\'erieure (Vol. 3, No. 2, pp. 185-233).

\bibitem{Boyer} Boyer, F. (2005). Trace theorems and spatial continuity properties for the solutions of the transport equation. Differential and integral equations, 18(8), 891-934.

\bibitem{BF} Boyer, F.,  Fabrie, P. (2012). Mathematical Tools for the Study of the Incompressible Navier-Stokes Equations and Related Models (Vol. 183). Springer Science \& Business Media.


\bibitem{IOT} Bravin, M. (2019). Dynamics of a viscous incompressible flow in presence of a rigid body and of an inviscid incompressible flow in presence of a source and a sink (Doctoral dissertation, Universit\'e de Bordeaux).


\bibitem{IO3} Bravin, M., Sueur, F. (2021). Existence of weak solutions to the two-dimensional incompressible Euler equations in the presence of sources and sinks. arXiv preprint arXiv:2103.13912.

\bibitem{Cri:Spi}  Crippa, G., Spirito, S. (2015). Renormalized Solutions of the 2D Euler Equations. Communications in Mathematical Physics, 339(1).

\bibitem{CNSS} Crippa, G., Nobili, C., Seis, C., Spirito, S. (2017). Eulerian and Lagrangian Solutions to the Continuity and Euler Equations with $L^1$ Vorticity. SIAM Journal on Mathematical Analysis, 49(5), 3973-3998.


\bibitem{Del} Delort, J. M. (1991). Existence de nappes de tourbillon en dimension deux. Journal of the American Mathematical Society, 4(3), 553-586.

\bibitem{DPL} DiPerna, R. J., Lions, P. L. (1989). Ordinary differential equations, transport theory and Sobolev spaces. Inventiones mathematicae, 98(3), 511-547.

\bibitem{DM} Dolzmann, G., M\"uller, S. (1995). Estimates for Green's matrices of elliptic systems by $L^p$ theory. Manuscripta mathematica, 88(1), 261-273.


\bibitem{FLF} Fernandes, F. Z., Lopes Filho, M. C. (2007). Two-dimensional incompressible ideal flows in a noncylindrical material domain. Mathematical Models and Methods in Applied Sciences, 17(12), 2035-2053.

\bibitem{dist:fun} Foote, R. L. (1984). Regularity of the distance function. Proceedings of the American Mathematical Society, 92(1), 153-155. 

\bibitem{HH} He, C., Hsiao, L. (2000). Two-dimensional Euler equations in a time dependent domain. Journal of Differential Equations, 163(2), 265-291.


\bibitem{ILN} Iftimie, D., Lopes Filho, M. C., Nussenzveig Lopes, H. J. (2020). Weak vorticity formulation of the incompressible 2D Euler equations in bounded domains. Communications in Partial Differential Equations, 45(2), 109-145.



\bibitem{Lic} Lichtenstein, L. (1921). Neuere Entwicklung der Potentialtheorie. Konforme Abbildung. In Encyklop\"adie der mathematischen Wissenschaften mit Einschluss ihrer Anwendungen (pp. 177-377). Vieweg+ Teubner Verlag, Wiesbaden.

\bibitem{lions}
Lions, P.-L.~(1996). 
Mathematical topics in fluid mechanics. Vol. 1. Incompressible models. Oxford Lecture Series in Mathematics and its Applications 3. 


\bibitem{NFS} Noisette, F., Sueur, F. (2021). Uniqueness of Yudovich's solutions to the 2D incompressible Euler equation despite the presence of sources and sinks. arXiv preprint arXiv:2106.11556.

\bibitem{NS} Novotny, A., Straskraba, I. (2004). Introduction to the mathematical theory of compressible flow (No. 27). Oxford University Press on Demand.

\bibitem{Scho} Schochet, S. (1995). The weak vorticity formulation of the 2-D Euler equations and concentration-cancellation. Communications in partial differential equations, 20(5-6), 1077-1104.

\bibitem{Jud} Judovi\v{c}, V. I. (1964). A two-dimensional non-stationary problem on the flow of an ideal incompressible fluid through a given region. Matematicheskii sbornik, 106(4), 562-588.

\end{thebibliography}
\end{document}